\documentclass[11pt]{amsart}
\usepackage{amsmath,amssymb,amsthm}
\usepackage[T1]{fontenc}
\usepackage[a4paper,margin=27mm]{geometry}
\usepackage[hidelinks]{hyperref}
\numberwithin{equation}{section}
\newtheorem{theorem}{Theorem}[section]
\newtheorem{lemma}[theorem]{Lemma}
\newtheorem{proposition}[theorem]{Proposition}
\newtheorem{corollary}[theorem]{Corollary}
\theoremstyle{definition}
\newtheorem{definition}[theorem]{Definition}
\theoremstyle{remark}
\newtheorem{remark}[theorem]{Remark}
\newcommand{\R}{\mathbb R}
\newcommand{\T}{{\mathbb R/\mathbb Z}}
\newcommand{\B}{\mathcal B}
\newcommand{\E}{\mathcal E}
\newcommand{\LL}{\mathcal L}
\newcommand{\CC}{\mathcal C}
\newcommand{\M}{\mathsf M}
\newcommand{\Id}{\operatorname{Id}}
\newcommand{\tr}{\operatorname{tr}}
\newcommand{\dist}{\operatorname{dist}}
\newcommand{\supp}{\operatorname{supp}}
\newcommand{\dd}{\,\mathrm d}
\newcommand{\TC}{\operatorname{TC}}
\newcommand{\Rop}{\operatorname{Rop}}
\title[Elastic knots: compactness, contact, and variational structure]{%
Elastic knots \\ compactness --- contact --- variational structure}
\author{Philipp Reiter}
\address[Philipp Reiter]{Chemnitz University of Technology,
Faculty of Mathematics, 09107~Chemnitz, Germany}
\email{reiter@math.tu-chemnitz.de}	

\author{Heiko von der Mosel}
\address[Heiko von der Mosel]{RWTH Aachen University,
Institute for Mathematics,
Templergraben~55,
52062~Aachen, Germany}
\email{heiko@instmath.rwth-aachen.de}
\date{October 19, 2026}
\begin{document}
\begin{abstract}
We study bending-minimizing sequences of closed curves in a fixed knot
class and the self-contact of their limits. A positive matrix-valued
measure records a possible loss of curvature energy. Ambient variations,
localized on arbitrary coincidence sets, force this measure to vanish
and give strong convergence in $H^2$. We then derive the elastic equation
on isolated strands and overlap bundles, describe the distribution that
remains at contact, and prove balance and rigidity statements under their
precise additional hypotheses. Equality in the bridge bound gives
constant curvature. Combining strong compactness with quantitative
competitors gives an alternative proof that elastic limits in classes
with equal bridge and braid indices are covered circles. An independent
recovery theorem identifies the strong $H^2$ closure with the regular
$H^2$ part of the $C^1$ closure and also implies strong compactness.
Finally, a whole-line contact model illustrates which estimates are
needed for finer shape selection.
\end{abstract}
\maketitle

\tableofcontents

\section*{Introduction}
A closed wire of fixed length resists bending, while its knot type
restricts how it can move. Minimizing bending among embeddings does not
produce a closed admissible set: a limit may touch itself. The central
problem is to pass the energy and its variational information to such a
limit. We prove strong compactness of bending-minimizing sequences,
then separate the free elastic equation from the reaction at contact.

The underlying bending functional appears in Euler's treatment of
elastic curves \cite[Additamentum I]{Euler1744}. Classical closed
elasticae and their stability were studied by Langer and Singer
\cite[Theorem 3.2]{LangerSinger}; self-repulsive and self-obstacle knot
models were developed by von der Mosel \cite{Mosel1998,Mosel1999}.
Gallotti and Pierre-Louis related thin-wire shapes to bridge and braid
indices \cite[Sections VI.C and VII]{GallottiPierreLouis}, while
Avvakumov and Sossinsky investigated computational normal forms
\cite[Abstract]{AvvakumovSossinsky}. Gerlach, Reiter and von der Mosel introduced the ropelength formulation
of elastic knots and identified the limits for $(2,b)$-torus knots
\cite[Theorems 1.1, 2.1--2.2]{GRM}. Symmetry constraints can produce
different limiting contacts \cite[Theorem 1.6]{Symmetric}.

Numerical work on self-avoiding curves \cite{BartelsReiterRiege,
BartelsReiterStability}, framed rods \cite{BartelsReiterRods}, and
multistable knots \cite{VidulisEtAl2023} concerns several equilibria and
their approximation. Here an elastic knot is a limit of global
regularized minimizers. For classes with equal bridge and braid indices \cite{DER}, the
covered-circle conclusion is already obtained by Reiter and von der
Mosel \cite[Sections 4--7]{RvM2026}, conditional on a closure curvature
bound, which Miura subsequently proved \cite[Theorem 1.3]{Miura2026}.
Our alternative proof uses the quantitative competitors from the former
work and replaces that closure inequality by strong compactness.

Theorem~\ref{thm:no-loss} gives strong compactness. Theorems
\ref{ct:thm:freearcs} and \ref{ct:thm:nowheredense} describe free arcs
and the nowhere-dense reaction support. Bridge saturation and elastic
BB limits are treated in Proposition~\ref{sa:prop:saturation} and
Corollary~\ref{sa:cor:bb}; Theorems~\ref{sa:thm:measure},
\ref{sa:thm:crookedness} and \ref{sa:thm:planar} give rigidity under
additional hypotheses. The contact model is separate from these results.
The independent recovery theorem, Theorem~\ref{rc:theorem}, also implies
Theorem~\ref{thm:no-loss}: recovery supplies the missing lower energy
bound at a weak limit. We keep the direct defect-measure proof because
it avoids the appendix's configuration-space construction and exposes
how ambient variations eliminate curvature loss.

\section{The setting and the analytic tools}\label{sec:framework}

\subsection{Curves, knot classes, and the compactness claim}
We use the parameter circle $\T=\R/\mathbb Z$, with Lebesgue measure
of total mass one. A function in $H^2(\T;\R^3)=W^{2,2}(\T;\R^3)$ has
its first two weak derivatives in $L^2$. A weak derivative is defined
component by component by
$\int f\varphi'=-\int f'\varphi$ for smooth periodic test functions
$\varphi$. We always choose the $C^1$ representative of an $H^2$
function; its existence is explained below. A curve is regular when
its derivative never vanishes, and is an embedding when it is also
injective on the circle.

The knot type of an embedded curve refers to its image, without an
orientation. Images have the same type when an ambient isotopy carries
one to the other. Here an ambient isotopy is a family of homeomorphisms
$\Phi_t$ of $\R^3$, $0\le t\le1$, with $\Phi_0=\Id$, such that
$(t,x)\mapsto\Phi_t(x)$ and $(t,x)\mapsto\Phi_t^{-1}(x)$ are jointly
continuous. A tame knot class is one containing a smooth regular
embedded representative. We fix a tame class $K$
with one component. Set
\[
 \begin{gathered}
 \CC(K)=\{v\in H^2(\T;\R^3):v(0)=0,\ |v'|=1,\
                  v\text{ is an embedding of type }K\},\\
 e_K=\inf_{v\in\CC(K)}\B(v).
 \end{gathered}
\]
The condition $v(0)=0$ removes translations. Unit speed on this circle
means length one, and for these curves
$\B(v)=\int_\T|v''|^2\,ds$. Section~\ref{sec:variation} defines the
same geometric energy for arbitrary regular parameters. Normalizing a
smooth representative shows that $e_K$ is finite. For the knot and
normalization conventions see also \cite[Section 2]{GRM} and
\cite[Section 2.3]{RvM2026}.

\begin{theorem}[No loss of bending energy]\label{thm:no-loss}
Suppose $v_j\in\CC(K)$ and $\B(v_j)\to e_K$. This sequence is relatively
compact in $H^2$. Moreover, if a subsequence converges weakly in $H^2$
to $u$, that entire subsequence converges strongly to $u$. Each limit
satisfies $u(0)=0$, $|u'|=1$, and $\B(u)=e_K$.
\end{theorem}
Different subsequences may have different limits. A limit may also have
self-intersections. Thus its energy cannot be bounded below by $e_K$
merely by calling it a competitor: membership in $\CC(K)$ can fail.
The direct proof in Sections~\ref{sec:defect}--\ref{sec:compactness}
represents possible loss by a positive tensor measure and eliminates
it with ambient deformations. The independent recovery theorem in
Appendix~\ref{rc:section} gives a second proof by making the limiting
energy a limit of embedded competitor energies.

\subsection{Sobolev estimates in one dimension}
On a bounded interval $I$, an $H^1$ function has an absolutely continuous
representative satisfying
\begin{equation}
 w(y)-w(x)=\int_x^y w'(t)\,dt,\qquad
 |w(y)-w(x)|\le |y-x|^{1/2}\|w'\|_{L^2(I)}.
 \label{eq:nl-1}
\end{equation}
The integral changes sign when its endpoints are reversed. This is
the fundamental theorem for Sobolev functions in one dimension
\cite[Theorem 8.2]{Brezis}. To bound $\|w\|_\infty$, start at a point
where $|w|\le |I|^{-1/2}\|w\|_2$ and use \eqref{eq:nl-1}. Apply this
argument to both $f$ and $f'$ to obtain
\begin{equation}
 \|f\|_{C^1(\overline I)}\le C_I\|f\|_{H^2(I)}.
 \label{eq:nl-2}
\end{equation}
In particular, the representative of $f$ is $C^1$. Bounded sets in
$H^2(I)$ have uniformly bounded, equicontinuous functions and first
derivatives. Arzel\`a--Ascoli, applied to the pair $(f,f')$, gives a
subsequence converging uniformly in both entries. Passing to
$f(y)-f(x)=\int_x^y f'$ identifies the second entry as the derivative
of the first. This proves compactness of $H^2(I)\hookrightarrow C^1$;
the periodic version follows by the same argument
\cite[Theorems 4.25 and 8.8]{Brezis}.

We also need smooth approximation without extra boundary conditions.
At a left endpoint translated to zero, extend $f$ to negative $x$ by
$3f(-x)-2f(-2x)$. The extension agrees in value and first derivative
at zero, and its second derivative is in $L^2$. There is no delta
term at the joining point, as integration by parts on the two sides
shows using those matching values. The extension operator is bounded
in $H^2$ by change of variables. Extend similarly at the right endpoint,
apply a cutoff outside $I$, and convolve with smooth mollifiers.
Convolution of each weak derivative converges in $L^2$, so the resulting
smooth functions converge in $H^2$, and hence in $C^1$
\cite[Theorem 4.22]{Brezis}. On $\T$ use periodic convolution.

The space $H^2(I)$ is closed under multiplication, with a bounded
bilinear product map: in
$(fg)''=f''g+2f'g'+fg''$, the factors $f,g,f',g'$ are bounded by
\eqref{eq:nl-2}, and the remaining factors are in $L^2$.
Smooth approximation proves this product rule for weak derivatives.
The argument applies componentwise to vector and matrix products.

Here is the coordinate-change result we shall use. Let $x:J\to I$ be
a monotone $C^1$ bijection between bounded intervals, with $x\in H^2(J)$
and $0<c\le |x'|\le C$. Composition by this fixed map is bounded from
$H^2(I)$ to $H^2(J)$, and
\begin{equation}
 (h\circ x)'=(h'\circ x)x',\qquad
 (h\circ x)''=(h''\circ x)(x')^2+(h'\circ x)x''.
 \label{eq:nl-3}
\end{equation}
Indeed the two second-derivative terms have $L^2(J)$ norms at most
$C^2c^{-1/2}\|h''\|_{L^2(I)}$ and
$\|h'\|_\infty\|x''\|_{L^2(J)}$, respectively. Change of variables
bounds the zeroth and first derivatives too. Establish the formula for
smooth $h$ and pass by smooth approximation with these estimates.
The inverse $\sigma=x^{-1}$ is $C^1$ and $H^2$, with
\begin{equation}
 \sigma'=\frac{1}{x'\circ\sigma},\qquad
 \sigma''=-\frac{x''\circ\sigma}{(x'\circ\sigma)^3}.
 \label{eq:nl-4}
\end{equation}
For justification, $x'$ is absolutely continuous and bounded away
from zero. Its reciprocal is absolutely continuous, with derivative
$-x''/(x')^2$. Composing that function with the Lipschitz monotone
inverse preserves absolute continuity; change of variables proves
this by the integral characterization of absolute continuity.
The first formula then has the displayed $L^2$ derivative.
All measure substitutions use $|x'|$, so decreasing coordinates are
included.

\subsection{Weak limits and finite measures}
The Hilbert spaces $H^2$ and $L^2$ are reflexive, and every bounded
sequence has a weakly convergent subsequence
\cite[Theorem 3.18]{Brezis}. In a Hilbert space, weak convergence
$z_j\rightharpoonup z$ implies $\|z\|\le\liminf\|z_j\|$: when
$z\ne0$, test against $z/\|z\|$. If the norms also converge, then
\begin{equation}
 \|z_j-z\|^2=\|z_j\|^2+\|z\|^2-2\langle z_j,z\rangle\longrightarrow0.
 \label{eq:nl-5}
\end{equation}
This identity will turn equality of bending energies into strong
convergence.

A finite Radon measure on a compact metric space is a finite Borel
measure with inner and outer regularity. We use signed measures too;
$|\nu|$ denotes the total variation of a signed measure $\nu$.
The Riesz representation theorem represents bounded linear functionals
on continuous functions by finite signed Radon measures, and positive
functionals by positive measures \cite[Theorems 11.1 and 11.26]{Knapp}.
A sequence with uniformly bounded total variation on $\T$ has a
weak-star convergent subsequence: choose a countable uniformly dense
subset of $C(\T)$, diagonalize its integrals, and use the variation
bound to extend convergence to every continuous function. Riesz
representation supplies the limiting measure. ``Weak-star'' below
means convergence of these continuous test integrals.

We use the Radon--Nikodym theorem for a finite signed measure $\nu$
absolutely continuous with respect to a finite positive measure $\lambda$:
there is $a\in L^1(\lambda)$ with $d\nu=a\,d\lambda$
\cite[Theorem 9.36]{Axler}. We also use dominated convergence for finite
measures and Tonelli's theorem for nonnegative integrands
\cite[Theorems 3.31 and 5.28]{Axler}. In each application below we give
the dominating function. Finally, if $q\in L^1$ and $|E_j|\to0$, then
$\int_{E_j}|q|\to0$: truncate $|q|$ at a fixed large height, control
the bounded part by that height times $|E_j|$, and then control the
tail by integrability. This is absolute continuity of the integral.

Throughout, $d\xi$, $ds$, and $dx$ denote Lebesgue measure in the named
parameter; $d\ell$ denotes geometric arclength; and $d\mu$ will denote
a finite positive measure which may be singular with respect to $ds$.

\section{Geometric energy and variations in a fixed parameter}\label{sec:variation}

\subsection{Speed, normal projection, and normalization}
For this section $v:\T\to\R^3$ is any regular $H^2$ curve, and its
parameter is $\xi$. Define
\[
 r=|v_\xi|,\qquad \tau=\frac{v_\xi}{r},\qquad
 P=\Id-\tau\otimes\tau,\qquad n=Pv_{\xi\xi}.
\]
For vectors $a,b\in\R^3$, the tensor $a\otimes b$ is the matrix acting
by $(a\otimes b)z=a(b\cdot z)$; its $(\alpha,\beta)$ entry is
$a_\alpha b_\beta$. Thus $P$ is orthogonal projection onto the plane
perpendicular to $\tau$. For any nonzero $z$ we also use
$P_z^\perp=\Id-z\otimes z/|z|^2$, so $P=P_\tau^\perp=P_{v_\xi}^\perp$.
Regularity and compactness give $\min r>0$. The chain rule gives
$r_\xi=\tau\cdot v_{\xi\xi}$ and $\tau_\xi=n/r$ almost everywhere.
Consequently, with $d\ell=r\,d\xi$ and
$D_\ell=r^{-1}\partial_\xi$, the curvature vector is
\begin{equation}
 \kappa=D_\ell\tau=\frac{n}{r^2},\qquad
 v_{\xi\xi}=r_\xi\tau+r^2\kappa.
 \label{eq:nl-6}
\end{equation}
The first term records speed change; the second records bending.

For completeness, the chain rule here does not require classical
second derivatives. If $w\in H^1$ and its range lies in a compact
subset of the domain of a $C^1$ map $F$, approximate $w$ smoothly in
$H^1$ and uniformly. On a fixed compact neighborhood, $DF$ is bounded
and uniformly continuous. Hence
\[
 \|DF(w_j)w_j'-DF(w)w'\|_2
 \le \|DF(w_j)\|_\infty\|w_j'-w'\|_2
      +\|DF(w_j)-DF(w)\|_\infty\|w'\|_2\to0.
\]
Passing in the weak derivative identity proves $(F(w))'=DF(w)w'$.
Apply this with $w=v_\xi$ and the norm and normalization maps, whose
derivatives are bounded on an annulus separated from zero; compare
\cite[Corollary 8.11]{Brezis}.

The geometric functionals used in the sequel are
\begin{equation}
 \LL(v)=\int_\T r\,d\xi,\qquad
 \B(v)=\int_\T|\kappa|^2\,d\ell
       =\int_\T\frac{|Pv_{\xi\xi}|^2}{r^3}\,d\xi,
 \qquad \E(v)=\LL(v)\B(v).
 \label{eq:nl-7}
\end{equation}
Total curvature is the separate quantity
$\TC(v)=\int_\T|\kappa|\,d\ell
=\int_\T|Pv_{\xi\xi}|/r\,d\xi$.
Here a regular change of parameter is a monotone $C^1$ circle
diffeomorphism $\xi=\phi(\eta)$ whose lift is $H^2$ and whose
derivative is bounded away from zero. Under this change the second
derivative is $v_{\xi\xi}(\phi')^2+v_\xi\phi''$.
Projection kills the second term. The remaining bending density gains
$|\phi'|^4/|\phi'|^3=|\phi'|$, the factor for change of variables.
Thus $\LL$, $\B$, and $\TC$ are independent of the regular parameter.
A dilation by $\lambda>0$ multiplies $\LL$ by $\lambda$, $\kappa$ by
$\lambda^{-1}$, and $\B$ by $\lambda^{-1}$. In particular $\E$ is
invariant under dilation, translation, and reparametrization.

Given a regular embedded $w$ of type $K$, set
$L=\LL(w)$ and $z=(w-w(0))/L$. Its length is one. The increasing
arclength map $a_w(\xi)=\int_0^\xi|z_\eta|\,d\eta$ has $a_w(1)=1$,
a positive continuous derivative, and an $H^2$ lift. By
\eqref{eq:nl-3}--\eqref{eq:nl-4} its inverse is also $H^2$ and
$\widehat w=z\circ a_w^{-1}$ is $H^2$. It is anchored and unit
speed. Translation and positive dilation are ambient isotopies, and
reparametrization preserves the image. Therefore
\begin{equation}
 \widehat w\in\CC(K),\qquad \B(\widehat w)=\E(w)\ge e_K.
 \label{eq:nl-8}
\end{equation}
This procedure changes scale and then parameter. We will use its
energy identity without differentiating either normalization operation.

\subsection{Variations in an arbitrary parametrization}
Fix $h\in H^2(\T;\R^3)$ and let $v_t=v+th$, holding $\xi$ fixed.
The variation parameter is $t$; it is distinct from $\xi$.
Put $c=\min|v_\xi|$ and
$t_0=c/[2(1+\|h_\xi\|_\infty)]$. For $|t|\le t_0$, the speed is at
least $c/2$, so the curves remain regular. No embeddedness is needed
for this differentiation. Write $\eta=h_\xi$, $\zeta=h_{\xi\xi}$,
and use a dot for $d/dt$ at zero. Differentiating the defining
finite-dimensional maps gives
\[
 \dot r=\tau\cdot\eta,\quad \dot\tau=P\eta/r,\quad
 \dot P=-\dot\tau\otimes\tau-\tau\otimes\dot\tau,
 \quad \dot n=P\zeta-\dot\tau(\tau\cdot v_{\xi\xi})
                    -\tau(\dot\tau\cdot v_{\xi\xi}).
\]
Since $n\cdot\tau=0$ and $\tau\cdot v_{\xi\xi}=r_\xi$,
$n\cdot\dot n=n\cdot\zeta-(r_\xi/r)n\cdot\eta$.
Differentiating $|n|^2r^{-3}$ now gives
\begin{equation}
 \begin{split}
 D\B(v)[h]&=\int_\T\left(
 \frac{2n\cdot h_{\xi\xi}}{r^3}
 -\frac{2r_\xi n\cdot h_\xi}{r^4}
 -\frac{3|n|^2\tau\cdot h_\xi}{r^4}\right)d\xi,\\
 D\LL(v)[h]&=\int_\T\tau\cdot h_\xi\,d\xi,
 \end{split}\label{eq:nl-9}
\end{equation}
and the product rule yields
\begin{equation}
 D\E(v)[h]=\LL(v)D\B(v)[h]+\B(v)D\LL(v)[h].
 \label{eq:nl-10}
\end{equation}
In particular the projection varies, and its variation produces the
middle term of $D\B$. Dropping it before imposing unit speed would
change the answer.

Here is a justification of the differentiation under the integral.
For $|t|\le t_0$, the vectors $v_\xi+th_\xi$ stay in a compact annulus
away from zero. All coefficient derivatives of the integrand are
bounded there. Since $h_\xi$ is bounded, the absolute value of its
$t$ derivative is at most
$C(1+|v_{\xi\xi}|^2+|h_{\xi\xi}|^2)$, with $C$ depending on this
fixed variation. This is an $L^1(d\xi)$ bound independent of $t$.
Apply the fundamental theorem in $t$ to difference quotients and
then dominated convergence. The length calculation has the constant
majorant $\|h_\xi\|_\infty$.

The same estimates give Fr\'echet differentiability, with locally
bounded second derivative. We record the precise estimate because
later ambient trace approximations use it.
\begin{lemma}[Uniform quadratic remainder]\label{lem:taylor}
For each regular $v\in H^2$ there are an open $H^2$ ball $U$ centered
at $v$ and $C_v<\infty$ such that
\[
 |\E(w+h)-\E(w)-D\E(w)[h]|\le C_v\|h\|_{H^2}^2
\]
whenever the line segment from $w$ to $w+h$ lies in $U$.
\end{lemma}
\begin{proof}
Let $C_S$ be the constant in the periodic version of
\eqref{eq:nl-2}, and choose the radius
$\rho=\min\{1,c/(4C_S)\}$, where $c=\min|v'|$.
For $w\in U$ we have $|w'|\ge3c/4$,
$\|w'\|_\infty\le\|v'\|_\infty+C_S$, and
$\|w''\|_2\le\|v''\|_2+1$.
Write
\[
 G(z)=|z|^{-3}\left(\Id-\frac{z\otimes z}{|z|^2}\right),
 \qquad \B(w)=\int_\T w''\cdot G(w')w''\,d\xi.
\]
The norms of $G$, $DG$, and $D^2G$ have finite upper bounds $C_0,C_1,C_2$
on the fixed annulus containing all these first derivatives.
Twice differentiating along a line gives
\[
 D^2\B(w)[h,h]=\int_\T\bigl(
 2h''\cdot G(w')h''+4h''\cdot DG(w')[h']w''
 +w''\cdot D^2G(w')[h',h']w''\bigr)d\xi.
\]
Each term is integrable by Cauchy--Schwarz and bounded coefficients,
and their absolute sum is bounded by
\[
 2C_0\|h''\|_2^2+4C_1\|h''\|_2\|h'\|_\infty\|w''\|_2
                  +C_2\|h'\|_\infty^2\|w''\|_2^2
 \le C\|h\|_{H^2}^2.
\]
The differentiation is justified by the same $L^2$ product estimates;
continuity of the coefficients also makes these derivatives continuous
in $w$. For length,
$D^2\LL(w)[h,h]=\int|P_{w'}^\perp h'|^2/|w'|\,d\xi$.
The speed lower bound, \eqref{eq:nl-9}, and the uniform $H^2$ bounds
therefore bound $D^2(\LL\B)(w)[h,h]$ by $C\|h\|_{H^2}^2$ throughout
$U$. Apply twice the fundamental theorem of calculus to
$g(t)=\E(w+th)$:
$g(1)-g(0)-g'(0)=\int_0^1(1-t)g''(t)\,dt$.
This proves the claim. Repeated differentiation introduces only bounded
derivatives of $G$ multiplying bounded first derivatives of variations
and at most two $L^2$ second-derivative factors; the same estimates
therefore give derivatives of every order. This establishes the smooth
energy calculus used in \cite[Lemma 2.1]{RvM2026}.
\end{proof}

Now specialize the base curve to $u$ of length one and unit speed,
and write $s$ for its parameter, $T=u'$, and $k=u''$. Since
$r=1$ and $r_s=0$, \eqref{eq:nl-9}--\eqref{eq:nl-10} become
\begin{equation}
 D\E(u)[h]=2\int_\T k\cdot h''\,ds
 -3\int_\T|k|^2T\cdot h'\,ds
 +\B(u)\int_\T T\cdot h'\,ds.
 \label{eq:nl-11}
\end{equation}
This says nothing about the speed of $u+th$. Its change of speed is
already accounted for by the last two terms. In particular no condition
$T\cdot h'=0$ is imposed. The functional is continuous on $H^2$:
\begin{equation}
 |D\E(u)[h]|\le2\|k\|_2\|h''\|_2
             +4\B(u)\|h'\|_\infty\le C_u\|h\|_{H^2}.
 \label{eq:nl-12}
\end{equation}

\subsection{Curvature under ambient maps}
Let $\Phi$ be a smooth diffeomorphism of $\R^3$. At the curve $v$ put
\[
 b=D\Phi(v)\tau,\quad a=|b|,\quad
 P_b=\Id-\frac{b\otimes b}{|b|^2},\quad
 Q=P_bD\Phi(v),\quad c=P_bD^2\Phi(v)[\tau,\tau].
\]
Here $D\Phi$ is a matrix, $D^2\Phi[\tau,\tau]$ is the Hessian applied
bilinearly to the two vectors, $Q$ is a matrix, and $c$ is a vector.
The transformed speed is $ra$. Differentiating twice and using
\eqref{eq:nl-6} gives
\[
 (\Phi\circ v)_{\xi\xi}
 =r_\xi D\Phi(v)\tau
       +r^2\bigl(D\Phi(v)\kappa+D^2\Phi(v)[\tau,\tau]\bigr).
\]
Projection onto $b^\perp$ removes the first term. The curvature and
energy identities are therefore
\begin{equation}
 \widetilde\kappa=\frac{Q\kappa+c}{a^2},\quad
 \LL(\Phi\circ v)=\int_\T a\,d\ell,\quad
 \B(\Phi\circ v)=\int_\T a^{-3}|Q\kappa+c|^2\,d\ell.
 \label{eq:nl-13}
\end{equation}
They apply to arbitrary regular parameters. When the base is unit
speed, its ambient image still has speed $a$, which these formulas
retain.

\section{Measuring possible curvature loss}\label{sec:defect}

\subsection{A weak limit and a positive tensor measure}
Let $v_j$ be a sequence from Theorem~\ref{thm:no-loss}. Unit speed and
anchoring imply $|v_j(s)|\le1$, so
$\|v_j\|_{H^2}^2\le C+\B(v_j)$ is bounded. Take any weakly convergent
$H^2$ subsequence and call its limit $u$. On this entire subsequence,
compactness implies $v_j\to u$ in $C^1$. To see why a further extraction
is unnecessary, any subsequence has a further $C^1$ convergent
subsequence, and its limit must agree with its weak $H^2$ limit $u$.
A subsequence remaining a positive $C^1$ distance from $u$ would
contradict that fact. Put
\[
 T_j=v_j',\quad k_j=v_j'',\qquad T=u',\quad k=u''.
\]
Then $T_j\to T$ uniformly, $k_j\rightharpoonup k$ in $L^2$,
$u(0)=0$, and $|T_j|=|T|=1$. The weak product rule gives
$T\cdot k=\tfrac12(|T|^2)'=0$ almost everywhere; likewise
$T_j\cdot k_j=0$. In this unit-speed setting $ds$ is both parameter
measure and arclength measure for the base curves.

For a Borel set $E\subset\T$, define
$W_j(E)=\int_E k_j\otimes k_j\,ds$. This is a symmetric $3$ by $3$
matrix of finite signed measures. Each entry has total variation at
most $\|k_j\|_2^2$. Extract a weak-star limit simultaneously for its
nine entries along a further subsequence, and denote the limit by $W$.
All limits involving $W$ below use this extracted subsequence.
For continuous matrix fields $H$,
\begin{equation}
 \int_\T H:dW_j\longrightarrow\int_\T H:dW.
 \label{eq:nl-14}
\end{equation}
The contraction $H:C=\sum_{\alpha,\beta}H_{\alpha\beta}C_{\alpha\beta}$
is the Frobenius contraction; integration means summing the nine
entrywise integrals. In particular
$H:(z\otimes z)=\sum_{\alpha,\beta}H_{\alpha\beta}z_\alpha z_\beta
=z\cdot Hz$. Define the defect by
\begin{equation}
 \M=W-(k\otimes k)\,ds.
 \label{eq:nl-15}
\end{equation}
The ordinary part accounts for curvature already present in $u$.
We must show that the remaining matrix of measures vanishes.

For a constant vector $q\in\R^3$ and a nonnegative continuous function
$\varphi$, weak $L^2$ convergence gives
$\sqrt\varphi\,q\cdot k_j\rightharpoonup\sqrt\varphi\,q\cdot k$.
Indeed pairing against an $L^2$ test is equivalent to testing $k_j$
against the fixed $L^2$ vector $\sqrt\varphi\,q$ times that test.
Weak lower semicontinuity and \eqref{eq:nl-14} yield
\[
 \int\varphi(q\otimes q):d\M
 =\lim_j\int\varphi|q\cdot k_j|^2\,ds
                     -\int\varphi|q\cdot k|^2\,ds\ge0.
\]
Riesz representation implies that $q^T\M q$ is a positive measure.
Consequently $\M(E)$ is positive semidefinite for every Borel set $E$:
its quadratic form is nonnegative for every $q$.
Define its trace measure and mass by
\begin{equation}
 \mu=\tr\M,\qquad d=\mu(\T)=e_K-\B(u)\ge0.
 \label{eq:nl-16}
\end{equation}
The equality follows by using $H=\Id$ in \eqref{eq:nl-14} and the
minimizing property $\|k_j\|_2^2\to e_K$.

The positive trace controls every entry. Applied to
$e_\alpha+e_\beta$ and $e_\alpha-e_\beta$, positivity says that
$\tfrac12(\M_{\alpha\alpha}+\M_{\beta\beta})\pm\M_{\alpha\beta}$
are positive measures. Thus
$|\M_{\alpha\beta}|\le\tfrac12(\M_{\alpha\alpha}+\M_{\beta\beta})
\le\mu$ as measures: on each finite Borel partition, bound the
absolute value of each entry by the dominating positive measure and
sum; then take the supremum defining total variation.
Radon--Nikodym, applied entrywise, gives
\begin{equation}
 d\M=A\,d\mu,\qquad A=A^T,\quad A\ge0,\quad\tr A=1
                           \quad\mu\text{-almost everywhere}.
 \label{eq:nl-17}
\end{equation}
For the positivity assertion, each rational vector $q$ has
$q^TAq\ge0$ almost everywhere, since this is the density of a positive
measure. Take the intersection of the countably many resulting sets
of full measure and then approximate arbitrary $q$ by rational vectors.
Symmetry and trace follow from the corresponding measure identities.
The eigenvalues of $A$ are nonnegative with sum one, so its operator
norm and Frobenius norm are at most one. These bounds justify all
later contractions against $\mu$. If $\mu=0$, the entrywise domination
already proves $\M=0$.

\subsection{The defect has no tangential component}
Because $T_j\cdot k_j=0$,
\[
 \int|T\cdot k_j|^2\,ds
 \le\|T-T_j\|_\infty^2\|k_j\|_2^2\longrightarrow0.
\]
Use the continuous test matrix $T\otimes T$ in \eqref{eq:nl-14}.
The part $(k\otimes k)\,ds$ contributes zero since $T\cdot k=0$.
We obtain $\int T^TAT\,d\mu=0$. Its integrand is nonnegative, hence
zero almost everywhere. For a positive semidefinite matrix the
inequality $|z^TAw|^2\le(z^TAz)(w^TAw)$ follows by considering its
quadratic form on $z+tw$ for real $t$. Therefore $T^TAT=0$ forces
$AT=0$, and
\begin{equation}
 AT=0\quad\mu\text{-almost everywhere},\qquad \M T=0.
 \label{eq:nl-18}
\end{equation}
The second equality is an equality of vector measures: its $\alpha$
component is $\sum_\beta T_\beta\,d\M_{\alpha\beta}$.

\begin{remark}\label{rem:rank}
Normality gives rank at most two, rather than rank one. Positivity and
normality are precisely the properties needed below; no rank-one
representation or density relative to $ds$ is assumed.
\end{remark}

\subsection{Fixed ambient deformations}
Fix $X\in C_c^\infty(\R^3;\R^3)$ and set
$\Phi_t(x)=x+tX(x)$. For
$|t|<t_X=[2(1+\|DX\|_\infty)]^{-1}$,
\[
 |\Phi_t(x)-\Phi_t(y)|\ge(1-|t|\|DX\|_\infty)|x-y|.
\]
Thus $\Phi_t$ is injective. For any prescribed image $y$, the iteration
$x_{n+1}=y-tX(x_n)$ has successive differences decreasing by a fixed
factor below one, so it converges to the unique solution of
$\Phi_t(x)=y$. This proves surjectivity and a Lipschitz inverse.
The derivative $\Id+tDX$ is invertible, and the inverse function
theorem then gives a smooth inverse. The inverses depend continuously
on $t$ as well: their Lipschitz constant is at most two, and
$|\Phi_t^{-1}(y)-\Phi_s^{-1}(z)|\le2(|y-z|+|t-s|\|X\|_\infty)$.
All these facts also hold for $\Phi_{\lambda t}$, $0\le\lambda\le1$,
giving an ambient isotopy.
The images $\Phi_t\circ v_j$ therefore remain regular $H^2$ embeddings
of type $K$. Normalize them with \eqref{eq:nl-8}; this yields
\begin{equation}
 \E(\Phi_t\circ v_j)\ge e_K.
 \label{eq:nl-19}
\end{equation}

For the base $u$, define $b_t,a_t,Q_t,c_t$ by
\eqref{eq:nl-13}, and write $b_{t,j},a_{t,j},Q_{t,j},c_{t,j}$ for the
same quantities at $v_j,T_j$. At each fixed $|t|<t_X$ they converge
uniformly, and all the speeds $a_{t,j},a_t$ are at least $1/2$.
This uses uniform convergence of $v_j,T_j$ and uniform continuity of
$X$ and its first two derivatives on one compact neighborhood of
the images. In particular
$L_{t,j}=\int a_{t,j}\,ds\to L_t=\int a_t\,ds$.
Expanding the bending part of \eqref{eq:nl-13} gives a quadratic,
a linear, and a constant term in $k_j$:
\[
 \int a_{t,j}^{-3}\bigl(
 k_j\cdot Q_{t,j}^TQ_{t,j}k_j
   +2k_j\cdot Q_{t,j}^Tc_{t,j}+|c_{t,j}|^2\bigr)\,ds.
\]
Put $H_{t,j}=a_{t,j}^{-3}Q_{t,j}^TQ_{t,j}$ and
$H_t=a_t^{-3}Q_t^TQ_t$. Replacing $H_{t,j}$ by $H_t$ in the
quadratic term has error at most
$\|H_{t,j}-H_t\|_\infty\|k_j\|_2^2\to0$; then
\eqref{eq:nl-14} applies. The linear coefficient converges uniformly,
hence in $L^2$, so its pairing with weakly convergent $k_j$ converges.
The constant integrand converges uniformly. Multiplying by the
convergent length factor proves
\begin{equation}
 J(t):=\lim_j\E(\Phi_t\circ v_j)
 =L_t\left(\int_\T a_t^{-3}|Q_tk+c_t|^2\,ds
                   +\int_\T H_t:d\M\right).
 \label{eq:nl-20}
\end{equation}
This limit is taken at fixed $t$ along the further subsequence defining
$W$. We have not differentiated a weak limit.
Its first term is exactly $\E(\Phi_t\circ u)$.
By \eqref{eq:nl-19}, $J(t)\ge e_K$, whereas
$J(0)=\B(u)+d=e_K$.

\subsection{First variation of the limiting energy}
Write $G=DX(u)$, $\alpha=T\cdot GT$, and
$q=(\Id-T\otimes T)GT$. At zero,
\[
 a_0=1,\quad \dot a_0=\alpha,\quad Q_0=P_T:=\Id-T\otimes T,
 \quad \dot P_{b_t}|_0=-q\otimes T-T\otimes q,
 \quad \dot Q_0=\dot P_{b_t}|_0+P_TG.
\]
For $z\perp T$, $z\cdot\dot P_{b_t}z=0$, and hence
\[
 \left.\partial_t\bigl(a_t^{-3}|Q_tz|^2\bigr)\right|_0
                  =2z\cdot Gz-3\alpha|z|^2.
\]
Since $A$ is symmetric, positive, and annihilates $T$, diagonalize
$A$ on $T^\perp$. Applying this identity to an orthonormal eigenbasis
and multiplying by its nonnegative eigenvalues gives
$\dot H_0:A=2G:A-3\alpha\tr A=2G:A-3\alpha$.
This explains the contraction without assuming rank one.
Also $H_0:A=P_T:A=1$, and
$\dot L_0=\int\alpha\,ds$.

For $|t|\le t_X/2$, the coefficients of \eqref{eq:nl-20} and their
$t$ derivatives are uniformly bounded, with $a_t\ge1/2$.
The derivatives of the ordinary integrand are bounded by
$C(1+|k|^2)\in L^1(ds)$, and the defect integrands and derivatives by
$C\in L^1(\mu)$, using $|A|\le1$. The fundamental theorem in $t$
and dominated convergence therefore justify differentiation under both
integrals. Since $J$ has a two-sided minimum at zero, $J'(0)=0$.
Using \eqref{eq:nl-11} for the ordinary term, we obtain
\begin{equation}
 \begin{split}
 0={}&D\E(u)[X\circ u]+d\int_\T T\cdot DX(u)T\,ds\\
    &+2\int_\T DX(u):d\M
              -3\int_\T T\cdot DX(u)T\,d\mu.
 \end{split}\label{eq:nl-21}
\end{equation}
The term involving $d\,ds$ comes from varying the length factor in
$\E=\LL\B$ and remains necessary until the defect vanishes.
Only ambient fields have been tested; their values coincide at all
parameters with the same spatial image.

\subsection{Exclusion of atoms}
We next prove
\begin{equation}
 \mu(u^{-1}(\{p\}))=0\qquad(p\in\R^3).
 \label{eq:nl-22}
\end{equation}
There is nothing to prove for $p\notin u(\T)$. Each preimage of a point
in the image has a locally one-to-one neighborhood: projection onto
its tangent direction has derivative bounded away from zero there.
The fiber is closed and every one of its points is isolated. Its
compactness therefore makes it finite, say $u^{-1}(p)=\{s_1,\ldots,s_N\}$.

Choose $0<\delta<1/8$ so the arcs of radius $\delta$ about these
parameters are disjoint and
$|T(s)-T(s_i)|\le1/2$ on each arc. Uniform continuity and the finite
number of parameters permit these simultaneous choices. We use the
lift centered at $s_i$ on each arc. Projection onto $T(s_i)$ gives
$|u(s)-p|\ge |s-s_i|/2$ there. The compact set
$C_p=\T\setminus\bigcup_i(s_i-\delta/2,s_i+\delta/2)$
has image separated from $p$ by some distance $d_p>0$.
Fix $r_p=\min\{d_p/4,\delta/8\}$. For $0<r<r_p$,
\[
 u^{-1}(B_{2r}(p))\subset
 U_r:=\bigcup_i\{|s-s_i|\le4r\},\qquad |U_r|\le8Nr.
\]
Indeed the separation first puts the preimage inside the smaller
arcs, and then the projection estimate gives $|s-s_i|<4r$.
Absolute continuity of the integral gives
\begin{equation}
 E_r:=\int_{U_r}|k|^2\,ds\longrightarrow0.
 \label{eq:nl-23}
\end{equation}

Fix a smooth scalar cutoff $\chi$ equal to one on $[0,1]$ and zero
on $[2,\infty)$, with $0\le\chi\le1$, and put
$X_r(x)=\chi(|x-p|/r)(x-p)$. This is smooth, including at $p$, and
compactly supported; $DX_r$ is bounded independently of $r$ and equals
$\Id$ at $p$. For $h_r=X_r\circ u$ the weak chain rule gives
$h_r'=DX_r(u)T$ and
$h_r''=DX_r(u)k+D^2X_r(u)[T,T]$.
In \eqref{eq:nl-11}, the terms with $DX_r$ are bounded in absolute
value by $CE_r+C\B(u)|U_r|$. To handle the Hessian, write
$q=u(s)-p$, $\rho=|q|$, and $F(\rho)=\chi(\rho/r)$.
On $r\le\rho\le2r$ differentiation gives
\begin{equation}
 \begin{split}
 D^2X_r(u)[T,T]={}&2F'(\rho)\frac{q\cdot T}{\rho}T\\
 &+\left[F''(\rho)\frac{(q\cdot T)^2}{\rho^2}
       +F'(\rho)\left(\frac1\rho-
                          \frac{(q\cdot T)^2}{\rho^3}\right)\right]q.
 \end{split}\label{eq:nl-24}
\end{equation}
It is zero off this annulus. The bracket is bounded by $C/r^2$.
The first term is tangential and disappears when paired with $k$.

On the arc based at $s_i$, subtract $T(s)$ inside the integral for
$q$, then project normally to $T(s)$. For $s\ge s_i$ this gives
\begin{equation}
 |P_{T(s)}^\perp q|
 \le\int_{s_i}^s(a-s_i)|k(a)|\,da
 \le\frac{|s-s_i|^{3/2}}{\sqrt3}\|k\|_{L^2(U_r)}
 \le C r^{3/2}E_r^{1/2}.
 \label{eq:nl-25}
\end{equation}
The first bound follows by inserting
$T(t)-T(s)=-\int_t^s k(a)\,da$ and interchanging the triangular
integrals. Their absolute integral is at most
$|s-s_i|\int_{s_i}^s|k(a)|\,da<\infty$, so Fubini applies.
Reverse the interval to get the same estimate when $s<s_i$.
Equations \eqref{eq:nl-24}--\eqref{eq:nl-25} and Cauchy--Schwarz imply
\[
 \left|\int k\cdot D^2X_r(u)[T,T]\,ds\right|
 \le Cr^{-2}r^{3/2}E_r^{1/2}|U_r|^{1/2}E_r^{1/2}
 \le C_p E_r\to0.
\]
Thus
\begin{equation}
 D\E(u)[X_r\circ u]\to0,\qquad
 \int T\cdot DX_r(u)T\,ds\to0.
 \label{eq:nl-26}
\end{equation}
For each parameter off the fiber, $DX_r(u)$ is eventually zero; on
the fiber it equals $\Id$. Its uniform bound dominates all defect
integrands by a constant in $L^1(\mu)$. In \eqref{eq:nl-21}, the two
defect terms consequently tend to
$(2-3)\mu(u^{-1}(p))$. Together with \eqref{eq:nl-26} this proves
\eqref{eq:nl-22}. In particular the spatial pushforward
$u_\#\mu$, defined by $(u_\#\mu)(E)=\mu(u^{-1}(E))$, has no atoms.
Because every fiber is finite, spatial and parameter atomlessness are
equivalent. The spatial formulation is natural here: an ambient field
acts on all visits to a point together, and the exceptional crossings
below form a countable set in the image.

\section{Localizing on arbitrary sets of coincidence}\label{sec:localization}
The remaining difficulty is tangential coincidence. Two strands may
agree on a closed set with infinitely many complementary intervals.
We need a cutoff near such a set whose trace on the strands tends to
zero in $H^2$, even though its spatial derivative still expands their
normal directions. The next elementary lemma provides the necessary
cancellation of cutoff derivatives.

\subsection{A cutoff near a flat zero set}
\begin{lemma}\label{lem:flat-cutoff}
Let $I$ be a bounded open interval, let $F\Subset I$ be compact, and
let $g\in H^2(I;\R^m)$ satisfy $g=g'=0$ on $F$ in the continuous
representatives. There are smooth functions $0\le\chi_\varepsilon\le1$
which equal one on a neighborhood of $F$, are supported in
$F_{3\varepsilon}:=\{x:\dist(x,F)<3\varepsilon\}$, and satisfy
\[
 \|\chi_\varepsilon^{(q)}\|_\infty\le C_q\varepsilon^{-q}
 \ (q=1,2),\qquad
 \chi_\varepsilon g\to0\text{ in }H^2(I),\qquad
 \|\chi_\varepsilon'g\|_\infty\to0.
\]
They can be chosen to depend only on $F$ and $\varepsilon$, not on $g$;
the same cutoffs therefore work for finitely many such functions at once.
They also satisfy $\chi_\varepsilon\to\mathbf1_F$ pointwise.
\end{lemma}
\begin{proof}
For empty $F$ take the zero function. Otherwise let
$\delta_F=\dist(F,\R\setminus I)>0$ and
$0<\varepsilon<\delta_F/4$. Convolve the indicator of
$F_{2\varepsilon}$ on $\R$ with a nonnegative smooth mollifier of
mass one supported in $(-\varepsilon/2,\varepsilon/2)$.
The convolution $\chi_\varepsilon$ equals one on $F_{3\varepsilon/2}$
and has support in $\overline{F_{5\varepsilon/2}}\subset F_{3\varepsilon}$.
The $L^1$ norms of the first two derivatives of the mollifier give the
stated derivative bounds. A fixed point outside the closed set $F$
is eventually outside the support, proving pointwise convergence.

We first show $g''=0$ almost everywhere on $F$. The absolutely
continuous function $g'$ is differentiable almost everywhere, with
derivative $g''$ by the fundamental theorem for Lebesgue integrals
\cite[Theorem 4.19]{Axler}. At a non-isolated point $x\in F$ where
this derivative exists, choose other points $x_n\in F$ tending to
$x$. Since $g'(x_n)=g'(x)=0$, these difference quotients are zero,
so $g''(x)=0$. The isolated points of $F$ are countable: each can be
isolated by a rational-endpoint interval, and such an interval
isolates at most one point. Hence they form a null set. This argument
applies to each component. Dominated convergence, with majorant
$|g''|^2$, now gives
\begin{equation}
 A_\varepsilon:=\int_{F_{3\varepsilon}}|g''|^2\,dx\longrightarrow0.
 \label{eq:nl-27}
\end{equation}
This remains valid if $F$ has positive measure.

The open set $I\setminus F$ has countably many interval components,
since each contains a rational number. Split a component with two
endpoints in $F$ at its midpoint. Its part within distance
$3\varepsilon$ of $F$ consists of two disjoint strips, each measured
from its endpoint in $F$ and of length at most $3\varepsilon$.
For each of the two possible outer components use the one endpoint
in $F$. Ignore midpoint and endpoint sets, which have measure zero.
On a strip with endpoint coordinate $t=0$, integration gives
$g'(t)=\int_0^t g''(a)\,da$ and
$g(t)=\int_0^t(t-a)g''(a)\,da$. Cauchy--Schwarz gives
\begin{equation}
 |g'(t)|^2\le t\int_0^t|g''(a)|^2\,da,\qquad
 |g(t)|^2\le\frac{t^3}{3}\int_0^t|g''(a)|^2\,da.
 \label{eq:nl-28}
\end{equation}
On a strip of length $L$, integrate these inequalities and use Tonelli
for their nonnegative integrands. The inner integrals are bounded by
$L^2/2$ and $L^4/12$, respectively. Summing over the disjoint strips
and using $g=g'=0$ on $F$ proves
\begin{equation}
 \varepsilon^{-2}\int_{F_{3\varepsilon}}|g'|^2\,dx
 +\varepsilon^{-4}\int_{F_{3\varepsilon}}|g|^2\,dx
                       \le C A_\varepsilon\to0.
 \label{eq:nl-29}
\end{equation}
There is no factor counting components, since the strips are disjoint.
The product rule
$(\chi_\varepsilon g)''=\chi_\varepsilon g''
+2\chi_\varepsilon'g'+\chi_\varepsilon''g$, together with
\eqref{eq:nl-27}--\eqref{eq:nl-29}, makes its $L^2$ norm tend to zero.
The first and zeroth derivative estimates follow from the same
bounds, proving strong $H^2$ convergence.
Finally every point of $F_{3\varepsilon}\setminus F$ has a nearest
endpoint in $F$ at distance at most $3\varepsilon$. The second
inequality in \eqref{eq:nl-28} gives
$|\chi_\varepsilon'g|\le C\varepsilon^{1/2}\|g''\|_{L^2(I)}$ there.
On $F$ and outside the support this product is zero. This proves its
uniform convergence to zero.
\end{proof}

\subsection{Graph charts near a coincidence set}
Fix $p\in u(\T)$ and its finite fiber $\{s_1,\ldots,s_N\}$.
Choose a unit vector $e$ with $e\cdot T(s_i)\ne0$ for every $i$.
There is such a vector because a finite union of planes cannot cover
$\R^3$. Use orthonormal coordinates $(x,y)\in\R\times\R^2$ with
first axis $e$. On pairwise disjoint parameter arcs about $s_i$,
choose a common $c>0$ with $|e\cdot T|\ge c$; the sign is fixed
on each arc, though it may differ between arcs. The function
$x(s)=e\cdot u(s)$ is strictly monotone there.

Choose a bounded open interval $I$ about $x_p=e\cdot p$ whose closure
lies inside every one of these coordinate images. Restrict each arc
to $U_i=x^{-1}(I)$ within that arc. It has the graph representation
$u(s)=(x(s),f_i(x(s)))$. The inverse coordinate $\sigma_i$ and
\eqref{eq:nl-3}--\eqref{eq:nl-4} show $f_i\in H^2(I;\R^2)$.
Explicitly, writing $y(s)$ for the other coordinates,
\[
 f_i'=(y'/x')\circ\sigma_i,\qquad
 f_i''=\left(\frac{y''}{(x')^2}
                  -\frac{y'x''}{(x')^3}\right)\circ\sigma_i.
\]
The denominator is bounded away from zero, $y',x'$ are bounded,
and $y'',x''$ lie in $L^2$. Change of variables with $|x'|$ verifies
the claimed $L^2$ second derivative.

No other parameters map to $p$. Therefore the compact complement
$\T\setminus\bigcup_iU_i$ has image at positive distance $d_V$ from
$p$. Choose
\[
 r_V=\frac14\min\{d_V,\dist(x_p,\partial I)\},\qquad
 V=B_{r_V}(p),\qquad V_0=B_{r_V/2}(p).
\]
Every parameter whose image enters $V$ belongs to a listed branch,
and the projection of $\overline V$ is compactly inside $I$.
The sets $V_0$ from these constructions cover the compact curve image;
choose a finite subcover. We use only these finitely many charts below.
Their branch numbers may differ.

\subsection{Compact subsets with a fixed group of coincident branches}
Within one chart, fix a branch $a$, a subset
$J\subset\{1,\ldots,N\}$ containing $a$, and $m\in\mathbb N$.
Define the following subset of the coordinate interval:
\begin{equation}
 \begin{split}
 F_{a,J,m}=\{x\in I:\ &\dist(x,\partial I)\ge m^{-1},\\
 &\dist((x,f_a(x)),\R^3\setminus V_0)\ge m^{-1},\\
 &f_b(x)=f_a(x),\ f_b'(x)=f_a'(x)\quad(b\in J),\\
 &|f_b(x)-f_a(x)|\ge m^{-1}\quad(b\notin J)\}.
 \end{split}\label{eq:nl-30}
\end{equation}
All defining functions are continuous, and the first condition confines
the set to a compact subinterval of $I$. Thus this is a compact set.
There are only countably many such sets over our finite chart cover.
We shall prove that every one carries zero defect mass on its graph.

Fix a nonempty $F=F_{a,J,m}$ and write
$K_F=\{(x,f_a(x)):x\in F\}$ and
$K_b=\{(x,f_b(x)):x\in F\}$ for $b\notin J$.
These compact spatial sets are disjoint: equality of spatial points
forces equality of first coordinates, contradicting the last condition
in \eqref{eq:nl-30}. Also $K_F\Subset V_0$.
Set
\[
 d_F=\min\{1,\dist(K_F,\R^3\setminus V),
                         \min_{b\notin J}\dist(K_F,K_b)\}>0,
 \qquad\eta=d_F/8,
\]
omitting the final minimum when there are no excluded branches.
Convolving the indicator of the $2\eta$ neighborhood of $K_F$ with
a nonnegative smooth mollifier of mass one and radius $\eta$ gives
$\psi\in C_c^\infty(V)$, with $0\le\psi\le1$, equal to one on the
$\eta$ neighborhood of $K_F$. Its support lies within distance
$3\eta$ of $K_F$, so it is zero on the $\eta$ neighborhood of every
$K_b$, since $d_F-\eta>3\eta$.

The cutoff must have those exact values also above nearby coordinates.
Here are simultaneous quantitative choices. Set
$d_I=\dist(F,\R\setminus I)$,
$I_F=\{x:\dist(x,F)\le d_I/2\}$,
$L_F=1+\max_b\|f_b'\|_{C(I_F)}$, and
\[
 \varepsilon_0=\min\{d_I/12,\eta/(6L_F)\}.
\]
For $0<\varepsilon<\varepsilon_0$ and $x\in F_{3\varepsilon}$,
choose $x_0\in F$ with $|x-x_0|<3\varepsilon$.
The segment between them lies in $I_F$, and
\[
 |(x,f_b(x))-(x_0,f_b(x_0))|
 \le L_F|x-x_0|<3L_F\varepsilon<\eta/2.
\]
Hence $\psi$ is identically one in a neighborhood of these graph
points for $b\in J$, and identically zero in a neighborhood for
$b\notin J$. In particular all derivatives of $\psi$ vanish at them.
All objects $F,\psi,d_I,L_F$ have been fixed before $\varepsilon$
is allowed to vary.

\subsection{An expanding field whose trace becomes small}
Take $\chi_\varepsilon$ from Lemma~\ref{lem:flat-cutoff} for $F$ and
restrict to $0<\varepsilon<\varepsilon_0$. In the chart define
\begin{equation}
 X_\varepsilon(x,y)=\chi_\varepsilon(x)\psi(x,y)(0,y-f_a(x)).
 \label{eq:nl-31}
\end{equation}
Extend by zero outside the chart. The support of $\psi$ has its
$x$ projection compactly in $I$, so this extension is $C^1$.
The field need not be smooth, because $f_a$ is only $H^2$.
We must justify its use in the smooth-field balance
\eqref{eq:nl-21}.

For each fixed $\varepsilon$, choose smooth $f_{a,n}\to f_a$ in
$H^2(I)$ and replace $f_a$ by $f_{a,n}$ in \eqref{eq:nl-31},
obtaining $X_{\varepsilon,n}\in C_c^\infty(\R^3;\R^3)$.
Then
\begin{equation}
 X_{\varepsilon,n}\circ u\to X_\varepsilon\circ u\text{ in }H^2(\T),
 \qquad
 DX_{\varepsilon,n}(u)\to DX_\varepsilon(u)\text{ uniformly}.
 \label{eq:nl-32}
\end{equation}
For the uniform assertion, only $f_{a,n}-f_a$ and its first derivative
occur in the spatial derivative of the difference; both converge
uniformly. For the trace assertion, every nonzero trace lies in a
captured branch $U_b$. There the difference is a product of fixed
$H^2$ functions (the cutoffs composed with $u$) and
$(f_a-f_{a,n})\circ x$. The bounded composition map
\eqref{eq:nl-3} and bounded $H^2$ product map show convergence in
$H^2(U_b)$. The support stays away from the endpoints of $U_b$,
because the $x$ projection of $\supp\psi$ is compactly in $I$.
Extension by zero therefore introduces no jumps in value or derivative.
Sum the estimates on the finitely many branches to obtain
\eqref{eq:nl-32} on the circle.

Insert $X_{\varepsilon,n}$ into \eqref{eq:nl-21} and let $n\to\infty$.
The first term converges by \eqref{eq:nl-12}; all other terms converge
by uniform convergence of the matrices and finiteness of $ds$ and
$\mu$, using $|A|\le1$. Thus the balance is valid for
$X_\varepsilon$ with exactly the displayed integrals. This argument
is for fixed $\varepsilon$; only after this limit do we send
$\varepsilon$ to zero.

For a branch $b\in J$, set $g_b=f_b-f_a$. It satisfies $g_b=g_b'=0$
on $F$. The constancy of $\psi$ gives the exact trace
$X_\varepsilon\circ u=\chi_\varepsilon(x(s))(0,g_b(x(s)))$
on this branch. For $b\notin J$ the trace is zero wherever
$\chi_\varepsilon$ is supported. Off the chart the trace is zero.
Lemma~\ref{lem:flat-cutoff}, followed by the fixed bounded composition
map, proves
\begin{equation}
 h_\varepsilon:=X_\varepsilon\circ u\longrightarrow0
                              \quad\text{in }H^2(\T).
 \label{eq:nl-33}
\end{equation}
In particular $D\E(u)[h_\varepsilon]\to0$.

\subsection{The derivative still sees the whole normal defect}
Let $e_x=(1,0,0)$ in the current orthonormal coordinates and define
\begin{equation}
 R_a(x)=\Id-(1,f_a'(x))\otimes e_x,
 \qquad R_a(v_x,v_y)=(0,v_y-f_a'(x)v_x).
 \label{eq:nl-34}
\end{equation}
Along a branch $b\in J$ over the support of $\chi_\varepsilon$,
differentiation of \eqref{eq:nl-31} gives
\begin{equation}
 DX_\varepsilon(u)=\chi_\varepsilon R_a
          +\chi_\varepsilon'(0,g_b)\otimes e_x.
 \label{eq:nl-35}
\end{equation}
There are no terms from differentiating $\psi$, since it is constant
on a neighborhood of these points. On the excluded branches all
terms are zero. The second matrix in \eqref{eq:nl-35} tends uniformly
to zero by the flat-zero lemma. The first is uniformly bounded.
Thus $DX_\varepsilon(u)$ is bounded independently of $\varepsilon$
and converges pointwise to $R_a(x)$ on
$S_F:=u^{-1}(K_F)$, and to zero elsewhere. Indeed all preimages of
$K_F$ lie in the captured branches, and \eqref{eq:nl-30} identifies
them precisely with the branches in $J$ above $F$.

At a parameter in $S_F$, the tangent is parallel, possibly with the
opposite orientation, to $w=(1,f_a'(x))$. Therefore $R_aT=0$.
Although $R_a$ is an oblique rather than orthogonal projection, it
has the needed normal contraction. By \eqref{eq:nl-18}, $Aw=0$
for $\mu$-almost every such parameter. Entrywise contraction gives
\begin{equation}
 R_a:A=\tr A-\sum_{\alpha,\beta}w_\alpha(e_x)_\beta A_{\alpha\beta}
       =1-e_x\cdot Aw=1
                        \quad\mu\text{-almost everywhere on }S_F.
 \label{eq:nl-36}
\end{equation}
Symmetry of $A$ is used in the middle identity.

Now pass to the limit in \eqref{eq:nl-21} for $X_\varepsilon$.
The first term tends to zero by \eqref{eq:nl-33} and
\eqref{eq:nl-12}. The scalar
$T\cdot DX_\varepsilon(u)T$ tends to zero everywhere: on $S_F$
use $R_aT=0$, and elsewhere use the zero matrix limit.
Its uniform bound is an integrable constant for both $ds$ and
$\mu$, so both scalar integrals vanish in the limit.
Finally $DX_\varepsilon(u):A$ tends almost everywhere with respect
to $\mu$ to $\mathbf1_{S_F}$, by \eqref{eq:nl-36}.
Its absolute value is bounded by a fixed constant because $|A|\le1$.
Dominated convergence yields
\begin{equation}
 0=2\mu(S_F),\qquad \mu(u^{-1}(K_F))=0.
 \label{eq:nl-37}
\end{equation}
Thus every compact set of the form \eqref{eq:nl-30} has zero defect
mass on its spatial graph.

\subsection{The exceptional points form a countable set}
Fix two graphs $f_a,f_b$ in a chart. A zero of $f_a-f_b$ at which
$f_a'-f_b'\ne0$ is isolated among the coincidence coordinates.
Indeed one scalar component of the derivative is nonzero and keeps
its sign on a neighborhood. That component is strictly monotone,
so it has at most one zero there. The isolated coordinates are
countable, by rational isolating intervals. Taking their spatial
images, then taking the finite union over branch pairs and charts,
gives a countable spatial set $E$. By the spatial atom result,
\begin{equation}
 \mu(u^{-1}(E))=0.
 \label{eq:nl-38}
\end{equation}
The spatial formulation applies directly to the countable image set $E$;
equivalently, its preimage is countable because all fibers are finite.

For any $p\in u(\T)\setminus E$, choose a chart with $p\in V_0$
and a branch $a$ representing it, at coordinate $x$.
Let $J$ consist of all branches whose values at $x$ equal $f_a(x)$.
Their first derivatives agree, because otherwise $p\in E$.
All excluded branches have a positive gap at $x$. Also
$\dist(x,\partial I)>0$ and
$\dist(p,\R^3\setminus V_0)>0$. Take $a_p>0$ to be the minimum
of these finitely many positive quantities, omitting the gap minimum
if no branch is excluded. Choose
$m=1+\lceil a_p^{-1}\rceil$, so $m^{-1}<a_p$.
Then $x\in F_{a,J,m}$.

This includes an isolated single strand ($J=\{a\}$) as well as any
number of tangentially coincident strands. Hence the countably many
compact spatial graphs $K_F$, together with $E$, cover $u(\T)$.
Countable subadditivity, \eqref{eq:nl-37}, and \eqref{eq:nl-38} imply
$\mu(\T)=0$. The entrywise domination by $\mu$ then gives
\begin{equation}
 \M=0.
 \label{eq:nl-39}
\end{equation}
No decomposition of the contact set into finitely many arcs has been
used. In particular the argument covers closed coincidence sets
with positive measure and arbitrarily many gaps.

\section{Strong limits and vanishing self-avoidance penalties}\label{sec:compactness}

\subsection{Proof of Theorem~\ref{thm:no-loss}}
On the further subsequence defining $W$, use $H=\Id$ in
\eqref{eq:nl-14} and $\M=0$ to obtain
\begin{equation}
 \B(u)=\|u''\|_2^2=\lim_j\|v_j''\|_2^2=e_K.
 \label{eq:nl-40}
\end{equation}
The identity depends only on the fixed weak limit $u$, and hence is
independent of the matrix-measure extraction. On the entire original weakly convergent subsequence,
$\|v_j''\|_2^2\to e_K=\|u''\|_2^2$ by its minimizing property.
Together with weak convergence, \eqref{eq:nl-5} gives
\[
 \|v_j''-u''\|_2^2
 =\|v_j''\|_2^2+\|u''\|_2^2
                         -2\int v_j''\cdot u''\,ds\to0.
\]
The functions and first derivatives already converge uniformly,
hence in $L^2$ on the circle of measure one. Thus $v_j\to u$ strongly
in $H^2$ on that entire subsequence. Boundedness in
Section~\ref{sec:defect} supplies a weakly convergent subsequence from
any minimizing sequence; the argument just given makes it strongly
convergent. Anchoring and unit speed pass by $C^1$ convergence.
This proves Theorem~\ref{thm:no-loss}.

\subsection{The ropelength regularization produces minimizing sequences}
We recall the geometric penalty to connect this theorem with elastic
knots. For a regular curve $v$ and distinct parameters $s,t$, let
\[
 r_{\rm tp}[v](s,t)=
 \frac{|v(t)-v(s)|^2}
 {2|P_{\tau(s)}^\perp(v(t)-v(s))|},\qquad
 \tau(s)=v'(s)/|v'(s)|.
\]
This is the tangent-point radius. Give it value $+\infty$ if the
chord is nonzero and its normal projection is zero, and value zero
if the chord is zero. Its infimum over distinct parameters is the
thickness $\Delta(v)=\operatorname{Thi}(v)$. Define
$\Rop(v)=\LL(v)/\Delta(v)$, with value $+\infty$ for zero thickness.
These quantities are independent of regular reparametrization, and
ropelength is invariant under dilation. We use two facts from thickness
theory: a smooth closed embedding has positive thickness, and a regular
embedded curve with positive thickness has a $C^{1,1}$ arclength
parametrization, meaning a continuously differentiable representative
with Lipschitz tangent \cite[Lemma 2]{GMSM2002}. These equivalent
thickness formulations and their curvature/self-distance relations are
developed in \cite{GMSM2002,LitherlandEtAl1999}; see also the summary
in \cite[Section 2.2]{RvM2026}.

For $\theta>0$, consider global minimizers in $\CC(K)$ of
\[
 E_\theta(v)=\B(v)+\theta\Rop(v).
\]
Their existence is \cite[Theorem 2.1]{GRM}; denote any choice by
$v_\theta$. A smooth competitor gives finite minimum, so
$\Rop(v_\theta)<\infty$ and the thickness is positive. The preceding
regularity statement applies in their prescribed unit-speed parameter,
giving $v_\theta\in C^{1,1}$. Weak $H^2$ subsequential limits as
$\theta\downarrow0$ are called elastic knots
\cite[Definition 2.3]{GRM}; see also \cite[Definition 3.1]{RvM2026}.

\begin{proposition}\label{prop:elastic-minimizing}
For every choice of global minimizers,
\[
 \B(v_\theta)\to e_K,\qquad
 \theta\Rop(v_\theta)\to0,\qquad E_\theta(v_\theta)\to e_K.
\]
Every elastic-limit subsequence converges strongly in $H^2$ to its
weak limit, whose bending energy equals $e_K$.
\end{proposition}
\begin{proof}
We first check that smooth competitors compute the same bending
infimum. We use the following knot-theoretic stability result: a
regular $C^1$ embedded closed curve has a $C^1$ neighborhood consisting
of regular embeddings ambient isotopic to it \cite{Reiter2005};
see also \cite[Corollary 1.5]{Blatt2009} and
\cite[Lemma 2.3]{RvM2026}.
For $v\in\CC(K)$, periodic convolution yields smooth $w_n\to v$
strongly in $H^2$, and hence in $C^1$. Eventually they lie in that
neighborhood. Their speeds are bounded away from zero. In the
representation $\B(w)=\int w''\cdot G(w')w''$ used above,
$G(w_n')\to G(v')$ uniformly, and $w_n''\to v''$ in $L^2$.
Indeed replacing the coefficient costs at most
$\|G(w_n')-G(v')\|_\infty\|w_n''\|_2^2$, and replacing the two
second-derivative factors costs at most
$\|G(v')\|_\infty\|w_n''-v''\|_2(\|w_n''\|_2+\|v''\|_2)$.
Both errors tend to zero. Length also converges by uniform convergence
of speed, so $\E(w_n)\to\E(v)=\B(v)$. Normalize each smooth $w_n$
using \eqref{eq:nl-8}. This proves the assertion about the infimum.

Fix $\eta>0$ and choose a smooth $w_\eta\in\CC(K)$ with
$\B(w_\eta)<e_K+\eta$. Its ropelength
$R_\eta=\Rop(w_\eta)$ is finite. Minimality and positivity give
\[
 e_K\le\B(v_\theta)\le E_\theta(v_\theta)
                            \le e_K+\eta+\theta R_\eta.
\]
For $0<\theta<\eta/(1+R_\eta)$ the excess is below $2\eta$.
Since $w_\eta$ is fixed before $\theta$ tends to zero, these inequalities
prove both energy limits. Finally
$0\le\theta\Rop(v_\theta)
=E_\theta(v_\theta)-\B(v_\theta)
\le E_\theta(v_\theta)-e_K\to0$.
Theorem~\ref{thm:no-loss} proves the strong convergence assertion.
\end{proof}

\begin{remark}
All later conclusions based on no loss and ambient stationarity apply
to every strong limit of a bending-minimizing sequence. Elastic knots
are among these limits. The converse selection statement, that every
such limit is obtained from ropelength minimizers, has not been proved.
\end{remark}

\subsection{Different closures}
Let $S_K$ be the strong $H^2$ closure of $\CC(K)$, let
$W_K^{\rm seq}$ be the set of weak $H^2$ limits of sequences in
$\CC(K)$, and put $D_K=\overline{\CC(K)}^{C^1}\cap H^2$.
Their members are anchored and unit speed. For weak sequential
limits this follows from compact embedding, as at the start of
Section~\ref{sec:defect}. Strong $H^2$ convergence implies weak $H^2$
and $C^1$ convergence, so $S_K\subset W_K^{\rm seq}\subset D_K$.
The recovery theorem in Appendix~\ref{rc:section},
Theorem~\ref{rc:theorem}, proves the reverse inclusion by
constructing strongly convergent embedded approximants for every
$u\in D_K$. Consequently
\[
 S_K=W_K^{\rm seq}=D_K,\qquad
 \inf_{S_K}\B=\inf_{W_K^{\rm seq}}\B=\inf_{D_K}\B=e_K.
\]
For the infimum statement, continuity of $\B$ on strong $H^2$ limits
of unit-speed curves gives $\B(u)\ge e_K$ for $u\in S_K$, while
$\CC(K)\subset S_K$ gives the other inequality.
The recovery construction preserves the topology of its local
replacements even when the original $C^1$ approximants have unbounded
bending energy. It is independent of the no-loss proof above.
The later extension to arbitrary $C^1$ ambient fields with $H^2$ traces
uses this recovery result. The no-loss proof uses smooth ambient tests
and the specific $C^1$ fields \eqref{eq:nl-31}, whose use is justified
by the strong $H^2$ trace approximation \eqref{eq:nl-32}; it does not
use full closure recovery.

\begin{remark}
The notation $W_K^{\rm seq}$ concerns limits of sequences. We do not
identify it with closure in the unrestricted weak topology. A weakly
convergent sequence is bounded, whereas arbitrary weak neighborhoods
do not impose a uniform energy bound. The theorem about minimizing
sequences therefore supplies no assertion about general weak nets.
\end{remark}

\section{The free equation and the contact reaction}\label{ct:sec:structure}
A material variation assigns a vector to each parameter. An ambient
variation assigns a vector to each spatial point. Thus, when two parameters
have the same image, an ambient variation moves them together. This
restriction explains both the free equation away from contact and the
balance laws at contact.

For positive-thickness rod models, Schuricht and von der Mosel derive
contact-pair measures and balance laws \cite{SchurichtMosel2003,
SchurichtMosel2004}. Earlier contact regularity for a distance-obstacle model in
\cite[Theorems 3.2--3.3]{Mosel1999} includes free arcs and, under
its geometric assumption (G), $W^{3,1}$ regularity near isolated simple
touching parameters. Here the starting information is ambient
stationarity at arbitrary coincidence sets. The resulting full material
distribution need not be a measure; it is not identified with the
multipliers of the positive-thickness constraint.

In this section let $u:\T\to\R^3$ be a normalized strong limit supplied
by Theorem~\ref{thm:no-loss}. Write
\[
 T=u',\qquad k=u''\in L^2(\T;\R^3),\qquad
 \sigma=\B(u)=e_K,\qquad |T|=1.
\]
All derivatives in this section are with respect to the material
arclength parameter $s$ of the base curve. Specializing the previously
derived arbitrary-parametrization formula gives
\begin{equation}\label{ct:eq:variation}
 D\E(u)[h]=\int_\T
 \{2k\cdot h''+(\sigma-3|k|^2)T\cdot h'\}\,\dd s,
 \qquad h\in H^2(\T;\R^3).
\end{equation}
The line $u+\varepsilon h$ need not have unit speed. In particular,
\eqref{ct:eq:variation} already includes its changes in speed and length.

\subsection{Which ambient variations are available?}
\begin{proposition}[Ambient minimality]\label{ct:prop:ambient}
For every $C^2$ diffeomorphism $\Phi:\R^3\to\R^3$ ambient isotopic
to the identity in the sense of Section~\ref{sec:framework},
\[
 \E(\Phi\circ u)\ge\E(u).
\]
Consequently
\begin{equation}\label{ct:eq:ambient}
 D\E(u)[H\circ u]=0\qquad
       (H\in C_c^2(\R^3;\R^3)).
\end{equation}
\end{proposition}
\begin{proof}
Choose normalized embeddings $v_n$ of type $K$ with $v_n\to u$ in $H^2$
and $\B(v_n)\to e_K$. Their images and the image of $u$ lie in one compact
set. On that set $D\Phi$ and $D^2\Phi$ are bounded and uniformly continuous.
The identity
\[
 (\Phi\circ v_n)''=D^2\Phi(v_n)[v_n',v_n']+D\Phi(v_n)v_n''
\]
therefore proves convergence of the second derivatives in $L^2$: the first
summand converges uniformly, while for the second summand the error is
bounded by
\[
 \|D\Phi(v_n)-D\Phi(u)\|_\infty\|v_n''\|_2
       +\|D\Phi(u)\|_\infty\|v_n''-u''\|_2.
\]
The analogous formulas for orders zero and one give strong $H^2$
convergence. Invertibility of $D\Phi$ and compactness imply a positive
lower speed bound for $\Phi\circ u$. The curves $\Phi\circ v_n$ are
embeddings of type $K$. Normalize each by translation, dilation, and
arclength reparametrization as in Section~\ref{sec:variation}. Since this
preserves $\E$, their energies are at least $e_K$. Continuity of $\E$ on
regular $H^2$ curves gives the claimed inequality.

For the differential statement fix $H$, put $a=\|DH\|_\infty$, and take
$|\varepsilon|<[2(1+a)]^{-1}$. Then
\[
 |(x+\varepsilon H(x))-(y+\varepsilon H(y))|
           \ge(1-|\varepsilon|a)|x-y|.
\]
The map $\Phi_\varepsilon=\Id+\varepsilon H$ is onto as well: for each
$y$, $x\mapsto y-\varepsilon H(x)$ is a contraction of the complete space
$\R^3$, so its fixed point is the required preimage. The derivative
$\Id+\varepsilon DH$ is invertible, and the inverse function theorem
makes the inverse $C^2$. Replacing $\varepsilon$ by $t\varepsilon$,
$0\le t\le1$, gives an ambient isotopy. Moreover
\[
 (H\circ u)''=D^2H(u)[T,T]+DH(u)k\in L^2,
\]
so $h=H\circ u$ is an admissible $H^2$ direction. Minimality for positive
and negative $\varepsilon$ makes the derivative of
$\E(u+\varepsilon h)$ at zero equal to zero.
\end{proof}

\begin{proposition}[An additional $C^1$ statement]\label{ct:prop:C1ambient}
The recovery theorem in Appendix~\ref{rc:section} additionally gives
ambient minimality for a $C^1$ diffeomorphism $\Phi$ ambient isotopic
to the identity in the sense of Section~\ref{sec:framework}, provided
$\Phi\circ u$ is regular and $H^2$. It gives
$D\E(u)[H\circ u]=0$ for $H\in C_c^1$ provided $H\circ u\in H^2$.
\end{proposition}
\begin{proof}
A $C^1$ sequence of same-class embeddings converging to $u$ remains a
$C^1$ sequence of same-class embeddings after composition with $\Phi$.
By $C^1$ isotopy stability \cite{Reiter2005,Blatt2009}, smooth each
composed embedding within its isotopy neighborhood, with errors tending
to zero. Thus the regular $H^2$ curve $\Phi\circ u$ belongs to the same
geometric $C^1$ closure. Normalization \eqref{eq:nl-8} is continuous
in $C^1$: positive speeds, lengths, arclength maps and their inverses
converge uniformly, and the normalized tangent formula gives convergence
of derivatives. Its normalized curve therefore lies in $D_K$.
Use Theorem~\ref{rc:theorem} to obtain
$\E(\Phi\circ u)\ge e_K$. For $H$ the maps $\Id+\varepsilon H$ are
$C^1$ diffeomorphisms and an isotopy by the preceding contraction
argument. Their traces are the regular $H^2$ curves
$u+\varepsilon(H\circ u)$, so differentiating their two-sided energy
inequality proves stationarity.
\end{proof}
\begin{remark}\label{ct:rem:ambienttrace}
The assumption $H\circ u\in H^2$ cannot be omitted for a general $C^1$
field. All the localization arguments below instead use smooth fields;
they do not depend on Proposition~\ref{ct:prop:C1ambient} or on recovery
in the full $C^1$ closure.
\end{remark}

\subsection{Contact parameters and local ambient tests}
The contact set in parameter space is
\begin{equation}\label{ct:eq:contactset}
 S_u=\{s\in\T:\text{there is }t\ne s\text{ with }u(t)=u(s)\}.
\end{equation}
A fiber is the set $u^{-1}(p)$ of all visits to a spatial point $p$.
Contact is a geometric coincidence; it does not yet say that the first
variation has a nonzero residual there.

\begin{lemma}[Separation of visits]\label{ct:lem:fibers}
A regular closed $C^1$ curve is injective on sufficiently short parameter
intervals, with one interval length valid everywhere. Its fibers have a
uniform finite cardinality bound, and its parameter contact set is closed.
\end{lemma}
\begin{proof}
First suppose the curve is $u$ and is unit speed. Uniform continuity of
$T$ permits $0<\eta<1/4$ such that
$|T(t)-T(s)|\le1/2$ whenever $\dist_\T(s,t)\le\eta$.
Use a real lift of the shorter interval from $s$ to $s+r$, $|r|\le\eta$.
Since $T(s)\cdot T(t)\ge1/2$ along it, integration gives
\[
 |u(s+r)-u(s)|\ge |r|/2\quad(0<|r|\le\eta).
\]
For a general regular curve $w$, use its unit tangent and
$c_0=\min|w'|>0$; the same argument gives $c_0|r|/2$.
Partition a fundamental parameter period into finitely many half-open
intervals shorter than $\eta$. Each has at most one preimage of a fixed
point, giving the fiber bound. Finally, if $s_n\in S_w$ tends to $s$,
choose partners $t_n\ne s_n$. The chord bound forces
$\dist_\T(s_n,t_n)>\eta$. A subsequence of $t_n$ converges to $t$;
continuity yields $w(t)=w(s)$ and $\dist_\T(s,t)\ge\eta$. Thus $s\in S_w$.
\end{proof}

We need to turn a test supported on an isolated material arc into a limit
of ambient tests. The next statement also handles several exact copies
of an arc. By an affine isometry of parameter intervals we mean, in
chosen real lifts, a map $s\mapsto\varepsilon s+b$ with
$\varepsilon\in\{-1,1\}$.
\begin{lemma}[Density of smooth ambient traces]\label{ct:lem:tracedensity}
Let $W$ be an interval where $u$ is injective. Suppose disjoint intervals
$W_j=\varphi_j(W)$, $1\le j\le m$, satisfy $\varphi_1=\Id$ and
\begin{equation}\label{ct:eq:bundle}
 u(\varphi_j(s))=u(s),\qquad
 u^{-1}(u(s))=\{\varphi_1(s),\ldots,\varphi_m(s)\}\quad(s\in W),
\end{equation}
where every $\varphi_j$ is an affine isometry. Given
$h\in C_c^\infty(W;\R^3)$, define $\widetilde h$ on $\T$ by
$\widetilde h(\varphi_j(s))=h(s)$ and zero elsewhere. There are
$H_n\in C_c^\infty(\R^3;\R^3)$ with
$H_n\circ u\to\widetilde h$ strongly in $H^2$.
\end{lemma}
\begin{proof}
Choose $\delta>0$ so small that tangents at circular distance at most
$2\delta$ differ by at most $1/2$. For the compact set $F=\supp h\Subset W$
choose
\[
 0<\rho<\min\{\delta/4,\dist_\T(F,\T\setminus W)/8\}.
\]
A finite smooth partition of unity near $F$ reduces $h$ to summands
supported in radius-$\rho$ intervals centered at points $s_*$ of $F$.
Such a partition is obtained from bumps supported in those intervals,
positive on radius-$\rho/2$ intervals covering $F$, by dividing by their
sum near $F$. It suffices to construct the traces for one summand.
Let
\[
 U_0=(s_*-3\rho/2,s_*+3\rho/2),\quad
 U_1=(s_*-2\rho,s_*+2\rho),\quad e=T(s_*),\quad q=e\cdot u.
\]
Their closures lie in $W$, the summand is supported in $U_0$, and
$q'\ge1/2$ on $U_1$. Thus $q$ is a $C^1$ bi-Lipschitz coordinate there.
Its inverse $r:q(U_1)\to U_1$ satisfies, almost everywhere,
\[
 r'=(q'\circ r)^{-1},\qquad r''=-\frac{q''\circ r}{(q'\circ r)^3},
 \qquad
 \int_{q(U_1)}|r''|^2\,\dd x
       =\int_{U_1}\frac{|q''|^2}{|q'|^5}\,\dd s\le32\|q''\|_2^2.
\]
Indeed $q'$ is absolutely continuous and separated from zero; its
reciprocal is absolutely continuous, and composition with the
bi-Lipschitz $r$ preserves that property and the one-dimensional
change-of-variable formula. These facts prove the displayed weak
second-derivative identity.

The vector function $f=h\circ r$, extended by zero to $\R$, belongs to
$H^2(\R)$ and has compact support in $q(U_0)$. Its derivatives are
$f'=h'(r)r'$ and $f''=h''(r)(r')^2+h'(r)r''$.
Choose a compact interval $J\Subset q(U_0)$ containing its support in
its interior. Convolution with smooth mollifiers of sufficiently small
radius gives $f_n\in C_c^\infty(\operatorname{int}J;\R^3)$ with
$f_n\to f$ in $H^2$ and $f_n'\to f'$ uniformly. The last assertion is
the one-dimensional embedding $H^1\subset C^0$ applied to $f_n'-f'$;
see \cite[Theorem 8.2 and Theorem 4.22]{Brezis}.

The compact sets
\[
 A=u(r(J)),\qquad
 C=u\left(\T\setminus\bigcup_{j=1}^m\varphi_j(U_1)\right)
\]
are disjoint by the exact-fiber condition. Choose a smooth compactly
supported function $\chi$ equal to one near $A$ and zero near $C$.
Here $r$ is the inverse on $U_1$ only, so $A$ contains precisely the
selected bundle segment, not other strands crossing the same height
slab. If $d=\dist(A,C)>0$, mollifying the indicator of the $d/4$
neighborhood of $A$ at radius less than $d/8$ gives such a cutoff.
Set $H_n(x)=\chi(x)f_n(e\cdot x)$. Outside the branch intervals,
$\chi$ vanishes near the image. On a branch $s=\varphi_j(t)$ with
$t\in U_1$, one has $u(s)=u(t)$; whenever $f_n(q(t))\ne0$,
$q(t)\in J$ and hence $u(t)\in A$. Thus the cutoff is identically
one there, and the trace on that branch is exactly $f_n(q(t))$.
The same argument applies to $f$.

For $g=f_n-f$, the chain rule gives
\[
 (g\circ q)''=g''(q)(q')^2+g'(q)q''.
\]
If $C_q=\|q'\|_\infty$, change of variables bounds the squared $L^2$
norm of the first term by $2C_q^4\|g''\|_2^2$. The second has $L^2$
norm at most $\|g'\|_\infty\|q''\|_2$. Both tend to zero; the same
substitution controls orders zero and one. Isometric pullback gives
these estimates on every copy. Adding the finitely many localized
summands proves the assertion.
\end{proof}

\subsection{A distribution supported at contact}
A vector distribution acts linearly on smooth vector test functions,
with each component continuous in the usual test-function topology.
For an integrable vector function $g$ its derivative is defined by
$\langle Dg,h\rangle=-\int g\cdot h'$. This defines derivatives even
when no classical derivative exists. The support of a distribution is
the complement of the union of all open sets on which it annihilates
every compactly supported smooth test.
\begin{definition}\label{ct:def:reaction}
The full material reaction is the parameter distribution
\begin{equation}\label{ct:eq:reaction}
 \mathfrak f_u=2D^2k-\sigma k+3D(|k|^2T),\qquad
 R_u=\supp\mathfrak f_u\subset\T.
\end{equation}
\end{definition}
Integration by parts in the definition gives
$\langle\mathfrak f_u,h\rangle=D\E(u)[h]$. In particular the signs in
\eqref{ct:eq:reaction} follow from $DT=k$ and
$\langle D(|k|^2T),h\rangle=-\int|k|^2T\cdot h'$.
With $\|h'\|_\infty\le C_S\|h\|_{H^2}$,
\begin{equation}\label{ct:eq:reactionbound}
 |\langle\mathfrak f_u,h\rangle|
 \le\{2\|k\|_2+|\sigma|+3C_S\|k\|_2^2\}\|h\|_{H^2}.
\end{equation}
Here the $\sigma$ term uses $\|T\|_2=1$. Thus $\mathfrak f_u\in H^{-2}$,
the continuous dual of $H^2$. On the unit parameter circle,
$\|h\|_{H^2}\le\sqrt3\max_{j\le2}\|h^{(j)}\|_\infty$, so it also has
distributional order at most two. This bound does not make it a measure.

\begin{theorem}[Free arcs, including their endpoint traces]\label{ct:thm:freearcs}
One has $R_u\subset S_u$. On a connected open interval $I$ where
$\mathfrak f_u=0$, there is a vector $J_I\in\R^3$ with
\begin{equation}\label{ct:eq:firstintegral}
 2k'+(3|k|^2-\sigma)T=J_I,\qquad |k|^2=\sigma+J_I\cdot T.
\end{equation}
The curve is real analytic on $I$, and on that whole interval
\begin{equation}\label{ct:eq:endpointbounds}
 |k|\le\sqrt{\sigma+|J_I|},\qquad
 |k'|\le\tfrac12\{|J_I|+3(\sigma+|J_I|)+\sigma\}.
\end{equation}
At either endpoint of $I$ in a real parameter lift, $(u,T,k)$ has
a one-sided limit. The free ODE solution extends analytically through
that limiting state, and all one-sided derivatives have finite traces.
This continuation need not agree with the original curve beyond $I$.
\end{theorem}
\begin{proof}
Take a smooth test with compact support in $\T\setminus S_u$. Cover its
support by finitely many short injective intervals whose closures are
still disjoint from $S_u$, and split the test using a smooth partition.
Each interval has only one visit to each of its image points, so
Lemma~\ref{ct:lem:tracedensity} with $m=1$ applies to each summand.
Equation~\eqref{ct:eq:ambient} and the bound
\eqref{ct:eq:reactionbound} pass zero first variation to its $H^2$
limit. Hence the original test is annihilated, proving $R_u\subset S_u$.

On $I$, equation~\eqref{ct:eq:reaction} says that
$D\{2Dk+(3|k|^2-\sigma)T\}=0$. A distribution with zero derivative
on an interval is constant: any compactly supported smooth test of
integral zero has a compactly supported smooth primitive, on which the
derivative identity can be tested. Subtracting a fixed test times the
integral proves the assertion for all tests. Apply this componentwise
to obtain the first identity in \eqref{ct:eq:firstintegral}.
It first gives $Dk\in L^1_{\rm loc}$. Integrating that density shows
that $k$ has a locally absolutely continuous representative. Then $k$
is continuous and the same equation gives $k\in C^1$. Repeated use of
the equation proves smoothness.

Now $T\cdot k=0$ follows by differentiating $|T|^2=1$, and
$T\cdot k'=-|k|^2$ follows by differentiating $T\cdot k=0$.
The scalar product of the first integral with $T$ is therefore
$-2|k|^2+3|k|^2-\sigma=J_I\cdot T$. This proves the second
identity and both bounds.

For analyticity consider the finite-dimensional initial-value problem
\begin{equation}\label{ct:eq:odesystem}
 u'=T,\qquad T'=k,\qquad
 k'=\tfrac12\{J_I-(3|k|^2-\sigma)T\}.
\end{equation}
Its vector field is polynomial in $(u,T,k)\in\R^9$; in particular it
is locally Lipschitz and real analytic. The analytic ODE theorem used
here states that an analytic finite-dimensional vector field has a
unique local analytic solution through every state, and every classical
solution through that state agrees with it. One can obtain the needed
vector version from the complex Picard proof in
\cite[Section 4.1]{Teschl}: extend the polynomial to $\mathbb C^9$,
replacing $|k|^2$ by $\sum k_i^2$; on a complex state ball of radius $R$,
let $M\ge1+\sup|P|$ and $A\ge1+\sup\|DP\|$. On the time disk of radius
$r=\min\{R/(4M),1/(4A)\}$, the integral operator
$y\mapsto y_0+\int_0^zP(y(\zeta))\,\dd\zeta$ preserves the ball of
holomorphic maps with $\|y-y_0\|_\infty\le R/2$ and contracts by at
most $1/4$. Uniform limits on smaller disks are holomorphic by the
Cauchy formula \cite[Theorems B.10--B.11]{Knapp}; termwise integration
of the local power series \cite[Theorem B.21]{Knapp} gives the
holomorphic primitive used by the operator. Geometric convergence of
the iterates gives the solution. Complex conjugation and uniqueness
make it real on real time. Real uniqueness follows from the integral
Lipschitz estimate, iterated to give $A^nt^n/n!$ times a fixed bound.
These observations check the version of the analytic theorem needed
for the nine-dimensional system.

Finally $u$ is bounded, $|T|=1$, and \eqref{ct:eq:endpointbounds} bounds
both $k$ and $k'$. Thus $u,T,k$ are Lipschitz on $I$ and have limits
at each endpoint of a real lift of $I$. Solve \eqref{ct:eq:odesystem} through that
limiting state. Local uniqueness identifies the new solution with the
old one on the one-sided overlap, proving extension and all the traces.
\end{proof}
\begin{remark}\label{ct:rem:componentbounds}
The bounds depend on the free component through $J_I$. With infinitely
many components this argument supplies no common bound. It therefore
does not prove global $C^{1,1}$ regularity or finiteness of the reaction
support. Two free solutions meeting at an endpoint may have different
curvature traces.
\end{remark}

\section{Isolated reactions}\label{ct:sec:balance}
The set $R_u$ lives in the material circle. An isolated point of $R_u$
has two free intervals next to it, so their endpoint data describe its
entire local reaction. Balancing a spatial contact needs this description
at every visit to that spatial point.

\begin{proposition}[Force and derivative of a point mass]\label{ct:prop:localreaction}
Let $s_0$ be isolated in $R_u$. Locally $u\in C^{1,1}$, and $u$ is
analytic on either side of $s_0$. Write $[g]=g(s_0^+)-g(s_0^-)$ and
$T_0=T(s_0)$. In a neighborhood of $s_0$,
\begin{equation}\label{ct:eq:localreaction}
 \mathfrak f_u=a\delta_{s_0}+b\delta'_{s_0},\qquad
 a=[J]=2[k']+3[|k|^2]T_0,\qquad b=2[k].
\end{equation}
The coefficients satisfy
\begin{equation}\label{ct:eq:localnormal}
 b\cdot T_0=0,\qquad a\cdot T_0=[|k|^2].
\end{equation}
In particular $b=0$ implies $a\perp T_0$. At a point outside $R_u$
the same local description holds with $a=b=0$.
\end{proposition}
\begin{proof}
Choose a real interval about $s_0$ containing no other point of $R_u$.
Theorem~\ref{ct:thm:freearcs} bounds $k,k'$ on its two halves and gives
all one-sided traces. The continuous tangent has equal traces from both
sides, and the separate bounds for $T'=k$ make $T$ Lipschitz across the
join. Thus $u$ is $C^{1,1}$ there.

For a function $g$ smooth on each side, integration by parts against
$h\in C_c^\infty$ gives $Dg=\{g'\}+[g]\delta_{s_0}$; braces mean the
ordinary piecewise function. Applying this to $k$ and then to its
piecewise derivative yields
\[
 D^2k=\{k''\}+[k']\delta_{s_0}+[k]\delta'_{s_0}.
\]
Likewise, continuity of $T$ gives
$D(|k|^2T)=\{(|k|^2T)'\}+[|k|^2]T_0\delta_{s_0}$.
Insert these identities in \eqref{ct:eq:reaction}. The ordinary parts
vanish on both free halves, giving the stated coefficients. Taking
one-sided traces of \eqref{ct:eq:firstintegral} also gives $a=[J]$.
No second derivative of a point mass occurs: the calculation differentiates
$k$ only twice and $k$ has bounded one-sided traces.
Finally $k^\pm\cdot T_0=0$ proves $b\cdot T_0=0$, while
$J^\pm\cdot T_0=|k^\pm|^2-\sigma$ proves the second identity.
If $b=0$, the curvature traces agree and their squared norms agree.
\end{proof}
Here $\langle\delta_{s_0},h\rangle=h(s_0)$ componentwise, whereas
$\langle b\delta'_{s_0},h\rangle=-b\cdot h'(s_0)$.
The second term is often called a dipole: it couples to a derivative
of the material variation. Both point evaluations are continuous on
$H^2$ in one dimension, so the local formula also holds on local $H^2$
tests by smooth approximation.

\begin{proposition}[Balance at a complete fiber]\label{ct:prop:fiberbalance}
Let $u^{-1}(p)=\{s_1,\ldots,s_m\}$. Suppose each $s_i$ either is
isolated in $R_u$ or belongs to its complement. Denote its coefficients
by $a_i,b_i$ and its tangent by $T_i=T(s_i)$. Then
\begin{equation}\label{ct:eq:fiberbalance}
 \sum_i a_i=0,\qquad \sum_i b_i\otimes T_i=0.
\end{equation}
The tensor $v\otimes w$ is the matrix with entries $v_rw_\ell$, acting
on $z$ as $v(w\cdot z)$.
\end{proposition}
\begin{proof}
Choose disjoint parameter intervals $U_i$ about the $s_i$, each carrying
the local formula and no other reaction point. The image of the compact
complement of $\bigcup_iU_i$ omits $p$. Its distance $d$ from $p$ is
positive. Take a smooth function $\chi$ supported in $B_{d/2}(p)$ and
identically one on $B_{d/4}(p)$. For an arbitrary vector $v$ and matrix
$A\in\R^{3\times3}$ use
\[
 H(x)=\chi(x)\{v+A(x-p)\}.
\]
Its trace is supported in $\bigcup_iU_i$, and at $s_i$ it has value
$v$ and derivative $AT_i$. Stationarity gives
\[
 0=\sum_i a_i\cdot v-\sum_i b_i\cdot AT_i.
\]
For clarity, the matrix term is
$\sum_{r,\ell}A_{r\ell}\sum_i(b_i)_r(T_i)_\ell$.
Set $A=0$ and vary each coordinate of $v$ to obtain the first identity.
Set $v=0$ and choose $A$ successively with one entry equal to one and
all others zero to obtain every entry of the second identity.
\end{proof}
\begin{remark}\label{ct:rem:wholefiber}
Isolation at just one visit does not give this finite balance law. If
another visit to $p$ is an accumulation point of $R_u$, a spatial cutoff
near $p$ also samples that strand's reaction. The proposition requires
control at the whole fiber. It requires isolation in $R_u$, rather than
in the possibly larger set $S_u$.
\end{remark}

\subsection{Tangent directions}
\begin{corollary}\label{ct:cor:independent}
Under the hypotheses of Proposition~\ref{ct:prop:fiberbalance}, linearly
independent $T_1,\ldots,T_m$ force $b_i=0$ for every $i$. In this case
\[
 a_i=2[k']_i,\qquad a_i\cdot T_i=0,\qquad \sum_i a_i=0.
\]
This applies to a nonparallel double fiber and to a triple whose three
tangents are linearly independent.
\end{corollary}
\begin{proof}
For each $i$ solve the invertible Gram system to find
$w_i\in\operatorname{span}\{T_1,\ldots,T_m\}$ with
$w_i\cdot T_j=\delta_{ij}$. Applying $\sum_jb_j\otimes T_j$ to $w_i$
gives $b_i=0$. Thus $[k]_i=0$, hence $[|k|^2]_i=0$; the formulas in
Proposition~\ref{ct:prop:localreaction} give all remaining assertions.
\end{proof}
\begin{corollary}\label{ct:cor:finiteindependent}
If $S_u$ is finite and each multiple fiber is a nonparallel double or
an independent triple, then $u\in W^{3,\infty}\cap C^{2,1}$. Its full
reaction is a finite sum of point forces.
\end{corollary}
\begin{proof}
There are finitely many free intervals, and their endpoint estimates
have a finite common maximum. The previous corollary identifies the
curvature traces across every cut. Thus the separate Lipschitz
representatives of $k$ join to a globally Lipschitz function: subdivide
any parameter interval at the finitely many cuts and add the Lipschitz
estimates. Its weak derivative is bounded, with no jump terms. Hence
$u\in W^{3,\infty}$ and its continuous second derivative is Lipschitz.
The local reaction formula has no dipoles and summing it over the cuts
proves the final statement. If $S_u$ is empty the free equation holds
on the whole circle and gives the same conclusion directly.
\end{proof}

At a nonparallel double, put
$n_{12}=(T_1\times T_2)/|T_1\times T_2|$. The force balance and
normality imply
\begin{equation}\label{ct:eq:doubleforce}
 a_1=\alpha_{12}n_{12},\qquad a_2=-\alpha_{12}n_{12},
 \qquad \alpha_{12}\in\R.
\end{equation}
Indeed a vector perpendicular to both tangents lies on the line spanned
by their cross product. An independent triple has a related unique
representation.
\begin{proposition}\label{ct:prop:tripleforces}
For an independent triple, let
$n_{ij}=(T_i\times T_j)/|T_i\times T_j|$, $i<j$. There are unique real
numbers $\alpha_{12},\alpha_{13},\alpha_{23}$ such that
\[
 \begin{split}
 a_1&=\alpha_{12}n_{12}+\alpha_{13}n_{13},\\
 a_2&=-\alpha_{12}n_{12}+\alpha_{23}n_{23},\\
 a_3&=-\alpha_{13}n_{13}-\alpha_{23}n_{23}.
 \end{split}
\]
\end{proposition}
\begin{proof}
The constraints $a_i\cdot T_i=0$ and $\sum_i a_i=0$ are six independent
scalar equations in nine unknown coordinates. To check independence,
a linear relation between the six constraints would give
$v+\lambda_iT_i=0$ for all $i$, with the same $v$. Independence of the
three $T_i$ forces $v=0$ and then every $\lambda_i=0$. Their solution
space therefore has dimension three. The displayed formula defines a
linear map from $\R^3$ into this space. It is injective: if its first
component vanishes, independence of $n_{12},n_{13}$ gives
$\alpha_{12}=\alpha_{13}=0$. Their independence follows because
$T_1\times(cT_2+dT_3)=0$ would make $cT_2+dT_3$ parallel to $T_1$.
The second component then gives $\alpha_{23}=0$. Equal finite
dimensions now give surjectivity and uniqueness.
\end{proof}
There is no sign conclusion for these coefficients. Ambient stationarity
alone provides an equality for two-sided spatial variations; a sign
restriction would require additional information about admissible
one-sided resolutions in the knot class.

For dependent directions, group the visits by their unoriented tangent
lines. Choose a unit representative $v_\ell$ for each line and write
$T_i=\varepsilon_iv_\ell$ on its group $G_\ell$, where
$\varepsilon_i=\pm1$. Define
$d_\ell=\sum_{i\in G_\ell}\varepsilon_i b_i$. The precise tensor law is
\begin{equation}\label{ct:eq:groupbalance}
 \sum_\ell d_\ell\otimes v_\ell=0,\qquad d_\ell\perp v_\ell.
\end{equation}
If the line representatives are independent, the same dual-vector test
forces every $d_\ell=0$. More generally this holds for a particular
$\ell$ whenever $v_\ell$ is outside the span of the other representatives:
apply the tensor to the nonzero projection of $v_\ell$ onto that span's
orthogonal complement. Thus a singleton in such a direction has zero
dipole, while a parallel group may retain dipoles with cancelling signed
sum. For a double with $T_2=\varepsilon T_1$ this is
$b_1+\varepsilon b_2=0$. Reversing the material orientation of a branch
reverses both its jump $[k]$ and its tangent, preserving $b\otimes T$.

In matrix notation, let $V,B\in\R^{3\times m}$ have columns $T_i,b_i$.
Then the law is $BV^{\mathsf T}=0$; the transpose of each row of $B$
lies in $\ker V$, together with $b_i\cdot T_i=0$. Pairwise nonparallel
tangents do not suffice for individual cancellation. For example
\[
 T_1=e_1,\quad T_2=e_2,\quad T_3=(e_1+e_2)/\sqrt2,\qquad
 b_1=b_2=e_3,\quad b_3=-\sqrt2e_3
\]
have normal nonzero dipoles and zero tensor sum. This is an algebraic
example, not an asserted minimizing curve. It identifies why a coplanar
triple does not fall under the independent-triple argument.

\begin{remark}\label{ct:rem:defectdistinction}
The reaction tensors $b_i\otimes T_i$ differ from the curvature-energy
defect in the proof of Theorem~\ref{thm:no-loss}. The former have rank
at most one and express a fiber balance. The latter was positive
semidefinite with $AT=0$, so could have rank two, and was proved to
vanish. Vanishing of that defect does not imply vanishing of the
contact reaction.
\end{remark}

\section{Overlapping arcs and the size of reaction support}
\label{ct:sec:nowheredense}
Contact intervals require a different localization from isolated contacts.
On an exact overlap an ambient test has the same value on every copy,
and all copies contribute the same first variation. They therefore
reinforce the free equation.

\begin{proposition}\label{ct:prop:bundlefree}
For a bundle satisfying \eqref{ct:eq:bundle}, the reaction vanishes on
every $W_j$. Each copy is an analytic free elastica with multiplier
$\sigma$.
\end{proposition}
\begin{proof}
For $h\in C_c^\infty(W;\R^3)$ form the repeated test $\widetilde h$.
Lemma~\ref{ct:lem:tracedensity}, ambient stationarity, and continuity of
the reaction on $H^2$ give
$\langle\mathfrak f_u,\widetilde h\rangle=0$.
Write $\varphi_j(s)=\varepsilon_js+b_j$. The weak chain rule for
$u\circ\varphi_j=u$ gives
\[
 T(\varphi_j(s))=\varepsilon_jT(s),\quad
 k(\varphi_j(s))=k(s),\quad
 \widetilde h'(\varphi_j(s))=\varepsilon_jh'(s),\quad
 \widetilde h''(\varphi_j(s))=h''(s).
\]
In \eqref{ct:eq:variation}, the two signs in $T\cdot h'$ cancel and
the second-derivative term is unchanged. Change of variables has absolute
Jacobian one, so
\[
 0=m\int_W\{2k\cdot h''+(\sigma-3|k|^2)T\cdot h'\}\,\dd s.
\]
Division by $m$ proves vanishing on $W$; pulling back any test on a
partner interval gives the same conclusion there.
Theorem~\ref{ct:thm:freearcs} gives the remaining assertions.
\end{proof}

\begin{lemma}\label{ct:lem:selection}
Suppose finitely many relatively closed sets $Z_1,\ldots,Z_N$ cover an
open interval $V$. Some nonempty open $W\Subset V$ has this property:
each $Z_j$ either contains $W$ or is disjoint from it, and at least one
contains it.
\end{lemma}
\begin{proof}
Keep a nonempty open interval and process the sets successively. If the
next closed set contains an open subinterval of the current interval,
retain one with compact closure inside it. Otherwise its open complement
is dense in the current interval; retain an interval with compact closure
inside that complement. All earlier decisions survive restriction.
After finitely many steps every set has one of the stated relationships
to the retained interval. The cover excludes the possibility that all
are disjoint from it.
\end{proof}

\begin{theorem}\label{ct:thm:nowheredense}
The closed set $R_u$ has empty interior, and hence is nowhere dense.
More precisely, every open interval in $S_u$ contains a nonempty open
interval belonging to an exact overlap bundle of fixed multiplicity
$m\ge2$, and the reaction vanishes on every copy in that bundle.
\end{theorem}
\begin{proof}
Let $I\subset S_u$ be open, choose $s_0\in I$, and list the finite fiber
$u^{-1}(u(s_0))=\{s_1=s_0,s_2,\ldots,s_k\}$. Choose disjoint closed
intervals $V_i'$ centered at the $s_i$, with interiors $V_i$, all shorter
than the uniform injectivity scale of Lemma~\ref{ct:lem:fibers}; also
require $V_1'\subset I$. For example their radii may be taken smaller
than one quarter of the minimum pairwise parameter separation and
injectivity scale, with the first radius additionally smaller than
$\dist_\T(s_0,\T\setminus I)/2$.
The compact set $u(\T\setminus\bigcup_iV_i)$ omits $u(s_0)$, so let
$d>0$ be its distance from that point. Choose an interval
$V\Subset V_1$ about $s_0$ with $u(\overline V)\subset B_{d/2}(u(s_0))$.
Then every preimage of every $u(s)$, $s\in V$, belongs to one of the $V_i$.

For $i\ge2$ set $Z_i=\{s\in V:u(s)\in u(V_i')\}$. These are relatively
closed, since $u(V_i')$ is compact. They cover $V$: a second visit exists
because $V\subset S_u$, and injectivity on $V_1'$ excludes its being in
that first interval. Lemma~\ref{ct:lem:selection} provides
$W\Subset V$ where precisely a fixed nonempty set $A$ of these indices
has a partner at every parameter, and the remaining indices have none.

For $i\in A$, let $\varphi_i(s)$ be the unique point in $V_i'$ with
$u(\varphi_i(s))=u(s)$. It actually lies in $V_i$, by the choice of $V$.
It is continuous: for any $s_n\to s$, a convergent subsequence of the
partners has image $u(s)$, hence must be the unique partner of $s$;
compactness then forces convergence of the whole sequence. It is
injective, since $u$ is injective on $W$, and a continuous injective
function between real intervals is strictly monotone. For $s<t$,
$\varphi_i([s,t])$ is the interval between its endpoint values. The
curve traverses the same embedded arc on the two intervals. Unit speed
makes its length both $t-s$ and
$|\varphi_i(t)-\varphi_i(s)|$. The fixed monotonicity sign therefore
gives $\varphi_i(s)=\varepsilon_i s+b_i$. Its image is an open interval
in $V_i$. These images, together with $W$, are disjoint, and the
construction excluded every additional preimage. This proves
\eqref{ct:eq:bundle} with $m=1+|A|$. The preceding proposition makes
all these intervals free.

If $R_u$ had interior, choose an open interval inside it. It is contained
in $S_u$ by Theorem~\ref{ct:thm:freearcs}, so the construction supplies
an open subinterval where the reaction vanishes. Such a subinterval is
disjoint from its support, a contradiction. Distributional support is
closed by definition; a closed set with empty interior is nowhere dense.
\end{proof}
\begin{remark}\label{ct:rem:supportscope}
For a multiply traversed round circle, $S_u=\T$ but $R_u$ is empty.
The theorem allows contact intervals and excludes intervals filled with
reaction support. A closed nowhere-dense set may still have positive
Lebesgue measure and infinitely many accumulation points. Neither
finiteness of $S_u$ or $R_u$, nor a measure representation of the reaction,
has been proved here.
\end{remark}

\section{When the bridge inequality is saturated}\label{sa:sec:saturation}
We now add a numerical hypothesis. The bridge index
$a=\operatorname{bridge}(K)$ is the minimum number of local maxima of
a Morse height function, minimizing over smooth representatives of the
knot class and height directions. A Morse height function has nonzero
second derivative at every critical point, so on the compact parameter
circle these points are finite. The braid index
$\operatorname{braid}(K)$ is the least strand number of a closed braid
representing $K$. For such a braid, angular projection about an axis
disjoint from the curve has nonzero derivative; after orientation it
is a covering of the angular circle, and its degree is the strand
number. These conventions, including the name BB when the two indices
agree, are those of \cite[Section 1]{DER}.
Assume for this section that
\begin{equation}\label{sa:eq:sharp}
 e_K=(2\pi a)^2.
\end{equation}
This is an additional assumption about the embedded bending infimum.

\begin{proposition}\label{sa:prop:saturation}
A normalized strong limit of a bending-minimizing sequence under
\eqref{sa:eq:sharp} satisfies
\[
 \TC(u)=2\pi a,\qquad |u''|=2\pi a\quad\hbox{a.e.},\qquad u\in C^{1,1}.
\]
The conclusion applies to elastic limits and to normalized minimizers in
the regular $H^2$ part of the $C^1$ closure.
\end{proposition}
\begin{proof}
For a smooth embedded representative $v$ of $K$, let
$\mu(v,\omega)$ count the local maxima of the height $v\cdot\omega$.
To justify the almost-everywhere Morse assertion, consider
\[
 \mathcal N_v=\{(s,\omega)\in\T\times S^2:v'(s)\cdot\omega=0\}.
\]
Since $v'$ never vanishes, variation of $\omega$ in its tangent plane
shows that $\mathcal N_v$ is a smooth two-dimensional surface.
Its tangent vectors $(\alpha,\xi)$ satisfy
\[
 (v''(s)\cdot\omega)\alpha+v'(s)\cdot\xi=0,
 \qquad \xi\in T_\omega S^2.
\]
The projection $(s,\omega)\mapsto\omega$ has surjective differential
exactly when $v''(s)\cdot\omega\ne0$. Its critical values are therefore
precisely the directions admitting a degenerate height-critical point.
Sard's theorem \cite{Sard1942} makes these values a spherical null set.
For every other direction all critical points are nondegenerate and,
by compactness, finite: the height is Morse. This argument requires
no positive-curvature assumption. The definition of bridge index now
gives $\mu(v,\omega)\ge a$ for almost every $\omega$. Milnor's integral identity
\cite[Theorem 3.1]{Milnor1950}, recalled in
\eqref{sa:eq:milnor}, therefore gives
\[
 2\TC(v)=\int_{S^2}\mu(v,\omega)\,\dd A(\omega)
 \ge a\,\operatorname{area}(S^2)=4\pi a.
\]
The same bound holds for regular $H^2$ embeddings
\cite[Lemma 2.3]{RvM2026}. Indeed, smooth periodic approximations
$w_n\to v$ in $H^2$ converge in $C^1$, retain positive speed and,
by isotopy stability, eventually represent $K$. Put
$\tau_n=w_n'/|w_n'|$ and
$A_n=(\operatorname{Id}-\tau_n\otimes\tau_n)/|w_n'|$,
with $A$ defined from $v$ in the same way. The arbitrary-parameter
curvature formula is $\TC(w_n)=\int_\T|A_nw_n''|$.
Since $A_n\to A$ uniformly,
\[
 |\TC(w_n)-\TC(v)|
 \le\|A_n\|_\infty\|w_n''-v''\|_1
       +\|A_n-A\|_\infty\|v''\|_1\longrightarrow0.
\]
Thus $\TC(v)\ge2\pi a$ also in this regularity. Now let
$v_j\to u$ strongly in $H^2$ be the minimizing sequence.
Since these base curves are unit speed,
\[
 |\TC(v_j)-\TC(u)|\le\|v_j''-u''\|_1\le\|v_j''-u''\|_2\to0.
\]
Theorem~\ref{thm:no-loss} and the unit measure of $\T$ now give
\[
 2\pi a\le\int_\T|u''|\,\dd s
 \le\left(\int_\T|u''|^2\,\dd s\right)^{1/2}=2\pi a.
\]
All inequalities are equalities. The equality case of Cauchy--Schwarz
for $|u''|$ and the constant function $1$ makes $|u''|$ constant almost
everywhere, and its integral gives the value. Integrating its bounded
weak derivative makes $u'$ Lipschitz. Elastic limits are covered by
Theorem~\ref{thm:no-loss}. For a minimizer in the full regular $C^1$
closure, Theorem~\ref{rc:theorem} supplies the strong embedded
recovery sequence and equality of its energy with $e_K$, after which
the same calculation applies.
\end{proof}

Write $c=2\pi a$. At this unit-length constant-curvature base curve,
\eqref{ct:eq:variation} becomes
\begin{equation}\label{sa:eq:variation}
 D\E(u)[h]=2\int_\T k\cdot h''\,\dd s
       -2c^2\int_\T T\cdot h'\,\dd s,
 \qquad f:=D\E(u)=2(D^2k+c^2k\,\dd s).
\end{equation}
The last expression means equality of parameter distributions. The
symbol $k\,\dd s$ records the distribution with Lebesgue density $k$.
Again only the base curve is unit speed.

\begin{corollary}\label{sa:cor:circular}
Every free interval of a minimizer satisfying \eqref{sa:eq:sharp} is an
arc of a circle of radius $1/c$. Thus the open dense free locus consists
of circular strands with one common radius.
\end{corollary}
\begin{proof}
On a free interval, $k$ is analytic and $k''+c^2k=0$ by
\eqref{sa:eq:variation}. Hence $A=k'+c^2T$ is constant.
Differentiating $|k|^2=c^2$ and $T\cdot k=0$ gives
$k\cdot k'=0$ and $T\cdot k'=-c^2$, so $A\cdot k=A\cdot T=0$.
Differentiate the first scalar equality and substitute
$k'=A-c^2T$: it yields $0=A\cdot k'=|A|^2$. Therefore
$k'=-c^2T$. The vectors $T,k$ evolve within their initial two-dimensional
span, and $(u+k/c^2)'=0$. Thus the curve lies in an affine plane at
distance $1/c$ from a fixed center and traverses that circle at unit
speed. The nowhere-dense support theorem gives density of these
intervals. It does not identify their planes or centers with one another.
\end{proof}
If the length multiplier is defined by
$D\B+\lambda_{\rm len}D\LL=0$, the present multiplier is
$\lambda_{\rm len}=\B(u)=c^2$. Accordingly the constant-curvature
free equation is $2k''+(3c^2-\lambda_{\rm len})k=0$; retaining this
length term is essential for the radius just obtained.

For $a=2$, the tangential-circle classification
\cite[Corollary 3.4]{GRM} already shows why saturation need not force
a covered circle. In the trefoil class, the planar tangent figure-eight
of two circles of radius $1/(4\pi)$ belongs to the $C^1$ closure
\cite[Remark 6.4]{GRM}. Its energy is $(4\pi)^2$, so it is a closure
minimizer; Theorem~\ref{rc:theorem} gives strong $H^2$ approximants.
Thus the selection by ropelength in the next result is additional
information beyond saturation.

\subsection{BB elastic limits}
\begin{corollary}[Alternative proof of the BB classification]\label{sa:cor:bb}
For a one-component tame BB knot class of bridge index $a$, every
elastic limit is an $a$-covered circle, up to Euclidean motions and
regular reparametrization.
\end{corollary}
\begin{proof}
We use two quantitative geometric inputs from \cite{RvM2026}.
For a fixed smooth $a$-strand braid representative of $K$,
\cite[Proposition 4.2]{RvM2026} supplies constants
$\rho_0,c_*,C_*>0$ and smooth normalized same-class embeddings
$v_\rho$ for every $0<\rho\le\rho_0$, with
\[
 |\B(v_\rho)-(2\pi a)^2|\le C_*\rho^2,
 \quad \Delta(v_\rho)\ge c_*\rho,
 \quad \Rop(v_\rho)\le C_*/\rho.
\]
The constants may depend on the chosen representative, but not on
$\rho$. We use this quantitative result as an external input.
The embedded bridge bound then gives $e_K=(2\pi a)^2$. If
$\gamma_\theta$ minimizes $\B+\theta\Rop$, choose
$\rho=\theta^{1/3}\le\rho_0$ for sufficiently small $\theta$. Minimality gives
\[
 0\le\B(\gamma_\theta)-(2\pi a)^2\le C\theta^{2/3},\qquad
 \theta\Rop(\gamma_\theta)\le C\theta^{2/3}.
\]
Since its length is one, its thickness
$\Delta_\theta=1/\Rop(\gamma_\theta)$ is at least
$c_0\theta^{1/3}$. This is also \cite[Proposition 5.1]{RvM2026}.
Strong convergence and constant curvature of a chosen elastic-limit
subsequence now follow from no-loss and the preceding proposition.

The second input, \cite[Lemma 6.1]{RvM2026}, states that for each fixed
$h\in C^2(\T;\R^3)$ there are constants $a_h,C_h>0$, independent of
the reference curve and its positive thickness $\Delta$, such that
$\gamma+t\varepsilon h$ is a regular embedding of the same knot type
for $|t|\le1$, $|\varepsilon|\le a_h\Delta$, and has ropelength at
most $C_h/\Delta$, provided $\gamma$ is a unit-speed embedded
$C^{1,1}$ curve of length one. The minimizers here meet these hypotheses:
their positive thickness gives $C^{1,1}$ regularity, as recorded in
Section~\ref{sec:compactness}.

Here is the mechanism behind the uniformity in $\Delta$ in that lemma.
Thickness bounds the curvature of $\gamma$ by $1/\Delta$.
Unit length and the embedded bridge bound give
$2\pi\le\TC(\gamma)\le1/\Delta$, hence $\Delta\le1/(2\pi)$.
For $\eta_t=\gamma+t\varepsilon h$, taking $a_h$ small in terms of
$\|h\|_{C^2}$ keeps the speed between $1/2$ and $3/2$ and bounds
both $\|\eta_t''\|_\infty$ and its curvature by $C_h/\Delta$.
To control self-approach, split pairs at circular parameter distance
$\alpha_h\Delta$, with a sufficiently small fixed $\alpha_h>0$.
For shorter pairs, Taylor's formula and the speed lower bound make
the chord's scalar product with its starting tangent nonzero. Such
pairs neither collide nor form a doubly critical chord, meaning one
perpendicular to both endpoint tangents. For longer pairs, minimize chord length on the compact set
$\{(s,r):\dist_\T(s,r)\ge\alpha_h\Delta\}$. At its boundary the
local thick chord--arc estimate applies, since $\alpha_h\le\pi$;
an interior minimum is a doubly critical chord of length at least
$2\Delta$. Consequently all longer pairs satisfy
\[
 |\gamma(s)-\gamma(r)|\ge2\Delta\sin(\alpha_h/2).
\]
Shrinking $a_h$ again makes the perturbation of this chord at most
$\Delta\sin(\alpha_h/2)$. Thus these chords retain length at least
$\Delta\sin(\alpha_h/2)$. The characterization of thickness as the
minimum of curvature radius and half the infimum of doubly critical
chord lengths now yields
$\Delta(\eta_t)\ge c_h\Delta$. Its length is at most
$3/2$, giving the stated ropelength bound. All estimates are uniform
for $|t|\le1$. Hence this entire path consists of regular embeddings;
local constancy of knot type along a $C^1$ path keeps it in $K$.
These are the short- and long-chord arguments of the cited lemma;
the fixed $C^2$ norm of $h$ is essential.

Fix smooth $h$ and let $\varepsilon_\theta=\sqrt\theta$.
Let $r_{\rm T}$ be a radius of a Taylor ball about the strong limit $u$
in Lemma~\ref{lem:taylor}, and let $\theta_0$ make the preceding
thickness bound valid. Restrict to
\[
 0<\theta<\min\left\{\theta_0,(a_hc_0/2)^6,
       \left(\frac{r_{\rm T}}{2(1+\|h\|_{H^2})}\right)^2\right\},
 \qquad \|\gamma_\theta-u\|_{H^2}<r_{\rm T}/2.
\]
These conditions hold eventually on the chosen subsequence. They give
$\varepsilon_\theta/\Delta_\theta\le c_0^{-1}\theta^{1/6}<a_h/2$
and keep both variation segments in the Taylor ball. Normalize the two
competitors $\gamma_\theta\pm\varepsilon_\theta h$; translation,
dilation, and arclength reparametrization preserve $\E$ and ropelength.
Their bending after normalization is their original $\E$.
Minimality and the uniform Taylor estimate imply, for both signs,
\[
 0\le\pm\varepsilon_\theta D\E(\gamma_\theta)[h]
          +C_h\varepsilon_\theta^2
          +C_h\theta/\Delta_\theta.
\]
Divide by $\varepsilon_\theta$ to obtain
\[
 |D\E(\gamma_\theta)[h]|
 \le C_h\sqrt\theta+C_hc_0^{-1}\theta^{1/6}\longrightarrow0.
\]
Continuity of $D\E$ in strong $H^2$ gives full material stationarity of
$u$. This spells out the two-sign argument of
\cite[Proposition 6.4]{RvM2026}. Thus the free circular equation holds
globally, so $u$ is a covered circle; $|k|=2\pi a$ and length one fix
the number of coverings as $a$. This is the linear classification also
proved in \cite[Lemma 7.1]{RvM2026}.
\end{proof}
The main theorem of \cite{RvM2026} assumes a bridge inequality on the
$C^1$ boundary. The quantitative inputs just used do not require that
assumption. In the present argument no-loss supplies the energy equality
for the particular limits under consideration; the bridge inequality
is applied to the embedded approximants.

\section{Rigidity when the full reaction is a measure}\label{sa:sec:measure}
The word ``full'' matters in the next statement. Its hypothesis concerns
$f$ on the parameter circle. Ambient stationarity only makes the spatial
resultant vanish, and does not by itself give a measure on that circle.
The theorem is an independent implication under its stated assumptions;
bridge saturation is needed only when applying it to elastic limits.

\begin{theorem}\label{sa:thm:measure}
Let $u:\T\to\R^3$ be closed, unit speed, and $H^2$, with
$|u''|=c$ almost everywhere. Suppose
\[
 D\E(u)[X\circ u]=0\quad(X\in C_c^\infty(\R^3;\R^3)),
\]
and suppose $f=D\E(u)$ is a finite vector Radon measure on $\T$.
Then $f=0$, $c=2\pi m$ for an integer $m\ge1$, and $u$ traverses a
circle of radius $1/c$ exactly $m$ times.
\end{theorem}
\begin{proof}
A finite vector Radon measure is a tuple of finite signed Radon measures;
its total variation $|f|$ controls integration of bounded Borel vector
fields by $|\int g\cdot\dd f|\le\int|g|\,\dd|f|$. Pushforwards and
all measure identities below are understood componentwise. A function
is BV, or of bounded variation, if its distributional derivative is a
finite measure.

\emph{Step 1: the measure hypothesis improves regularity.}
If $c=0$, $u'$ is a constant unit vector and cannot be periodic, so
$c>0$. Set $T=u'$ and $k=u''$. Equation~\eqref{sa:eq:variation} shows
that $D^2k=f/2-c^2k\,\dd s$ is a finite measure $\mu$.
Its mass is zero, because a distributional derivative paired with the
constant test $1$ is zero. A periodic distribution whose derivative is
$\mu$ has a bounded BV representative, up to an additive constant.
Here is the one-dimensional argument. Choose a cut of the circle that
is not an atom of $|\mu|$, which is possible since a finite measure has
at most countably many atoms. On the resulting interval $(0,1)$ define
$Q(s)=\mu((0,s])$. Its variation is at most $|\mu|(\T)$, and it is
bounded by that number. Fubini's theorem against a smooth periodic test
proves $DQ=\mu$: its use is justified by the bound
$\|\varphi'\|_\infty|\mu|(\T)$ for the absolute double integral.
There is no endpoint jump because the cut is not an atom and the total
mass is zero. Subtracting $Q$ from the distribution in question leaves
zero derivative, hence a constant. Apply this to $Dk$. Thus $Dk$ has a
bounded BV density, denoted $k'$. Integrating it gives a Lipschitz
representative of $k$, and $u\in C^{2,1}$. In particular $T,k$ are
continuous, $k'$ is bounded, and $Dk'$ is a finite measure.

\emph{Step 2: a scalar sign identity.}
Differentiating the absolutely continuous identities $|k|^2=c^2$ and
$T\cdot k=0$ gives
\[
 k\cdot k'=0,\qquad T\cdot k'=-c^2\quad\hbox{a.e.}
\]
The product rule with a Lipschitz continuous factor and a BV factor is
\[
 D(k\cdot k')=|k'|^2\,\dd s+k\cdot Dk'.
\]
It can be checked directly here. Replace $k$ by smooth periodic
mollifications $k_\epsilon$, which converge uniformly and have
$k_\epsilon'\to k'$ in $L^1$. In the smooth-factor product rule,
$k_\epsilon'\cdot k'$ converges in $L^1$ because $k'$ is bounded,
and $k_\epsilon\cdot Dk'$ converges as measures because
$\|k_\epsilon-k\|_\infty|Dk'|(\T)\to0$. The left sides converge
as distributions. This proves the formula. Its left side is zero,
so multiplying $f=2(Dk'+c^2k\,\dd s)$ by the continuous $k$ yields
\begin{equation}\label{sa:eq:negative-pairing}
 k\cdot f=-2(|k'|^2-c^4)\,\dd s
          =-2|k'+c^2T|^2\,\dd s\le0.
\end{equation}
For the second equality expand the square and use $T\cdot k'=-c^2$
and $|T|=1$. In particular its atomic part vanishes:
\begin{equation}\label{sa:eq:atom}
 k(s)\cdot f(\{s\})=0\quad(s\in\T).
\end{equation}

\emph{Step 3: different curvature vectors can meet at only countably
many spatial points.}
Lemma~\ref{ct:lem:fibers} supplies a separation $\eta>0$ for distinct
visits. Thus the contact-pair set
\[
 Z=\{(s,t)\in\T^2:u(s)=u(t),\ \dist_\T(s,t)\ge\eta\}
\]
is compact. At each pair choose a spatial axis nonperpendicular to both
tangents. On small disjoint parameter intervals, the two $C^{2,1}$
branches are $C^2$ graphs over this axis, written
$(x,\alpha(x))$ and $(x,\beta(x))$ with $\alpha,\beta$ taking values
in $\R^2$. Finitely many pair charts cover $Z$. In one chart put
$g=\alpha-\beta$. If a coincidence has different tangent lines,
$g'(x_0)\ne0$; a nonzero scalar component of this derivative makes
$x_0$ an isolated zero. If the tangent lines agree, the common graph
slope is $p=\alpha'(x_0)=\beta'(x_0)$. The physical curvature vector
of the first graph is
\[
 \frac{P_{(1,p)}^\perp(0,\alpha''(x_0))}{1+|p|^2}.
\]
This follows by differentiating the unit tangent with respect to
arclength; it is unchanged if the material orientation is reversed.
The map $q\mapsto P_{(1,p)}^\perp(0,q)$ is injective, because a vector
$(0,q)$ parallel to $(1,p)$ must be zero. Different curvature vectors
therefore give $g''(x_0)\ne0$. Since $g(x_0)=g'(x_0)=0$, Taylor's
formula gives $g(x_0+h)=g''(x_0)h^2/2+o(h^2)$, and a nonzero scalar
component again isolates the zero.

The isolated zeros of a function on an interval form a countable set:
assign to each a rational interval containing no other zero. Our finite
chart cover thus shows that
\[
 D=\{p\in u(\T):\text{two visits to }p\text{ have different }k\}
\]
is countable. Fibers are uniformly finite, so $u^{-1}(D)$ is countable
as well.

\emph{Step 4: pair the spatial cancellation with curvature.}
Cover $\T$ by finitely many closed injective arcs $I_1,\ldots,I_N$.
Each inverse $u(I_j)\to I_j$ is continuous. On the Borel set
$u(I_j)\setminus\bigcup_{i<j}u(I_i)$ choose the curvature from this
inverse branch. Set the resulting field equal to zero on $D$ and off
the image, calling it $K:\R^3\to\R^3$. This is a bounded Borel field,
and
\[
 K(u(s))=k(s)\quad\hbox{when }u(s)\notin D.
\]

Define the pushforward by $(u_\#f)(A)=f(u^{-1}(A))$ for Borel sets.
Its integral satisfies $\int X\cdot\dd(u_\#f)=\int X(u)\cdot\dd f$.
The distributional and measure pairings agree on $X\circ u\in H^2$:
use smooth approximation in $H^2$, which also converges uniformly.
Ambient stationarity therefore makes $u_\#f$ annihilate every smooth
compactly supported field. Convolution and a cutoff approximate any
continuous test uniformly on the compact support of this measure.
Consequently it annihilates all continuous tests, and uniqueness in
Radon measure representation gives $u_\#f=0$
\cite[Theorem 11.26]{Knapp}. In particular
\[
 0=\int K\cdot\dd(u_\#f)=\int_\T K(u(s))\cdot\dd f(s).
\]
The difference from $\int k\cdot\dd f$ is supported on the countable
set $u^{-1}(D)$. A finite measure restricted to a countable set is the
sum of its atoms; the sum here is absolutely convergent, bounded by
$c|f|(\T)$. At those points $K=0$, and each remaining contribution
$k(s)\cdot f(\{s\})$ vanishes by \eqref{sa:eq:atom}. Thus
$\int k\cdot\dd f=0$.

Integrating \eqref{sa:eq:negative-pairing} now gives
$\int|k'+c^2T|^2=0$, so $k'=-c^2T$ almost everywhere. The system
$T'=k$, $k'=-c^2T$ has the solution
\[
 T(s)=A\cos(cs)+B\sin(cs),\qquad
 |A|=|B|=1,\quad A\cdot B=0,
\]
where the scalar relations follow from $|T(0)|=1$, $|k(0)|=c$, and
$T(0)\cdot k(0)=0$. Integration gives a circle of radius $1/c$.
Periodicity in these two independent directions forces $c=2\pi m$
with integer $m\ge1$. Its equation $k''=-c^2k$ in
\eqref{sa:eq:variation} also proves $f=0$.
\end{proof}

\begin{lemma}\label{sa:lem:braid-neighborhood}
If normalized embeddings of a fixed knot class converge in $C^1$ to
an $m$-covered circle, that class has braid index at most $m$.
\end{lemma}
\begin{proof}
Use cylindrical coordinates about the line normal to the circle through
its center. Its radius is $R=1/(2\pi m)$. Orient the angular coordinate
so that the planar position $y$ and tangent $T$ of the reference curve
satisfy $|y|=R$ and $\det(y,T)=R$. For a nearby embedding $v$, denote
its planar position and parameter derivative by $\widetilde y,
\widetilde T$. If
\[
 \|v-u\|_{C^1}<\varepsilon_b,
 \qquad\varepsilon_b=R/[4(R+2)],
\]
then $|\widetilde y|\ge3R/4$ and
\[
 |\det(\widetilde y,\widetilde T)-R|
 \le\varepsilon_b(1+\varepsilon_b+R)\le R/4.
\]
Here expand the determinant difference as
$\det(\widetilde y-y,\widetilde T)+\det(y,\widetilde T-T)$
and use $|\widetilde T|\le1+\varepsilon_b$.
The angular derivative
$\det(\widetilde y,\widetilde T)/|\widetilde y|^2$ is positive.
The pointwise straight homotopy from $u$ to $v$ has planar radius at
least $R-\varepsilon_b>0$, so its angular map has the same degree $m$.
A circle map with everywhere positive derivative is locally invertible;
compactness makes it a covering, and its degree counts its $m$ sheets.
Thus $v$ is a closed $m$-strand braid about this axis. The definition
of braid index proves the conclusion.
\end{proof}

\begin{corollary}\label{sa:cor:nonmeasure}
Suppose \eqref{sa:eq:sharp} holds and $\operatorname{braid}(K)>a$.
The full reaction of every normalized bending-minimizing closure curve
and every elastic limit is not a finite vector Radon measure.
\end{corollary}
\begin{proof}
Saturation gives $|k|=2\pi a$ and ambient stationarity. A finite full
reaction would give an $a$-covered circle by
Theorem~\ref{sa:thm:measure}. Strong recovery supplies $C^1$-convergent
embeddings of type $K$; for elastic limits the original approximants
already do so. Lemma~\ref{sa:lem:braid-neighborhood} would imply
$\operatorname{braid}(K)\le a$, a contradiction.
\end{proof}
This conclusion permits dipoles or other distributional singularities;
it neither identifies their form nor determines their support. A
measure representing only an ambient resultant does not meet the full
parameter-measure hypothesis of the theorem.

\section{Directional crookedness and a regularity obstruction}
\label{sa:sec:crookedness}
For a regular closed curve, its height in direction $v\in S^2$ is
$s\mapsto u(s)\cdot v$. Write $\mu(u,v)$ for its number of local maxima
in one material period. At a Morse direction all critical points are
simple, maxima and minima alternate, and the number of critical points
is $2\mu(u,v)$. We use Milnor's integral identity
\begin{equation}\label{sa:eq:milnor}
 \int_{S^2}\mu(u,v)\,\dd A(v)=2\TC(u),
\end{equation}
where $\dd A$ is spherical area measure \cite[Theorem 3.1]{Milnor1950}.
For a $C^2$ curve parametrized by arclength, $\TC(u)=\int|u''|$
\cite[Theorem 2.2]{Milnor1950}. The next theorem has no stationarity
hypothesis.

\begin{theorem}\label{sa:thm:crookedness}
Let $u:\R/L\mathbb Z\to\R^3$ be closed, unit speed, and $W^{3,1}$.
Suppose its continuous curvature $\kappa=|u''|$ is positive everywhere
and $\mu(u,v)=a$ for almost every $v\in S^2$, with an integer $a$.
Then $u$ lies in a plane and has signed curvature of one sign there.
If $\kappa=c$ is constant, it is an $a$-covered circle and $cL=2\pi a$.
\end{theorem}
\begin{proof}
We prove that a nonconstant binormal would create two open sets of
directions with different numbers of maxima. The argument uses only
absolute continuity of the frame, not a twice differentiable binormal.

\emph{Step 1: the frame and its good parameters.}
The one-dimensional Sobolev theorem makes $u''$ absolutely continuous
\cite[Theorem 8.2]{Brezis}. Set
\[
 T=u',\qquad N=u''/\kappa,\qquad B=T\times N.
\]
Since the continuous positive $\kappa$ has a positive minimum, the
chain and product rules for absolutely continuous functions show that
$\kappa,N,B\in W^{1,1}$, while $T\in C^1$. Differentiating their scalar
products gives
$N'\cdot N=0$, $N'\cdot T=-\kappa$, and
$B'\cdot B=B'\cdot T=0$. With $\tau=N'\cdot B\in L^1$, these give
\[
 T'=\kappa N,\qquad N'=-\kappa T+\tau B,\qquad B'=-\tau N
 \quad\hbox{a.e.}
\]
Let $G$ be the set where $N$ is differentiable, these identities hold,
and $\tau$ has its Lebesgue value:
\[
 (2r)^{-1}\int_{s-r}^{s+r}|\tau(t)-\tau(s)|\,\dd t\to0.
\]
Absolute continuity gives differentiability almost everywhere; the
Lebesgue differentiation theorem gives the last property almost
everywhere \cite[Theorems 4.19 and 4.10]{Axler}. Applying these results
componentwise and intersecting finitely many full-measure sets shows
that $G$ has full measure in the parameter circle.

\emph{Step 2: a second-order normal expansion at good visits.}
For $\epsilon=\pm1$ put $q_\epsilon=\epsilon B$. Fix
$s_0\in G$ with $\tau_0=\tau(s_0)\ne0$, use the local real parameter
$t=s-s_0$, and write a subscript $0$ for values at $s_0$.
Differentiability gives
$N(s_0+t)=N_0+N'_0t+o(|t|)$, with $N'_0\perp N_0$.
The Lebesgue-point property gives, for both signs of $t$,
\[
 \int_0^t\tau(s_0+r)r\,\dd r=\tau_0t^2/2+o(t^2),\qquad
 \int_{\min(0,t)}^{\max(0,t)}|\tau(s_0+r)||r|\,\dd r=O(t^2).
\]
Indeed the first error is at most
$|t|\int_{-|t|}^{|t|}|\tau(s_0+r)-\tau_0|\,\dd r=o(t^2)$;
the second estimate follows by adding the constant $|\tau_0|$.
Integrate $q_\epsilon'=-\epsilon\tau N$ and project perpendicular to
$N_0$. The term $N_0$ disappears and the differentiability remainder
is bounded by $\sup_{|r|\le|t|}o(|r|)/|r|$ times the second integral.
Hence
\[
 P_{N_0}^\perp(q_\epsilon(s_0+t)-q_0)
           =-\epsilon\tau_0N'_0t^2/2+o(t^2).
\]
The tangent coordinate
$\zeta=N_0\cdot(q_\epsilon(s_0+t)-q_0)$ satisfies
$\zeta=-\epsilon\tau_0t+o(|t|)$. In particular it is comparable in
absolute value to $t$, without any claim that it is monotone. Eliminating
$t$ gives the normal Peano expansion
\begin{equation}\label{sa:eq:peano}
 P_{N_0}^\perp(q_\epsilon-q_0)
       =\tfrac12 A\zeta^2+o(\zeta^2),\qquad
 A=-\frac{\epsilon N'_0}{\tau_0}=-q_0+K,
 \qquad K=\frac{\epsilon\kappa_0}{\tau_0}T_0\ne0.
\end{equation}
Changing the sign of the chosen tangent axis changes $\zeta$ to
$-\zeta$ and leaves the vector $A$ unchanged.

\emph{Step 3: exceptional images have zero length.}
For an absolutely continuous curve $q$ and a measurable parameter set
$E$, one has
\[
 \mathcal H^1(q(E))\le\int_E|q'|\,\dd s,
\]
where $\mathcal H^1$ is one-dimensional Hausdorff measure, normalized
so that a segment has its ordinary length. To verify the estimate,
absolute continuity gives
$\operatorname{diam}q(J)\le\int_J|q'|$ for an interval $J$.
Partitioning $J$ into intervals of mesh tending to zero produces
Hausdorff covers of vanishing diameter, by uniform continuity of $q$,
and total diameter at most $\int_J|q'|$. For general $E$, outer
regularity of the finite measure $|q'|\,\dd s$ gives an open $O\supset E$
with integral at most $\int_E|q'|+\eta$. Decompose $O$ into countably
many intervals, apply the preceding bound, and let $\eta\downarrow0$.
A cut of the parameter circle adds only finitely many image points.
Apply this to both $q_\epsilon$. The images of parameters outside $G$
or with $\tau=0$ have zero $\mathcal H^1$ measure; denote their union
by $C$.

\emph{Step 4: unequal incidence data are exceptional too.}
Regard $q$ as the map on the disjoint union of two parameter circles
whose restrictions are $q_+,q_-$. Its incidence set
$Z=\{(s,t):q(s)=q(t)\}$ is closed in the product of this compact space
with itself. At good visits with $\tau\ne0$, $q'$ is nonzero.
Two visits at the same point with different tangent lines determine
an isolated incidence pair. In fact let their derivatives be $v_1,v_2$.
Their independence gives
\[
 \gamma=\min_{h^2+t^2=1}|v_1h-v_2t|>0.
\]
Choose one radius so small that each first-order remainder is at most
$\gamma|r|/(2\sqrt2)$ within it. For increments $h,t$ in that radius,
\[
 |q(s_1+h)-q(s_2+t)|
 \ge\gamma\sqrt{h^2+t^2}
          -\frac\gamma{2\sqrt2}(|h|+|t|)>0
\]
unless $h=t=0$.

If the tangent lines agree, choose a common unit axis $e$ on that line.
Its tangent coordinates satisfy
$|\zeta_i(r)|\ge|e\cdot q'(s_i)|\,|r|/2$ at sufficiently small radius.
Suppose their normal coefficients $A_1,A_2$ in
\eqref{sa:eq:peano} differ. Shrink that radius so that each normal
remainder has norm at most $|A_1-A_2||\zeta_i|^2/8$.
At a common image the coordinates are equal to one $\zeta$; subtraction
of the normal expansions gives a difference of norm at least
$|A_1-A_2||\zeta|^2/4$. It cannot vanish for $\zeta\ne0$.
For $\zeta=0$ the preceding tangent-coordinate bound forces both
increments to be zero. This isolates the pair in the full incidence
set, including nearby bad parameters. A second-countable space has only
countably many isolated points: choose a basic open neighborhood meeting
the set only at that point. Thus the common images of all pairs with
unequal tangent lines or unequal normal coefficients form a countable
set $J$.

\emph{Step 5: choose a direction where all degeneracies agree.}
Suppose $B$ is nonconstant. Its connected image has positive
$\mathcal H^1$ measure: projection on the line joining two different
image points covers the intervening interval and cannot increase
Hausdorff length. Choose $v_0\in B(\R/L\mathbb Z)\setminus(C\cup J)$.
Every visit of either $q_\epsilon$ to $v_0$ is good with nonzero
$\tau$, and all such visits have the same tangent line and coefficient
$A$. Thus they have the same nonzero vector $K=A+v_0$ in
\eqref{sa:eq:peano}; in particular $K\perp v_0$.
Each visit is isolated in the fiber, because its nonzero derivative gives
$|q(s+t)-q(s)|\ge|q'(s)||t|/2$ for small nonzero $t$.
The fiber is compact and all its points are isolated, so it is finite.

Set $F(s,v)=T(s)\cdot v$. Then
$F_s(s,v)=\kappa(s)N(s)\cdot v$ is continuous. Degenerate zeros of
$F(\cdot,v_0)$ are precisely those finitely many visits
$q_{\epsilon_i}(s_i)=v_0$. At such a visit,
$N'_i\cdot v_0=\epsilon_i\tau_i$ and $N_i\cdot v_0=0$.
Integrating $F_s=\kappa N\cdot v_0$, using continuity of $\kappa$ and
the first-order expansion of $N$, gives
\begin{equation}\label{sa:eq:fold}
 F(s_i+t,v_0)=a_it^2/2+o(t^2),\qquad
 a_i=\epsilon_i\kappa_i\tau_i\ne0,\qquad
 (T_i\cdot K)a_i=\kappa_i^2>0.
\end{equation}
These zeros are isolated. Every other zero is simple and is isolated by
nonvanishing of $F_s$. The compact zero set is therefore finite. There
is at least one degenerate zero because $v_0$ belongs to the binormal
image.

\emph{Step 6: compare directions on opposite sides of these folds.}
Let $r\ge1$ be the number of degenerate zeros and $n_0$ the number of
simple zeros just found. Let $d_*$ be their minimum pairwise circular
separation, using $d_*=L$ if there is only one zero. For each fold choose
an interval $I_i=[s_i-r_i,s_i+r_i]$, with
$r_i<\min\{L/8,d_*/4\}$, so small that, with
$\sigma_i=\operatorname{sign}a_i$ and $K_i=T_i\cdot K$,
\[
 \sigma_iF(s,v_0)\ge d_i(s-s_i)^2,
 \qquad \sigma_iT(s)\cdot K\ge b_i>0\quad(s\in I_i),
 \qquad d_i=|a_i|/4,\quad b_i=|K_i|/2.
\]
The Peano expansion and continuity give these inequalities. Put
$m_i=\min_{\partial I_i}\sigma_iF(s,v_0)>0$.
Around each simple zero $z_j$ choose
$J_j=[z_j-\rho_j,z_j+\rho_j]$, with
$\rho_j<\min\{L/8,d_*/4\}$, so that
\[
 |F_s(s,v_0)-F_s(z_j,v_0)|\le|F_s(z_j,v_0)|/2
 \quad(s\in J_j).
\]
All the intervals are disjoint. The compact complement $C_0$ of their
interiors has no zero; set
$\eta_0=\min_{C_0}|F(s,v_0)|>0$ and
$m_j^0=\min_{\partial J_j}|F(s,v_0)|>0$.
With $\kappa_{\max}=\max\kappa$, choose
\[
 \varepsilon_0=\min\left\{1,\eta_0/2,
       \min_j m_j^0/2,
       \min_j\frac{|F_s(z_j,v_0)|}{4\kappa_{\max}}\right\}>0,
\]
omitting minima over empty lists. If $|v-v_0|<\varepsilon_0$,
$C_0$ has no zero and each $J_j$ has exactly one simple zero. Indeed
$|F(s,v)-F(s,v_0)|\le|v-v_0|$ preserves endpoint signs, and
$|F_s(s,v)-F_s(s,v_0)|\le\kappa_{\max}|v-v_0|$ preserves the strict
sign of $F_s$ on $J_j$.

Let $M_K=\max|T\cdot K|$ and take
\[
 \delta_*=\min\left\{\frac1{2(1+|K|)},
       \frac{\varepsilon_0}{4(1+|K|)},
       \min_i\frac{m_i}{2(1+M_K)}\right\}>0.
\]
For $|\delta|\le\delta_*$ the unit directions
$v_\delta=(v_0+\delta K)/|v_0+\delta K|$ satisfy
$|v_\delta-v_0|\le2|\delta||K|$, and
\[
 F(s,v_\delta)=
 \frac{F(s,v_0)+\delta T(s)\cdot K}{|v_0+\delta K|}.
\]
For $\delta>0$, its numerator multiplied by $\sigma_i$ is at least
$d_i(s-s_i)^2+\delta b_i>0$ throughout $I_i$, so every fold disappears.
For $\delta<0$, the center has sign $-\sigma_i$ and both endpoints
retain sign $\sigma_i$, since $|\delta|M_K<m_i/2$. There are then at
least two zeros, one on each side of the center.

In particular $v_+=v_{\delta_*/2}$ is Morse with exactly $n_0$ zeros.
For $v_-=v_{-\delta_*/2}$ the finitely many center and endpoint values
on the $I_i$ have a positive minimum absolute value $m_-$. Choose
$0<\varepsilon_-<\min\{\varepsilon_0/2,m_-/2\}$. Any direction in the
spherical ball $|v-v_-|<\varepsilon_-$ preserves these signs and stays
within $\varepsilon_0$ of $v_0$. There is a Morse direction in this
ball. To justify this, the absolutely continuous binormal images have
finite length by Step 3. A finite-length set has zero spherical area:
covers of maximum diameter $\eta$ and bounded total diameter have
sum of squared diameters tending to zero as $\eta\downarrow0$; local
smooth sphere charts compare this with spherical area. Thus
$B(\R/L\mathbb Z)\cup-B(\R/L\mathbb Z)$ has area zero and cannot fill
a nonempty spherical ball. Outside these images a simultaneous zero of
$F$ and $F_s$ is impossible, since $v$ would have to be $\pm B$.
At the chosen Morse direction there are at least $n_0+2r$ zeros.

Finally the number of zeros is locally constant near every Morse
direction: isolate its finitely many zeros in disjoint intervals on
which $F_s$ has strict sign, and use a positive lower bound for $|F|$
on the compact complement. These estimates persist for nearby directions
exactly as above. Each such open set has positive area, so the hypothesis
$\mu=a$ almost everywhere forces every Morse direction to have precisely
$2a$ zeros. The two counts $n_0$ and at least $n_0+2r$ contradict this.
Therefore $B$ is constant.

\emph{Step 7: identify the planar curve.}
Since $(u\cdot B)'=T\cdot B=0$, $u$ lies in a plane normal to $B$.
In the orientation with normal $B$, its signed curvature is the positive
$\kappa$. Lift its tangent angle to a function $\alpha$ on a real
period; the planar identity $T'=\kappa N$ gives $\alpha'=\kappa>0$.
Periodicity of $T$ makes $\alpha(L)-\alpha(0)=2\pi m$ for an integer
$m\ge1$. Every direction with nonzero planar component has exactly
$m$ maxima, because its height derivative is a cosine of this strictly
increasing angle. Thus $m=a$. If $\kappa=c$, the frame satisfies
$N'=-cT$ and $(u+N/c)'=0$, so it is a circle covered $a$ times and
$cL=2\pi a$.
\end{proof}

\begin{corollary}\label{sa:cor:regularity}
Assume \eqref{sa:eq:sharp} and $\operatorname{braid}(K)>a$.
No bending-minimizing closure curve or elastic limit is $W^{3,1}$;
in particular none is $H^3$ or $C^{2,1}$.
\end{corollary}
\begin{proof}
Suppose one such $u$ were $W^{3,1}$. Saturation makes its continuous
curvature the positive constant $2\pi a$. Choose smooth normalized
embeddings $\gamma_j$ of type $K$ converging strongly to $u$ in $H^2$.
For closure minimizers use recovery, then smooth each embedded recovery
curve with a sufficiently small error and normalize; $C^1$ stability
retains the knot type. A diagonal choice retains strong convergence.
For almost every direction the embedded bridge bound gives
$\mu(\gamma_j,v)\ge a$. Milnor's identity gives
\[
 \int_{S^2}(\mu(\gamma_j,v)-a)\,\dd A(v)
       =2\TC(\gamma_j)-4\pi a\to0.
\]
Choose a subsequence whose nonnegative $L^1$ norms are summable;
Tonelli's theorem makes the corresponding functions tend to zero
almost everywhere. Thus $\mu(\gamma_j,v)\to a$ almost everywhere.

The binormal-image argument in the preceding proof shows that almost
every direction for $u$ is Morse. At any such direction its finitely
many strict maxima persist as at least that many maxima of the
approximating heights: choose disjoint small intervals with endpoint
heights strictly smaller than the central height; uniform convergence
retains that inequality, forcing an interior maximum in each interval.
Consequently $\mu(u,v)\le\liminf_j\mu(\gamma_j,v)=a$ almost everywhere.
Its average is $a$ by \eqref{sa:eq:milnor} and saturation, so
$\mu(u,v)=a$ almost everywhere. The theorem gives an $a$-covered circle.
The approximating embeddings and Lemma~\ref{sa:lem:braid-neighborhood}
contradict the strict braid inequality. Finally on a finite parameter
interval, $L^2\subset L^1$ makes $H^3\subset W^{3,1}$, and a Lipschitz
second derivative gives $C^{2,1}\subset W^{3,1}$.
\end{proof}

Under these same knot-class hypotheses, Corollary~\ref{sa:cor:regularity}
also implies Corollary~\ref{sa:cor:nonmeasure}: Step~1 of
Theorem~\ref{sa:thm:measure} shows that a finite full reaction implies
$u\in C^{2,1}\subset W^{3,1}$. The measure theorem itself is independent of these
knot-class assumptions.

No class with both \eqref{sa:eq:sharp} and
$\operatorname{braid}(K)>a$ is exhibited here. For $a=2$ these
hypotheses are impossible: \cite[Theorem 6.3]{GRM} identifies the
nontrivial classes with $e_K=(4\pi)^2$ as $(2,b)$-torus classes,
whose braid index is two. For $a\ge3$ these remain conditional
obstructions; the existence of a class satisfying their hypotheses
is not settled here.

\section{A structure theorem with at most three planar visits}
\label{sa:sec:planar}
In the plane a unit-speed curve with constant nonzero curvature magnitude
still has a measurable choice between two curvature signs. The
two-circle case at $c=4\pi$ is described in
\cite[Corollary 3.4]{GRM}. The proof below controls those signs by
ambient variations at arbitrary constant $c$. Finiteness of
the circle geometry is a conclusion, rather than an assumption about the
switching set.

\begin{theorem}\label{sa:thm:planar}
Let $u:\T\to\R^2$ be closed, unit speed, and $H^2$, with
$|u''|=c>0$ almost everywhere. Suppose it is planar ambient stationary,
meaning $D\E(u)[X\circ u]=0$ for every
$X\in C_c^\infty(\R^2;\R^2)$, and suppose
$\#u^{-1}(p)\le3$ for every $p\in\R^2$. Then its image is a finite
union of complete circles of radius $1/c$. Each distinct circle has a
positive integer geometric multiplicity $m_j$, and
\[
 c=2\pi\sum_jm_j.
\]
There are finitely many material switches between circles. Between
switches the curve is an analytic circle arc, and its full reaction is
\begin{equation}\label{sa:eq:planar-reaction}
 D\E(u)=2\sum_{s\in S}[k]_s\delta_s',
\end{equation}
where $S$ is the finite set of material switching parameters. In
particular there are no point-force terms. Distinct circles can be
joined by a material switch only at a tangency.
\end{theorem}
\begin{proof}
We use the planar version of \eqref{sa:eq:variation}. If the curve is regarded as lying in $\R^3$,
the same tests are obtained by extension with a spatial cutoff.
Since $|u''|=c$ almost everywhere, $u\in C^{1,1}$.

\emph{Step 1: put every local visit in one graph chart.}
Fix a point $p$ of the image and list its $N\le3$ visits. Choose disjoint
injective parameter neighborhoods of them. The compact remainder of
the parameter circle has image separated from $p$. Choose a spatial
neighborhood small enough to miss this remainder. An axis $a$ having
nonzero scalar product with each visiting tangent makes each branch a
graph there, after restricting its parameter neighborhood. In oriented
coordinates with first axis $a$ and second axis $Ja$, where $J$ is the
quarter-turn, write the graphs on a common interval $I$ as
$(x,f_i(x))$. Take their domains initially larger than $I$, so that
$\overline I$ lies inside each domain. Shrink $I$ so all graph pieces
over it lie in the chosen spatial neighborhood. A smooth spatial cutoff
can then be one near all these compact pieces and supported in a
slightly larger chart containing only these branches. Multiplying test
fields by this cutoff leaves their traces and derivatives unchanged
where the tests in $x$ are supported.

Let $p_i=f_i'$. These are Lipschitz, and constant curvature magnitude
gives measurable signs $\sigma_i(x)\in\{-1,1\}$ such that
\begin{equation}\label{sa:eq:graph-curvature}
 p_i'=\sigma_i c(1+p_i^2)^{3/2}\quad\hbox{a.e. on }I.
\end{equation}
The signs refer to increasing graph coordinate, independent of material
orientation. No continuity or bounded variation of them is assumed.

\emph{Step 2: derive two regular ambient stress moments.}
Define smooth vector functions of a real slope by
\[
 F(p)=\frac{(-p,1)}{1+p^2},\qquad
 V(p)=\frac{(1-p^2,2p)}{\sqrt{1+p^2}}.
\]
For a smooth scalar $q$, a scalar test $\chi\in C_c^\infty(I)$, and
$v\in\R^2$, use $X(x,y)=\chi(x)q(y)v$ with the chart cutoff from
Step 1. Write $Q_i=q(f_i)$ and $w_i=\sqrt{1+p_i^2}$.
On branch $i$, direct arclength differentiation gives
\[
 \begin{gathered}
 x_s=\pm w_i^{-1},\qquad x_{ss}=-p_ip_i'/w_i^4,\qquad
 k=\sigma_ic(-p_i,1)/w_i,\\
 T x_s=(1,p_i)/w_i^2,\qquad\dd s=w_i\,\dd x.
 \end{gathered}
\]
The measure change uses the positive Jacobian; reversing a branch
changes $x_s,T$ together and changes none of these other quantities.
For $h=X\circ u=(\chi Q_i)v$,
\[
 \begin{split}
 h'&=(\chi'Q_i+\chi Q_i')x_s v,\\
 h''&=(\chi''Q_i+2\chi'Q_i'+\chi Q_i'')x_s^2v
            +(\chi'Q_i+\chi Q_i')x_{ss}v,\\
 Q_i'&=q'(f_i)p_i,\qquad
 Q_i''=q''(f_i)p_i^2+q'(f_i)p_i'.
 \end{split}
\]
Insert these expressions in \eqref{sa:eq:variation} and divide by
$2c$. The coefficient of the first bracket in $h''$ becomes
$\sigma_iF(p_i)$. The other coefficient, combining the $x_{ss}$ term
with the first-derivative term of the variation, is
\[
 -\frac{\sigma_ip_ip_i'(-p_i,1)}{w_i^4}
                  -\frac{c(1,p_i)}{w_i}
 =-\frac c{w_i}\{p_i(-p_i,1)+(1,p_i)\}=-cV(p_i),
\]
where \eqref{sa:eq:graph-curvature} was used without differentiating
$\sigma_i$. Thus stationarity for every $v$ says
\[
 0=\int_I(S_q\chi''+B_q\chi'+C_q\chi)\,\dd x,
\]
where
\[
 \begin{split}
 S_q&=\sum_i\sigma_iF(p_i)Q_i,\\
 B_q&=\sum_i\{2\sigma_iF(p_i)Q_i'-cV(p_i)Q_i\},\\
 C_q&=\sum_i\{\sigma_iF(p_i)Q_i''-cV(p_i)Q_i'\}.
 \end{split}
\]
All quantities are locally bounded: $p_i,p_i'$ are bounded and the
values of $q,q',q''$ stay in a compact range. Distributionally
$D^2S_q-DB_q+C_q=0$, so $D(DS_q-B_q)=-C_q$.
Integrating the bounded right side shows that $DS_q-B_q$ is a
Lipschitz function plus a constant, by the zero-derivative distribution
argument. Hence $DS_q\in L^\infty_{\rm loc}$ and $S_q$ has a locally
Lipschitz representative. In particular this holds for the two moments
\begin{equation}\label{sa:eq:two-moments}
 S_0=\sum_i\sigma_iF(p_i),\qquad
 S_1=\sum_i\sigma_if_iF(p_i).
\end{equation}
For $q(y)=1$ and $q(y)=y$, multiply $q$ by a scalar cutoff equal to one
near the compact graph-value range if needed to make the field compactly
supported.

\emph{Step 3: the first moment fixes the group sign sums.}
The graph axis can be chosen more carefully at the outset. At the
central point $p$ put $P_i=T_i\otimes T_i$, the projector onto the
unoriented tangent line of visit $i$. Distinct planar rank-one
trace-one projectors, in number at most three, are linearly independent.
Indeed a projector onto angle $\theta$ has traceless coordinates
$\tfrac12(\cos2\theta,\sin2\theta)$. Three distinct such points on a
circle are not collinear, and taking the trace of a linear relation
makes its coefficients sum to zero. This proves the claim, with the
one- and two-projector cases immediate from the same argument.
Besides excluding the finitely many axes perpendicular to a tangent,
choose $a$ to exclude the kernels
\[
 \left(\sum_i\varepsilon_iP_i\right)a=0
 \quad\text{for all nonzero matrices }\sum_i\varepsilon_iP_i,
 \qquad\varepsilon_i\in\{-1,0,1\}.
\]
Each is a proper linear subspace, and finitely many such subspaces
cannot cover the unit circle, so this choice exists.

Let $x_0$ be the central graph coordinate. For each fixed sign pattern
$\sigma\in\{-1,1\}^N$ define the continuous candidate
$A_\sigma(x)=\sum_i\sigma_iF(p_i(x))$. In physical coordinates
$F(p_i(x_0))=JP_i a$. By the axis choice and independence of the
distinct projectors,
\[
 A_\sigma(x_0)=A_\eta(x_0)
 \quad\Longleftrightarrow\quad
 \sum_{i\in G}\sigma_i=\sum_{i\in G}\eta_i
 \text{ for every equal-tangent group }G\text{ at }x_0.
\]
To check the use of the kernel condition, divide the matrix difference
by two; its coefficients then belong to $\{-1,0,1\}$. Independence
after collecting equal projectors identifies exactly the group sums.

The continuous $S_0(x_0)$ equals some candidate value: approach $x_0$
through points of the full-measure set where its defining formula holds,
and extract a subsequence with one constant sign pattern. Thus it fixes
a list of group sums. Exclude patterns with different group sums on a
whole smaller neighborhood as follows. If any such patterns exist, put
$d=\min|A_\sigma(x_0)-S_0(x_0)|>0$ over them. Choose
$0<r<\dist(x_0,\partial I)/4$ so small that for $|x-x_0|\le r$
\[
 |S_0(x)-S_0(x_0)|+
           \max_\sigma|A_\sigma(x)-A_\sigma(x_0)|<d/2.
\]
Then none of those patterns can equal $S_0(x)$, by the triangle
inequality. If there are none to exclude, choose any smaller interval.
From now on use this smaller interval. Almost everywhere the actual
pattern belongs to the surviving finite list. Since $N\le3$, either
that list has one pattern, or two surviving patterns differ by exchanging
one plus sign and one minus sign within one equal-tangent group. In a
group of two there is one of each; in a group of three there is one
minority sign. There cannot be two different nonsingleton groups when
$N\le3$.

\emph{Step 4: both moments turn exchanges into permutations of centers.}
For a fixed pattern let the candidate second moment be
$\sum_i\sigma_if_iF(p_i)$. If two surviving patterns realize the same
pair $(S_0(x),S_1(x))$, their exchanged labels $i,j$ satisfy
\[
 F(p_i(x))=F(p_j(x)),\qquad
 f_i(x)F(p_i(x))=f_j(x)F(p_j(x)).
\]
The second coordinate of $F$ is positive and its first-to-second ratio
is $-p$. Hence $F$ is injective, giving $p_i=p_j$ and then $f_i=f_j$.
This is equality of graph value and slope at the point; it uses no
isolation of contact.

For each surviving fixed pattern define its candidate centers
\begin{equation}\label{sa:eq:centers}
 C_i^\sigma(x)=(x,f_i(x))+
       \frac{\sigma_i(-p_i(x),1)}{c\sqrt{1+p_i(x)^2}}.
\end{equation}
These are locally Lipschitz functions. Identify the plane with $\mathbb C$
and let $P_\sigma(x)$ be the vector of coefficients, except the leading
coefficient one, of the monic polynomial
$\prod_{i=1}^N(z-C_i^\sigma(x))$. Repeated roots retain multiplicity.
When two patterns realize the same moment pair, Step 4 says their
exchanged graph jets agree. Their two centers are therefore just
interchanged: at the common point and slope the two possibilities are
$z_0+n/c$ and $z_0-n/c$, with
$n=(-p,1)/\sqrt{1+p^2}$. Thus their coefficient vectors agree.

At every $x$ at least one surviving pattern realizes both continuous
moments, by the full-measure subsequence argument used for $S_0$.
Define $P(x)$ using any such pattern. It is well-defined by the exchange
argument. It is continuous: from any convergent sequence of $x$ values
extract a constant-pattern subsequence. Continuity of that pattern's
moments and coefficients gives the asserted limit; the finite list then
makes this true for the whole sequence.

The continuous finite selection $P$ is locally Lipschitz. Here is the
selection fact needed to justify this. If a continuous map $g$ on an
interval always equals one of finitely many $L$-Lipschitz maps $g_j$,
then near any point $x$, candidates with $g_j(x)\ne g(x)$ are excluded
by continuity. Every remaining candidate agrees with $g$ at $x$, so
$|g(x+h)-g(x)|\le L|h|$ for small $h$. For fixed $a$, continue the
estimate $|g(t)-g(a)|\le L(t-a)$ to the right: it holds initially;
at the supremum of its interval of validity continuity retains it and
the local estimate at that supremum extends it. Thus it holds to the
endpoint. Apply it from any initial $a$, in both directions, to obtain
the full Lipschitz bound. On a compact subinterval our finitely many
coefficient candidates have a common Lipschitz constant, so this fact
applies to $P$.

\emph{Step 5: the center multiset is constant.}
Partition a full-measure subset of the interval according to the actual
sign pattern. On the set with pattern $\sigma$, differentiating the
fixed candidate \eqref{sa:eq:centers} is legitimate almost everywhere.
Using \eqref{sa:eq:graph-curvature} gives its two coordinates explicitly:
\[
 (C_i^\sigma)'_x=1-
       \frac{\sigma_ip_i'}{c(1+p_i^2)^{3/2}}=0,\qquad
 (C_i^\sigma)'_y=p_i-
       \frac{\sigma_ip_ip_i'}{c(1+p_i^2)^{3/2}}=0.
\]
The product rule for the polynomial coefficients therefore gives
$P_\sigma'=0$ almost everywhere on that set. Since $P=P_\sigma$
there, their weak derivatives agree almost everywhere there. The locality
fact used in this last step has a short proof: for the absolutely
continuous difference, at every nonisolated zero where a classical
derivative exists, difference quotients along other zeros force that
derivative to be zero. Isolated zeros are countable, and absolutely
continuous functions are differentiable almost everywhere
\cite[Theorem 4.19]{Axler}. Hence the derivative vanishes almost
everywhere on its zero set. The finitely many active sets cover almost
all parameters, so $P'=0$ almost everywhere. Absolute continuity now
makes $P$ constant.

Its fixed polynomial determines a finite multiset of centers. At almost
every $x$ each branch lies on the radius-$1/c$ circle about its active
center, and the multiset counts each center the same number of times.
Continuity and closedness of the finite union of circles extend this
inclusion to every $x$. More precisely, the selected pattern at every
$x$ realizes the same polynomial, so each root center has its fixed
positive number of branch occurrences at every cross-section.

\emph{Step 6: from local centers to complete global circles.}
Compactness of the image gives a finite cover by such charts. Their
finite sets of centers give finitely many circles globally; merge equal
circles while retaining their branch counts. Distinct circles meet in
at most two points, so their mutual intersection set $P_*$ is finite.
Every fiber is finite by Lemma~\ref{ct:lem:fibers}, and consequently
$u^{-1}(P_*)$ is finite. On a complementary material interval the curve
stays on one circle. Its continuous unit tangent has one fixed orientation
along that circle, so it is an analytic unit-speed circle arc.

A circle that occurs in this construction is covered completely. To see
this, its intersection with the compact image is closed. Suppose an
accumulating covered arc approached a boundary point of its covered
subset. Away from circle intersections the local graph stays on that
circle, so the boundary point must lie in $P_*$. In its chart, the
approaching arc forces this same center to occur among the fixed roots:
a different circle cannot contain infinitely many distinct points of
that arc. Shrink the chart so that this circle has a unique local graph through
the point. Its fixed positive root count supplies that graph on both
sides of the central cross-section, regardless of which material label
realizes it. Thus the covered arc extends through the point.
Thus there is no boundary point. The covered subset is nonempty, open,
and closed in the connected circle, so it is the whole circle.

At an ordinary point of a circle, the number of its local branches is
the number of its material visits and is a positive integer. The same
fixed root count carries this integer across every intersection chart.
It is therefore a constant $m_j$ along the entire distinct circle $j$.
Splitting the finitely many material arcs and adding their lengths,
or equivalently changing variables on their injective circle pieces,
gives
\[
 1=\LL(u)=\sum_jm_j\frac{2\pi}{c}.
\]
Continuity of $T$ forces equal tangent lines whenever material arcs on
distinct circles join. Such an intersection of distinct equal-radius
circles is a tangency.

\emph{Step 7: compute the entire distributional reaction.}
All actual changes of circle occur among the finite set $u^{-1}(P_*)$;
let $S$ be their set. On each intervening arc $k'=-c^2T$. The one-sided
traces of this ordinary derivative agree at every cut since $T$ is
continuous. Thus $[k']=0$, although $[k]$ may be nonzero. Integration by
parts on the finitely many pieces gives
\[
 D^2k=-c^2k\,\dd s+\sum_{s\in S}[k]_s\delta_s'.
\]
There is no $\delta_s$ contribution because its coefficient would be
$[k']_s$. Equation~\eqref{sa:eq:variation} proves the reaction formula.
\end{proof}

\begin{proposition}[Converse balance for finite circle concatenations]
\label{sa:prop:circle-balance}
Let $u:\T\to\R^2$ be a closed, unit-speed $C^1$ concatenation of
finitely many circle arcs of radius $1/c$, $c>0$. Let $S$ be its set
of actual switches between distinct circles and put $k=u''$.
Then $u$ is ambient stationary if and only if
\[
 \sum_{\substack{s\in S\\u(s)=p}}[k]_s\otimes T(s)=0
 \qquad\text{for every }p\in u(S).
\]
\end{proposition}
\begin{proof}
The concatenation is $H^2$, has $|k|=c$ almost everywhere, and satisfies
$k'=-c^2T$ on every piece. Continuity of $T$ gives equal ordinary
$k'$ traces at every switch. The reaction is therefore
$2\sum_{s\in S}[k]_s\delta_s'$, and
\[
 D\E(u)[X\circ u]
 =-2\sum_{p\in u(S)}DX(p):
       \sum_{\substack{s\in S\\u(s)=p}}[k]_s\otimes T(s).
\]
Smooth compactly supported fields can prescribe their derivative
matrices independently at these finitely many spatial points.
This proves both implications.
\end{proof}

For example, traverse two externally tangent circles once each, with
opposite planar orientations so that the material tangent agrees at
their tangency. The two visits have equal tangents and opposite
curvature jumps, so the balance holds: this is the tangent figure-eight,
with $c=4\pi$ at unit length. A three-circle example uses centers
$(-2R,0),(0,0),(2R,0)$, with opposite orientations on neighboring
circles. Traverse the left circle, one middle semicircle, the right
circle, then the other middle semicircle. At each tangency the two
switch tensors cancel. Taking $R=1/(6\pi)$ gives a unit-length
stationary curve. Neither example alone proves selection by
ropelength minimizers of a prescribed knot class.

Planarity and the bound of three visits are explicit hypotheses. They
have not been derived from bridge saturation for a general spatial
elastic knot. The theorem also permits a finite dipole reaction, which
is outside the finite-measure hypothesis of
Theorem~\ref{sa:thm:measure} whenever a genuine switch has nonzero
curvature jump.

\section{A contact model and a conditional energy expansion}
\label{sa:sec:selection}
The preceding conclusions determine the limiting bending energy and
some of its geometry. Selecting one geometry among several with the same
energy requires a more precise estimate. We next prove estimates for a
specified local functional. Its relation to the full knot problem is an
additional mathematical question. Inner braid and outer loop asymptotics
for open elastic knots under tension were developed in
\cite{AudolyEtAl2007,ClauvelinEtAl2009}. In particular,
\cite[Sections 7.4--7.7.1, equation (85)]{ClauvelinEtAl2009} uses
mean/difference strands, a disk obstacle, winding, and endpoint slope
terms. The whole-line class below has its own affine-tail freedoms;
its variational estimates do not establish matching to the outer knot.

\subsection{How a switched tangency produces the functional}
Fix a curvature $c>0$ and write $R=1/c$. In coordinates $(y,z_1,z_2)$
centered at an external tangency of two coplanar circles of radius $R$,
their graphs are
\[
 z_\pm(y)=\pm\bigl(R-\sqrt{R^2-y^2}\bigr)e_1
         =\pm\tfrac12cy^2e_1+O(y^4),
 \qquad e_1=(1,0)\in\R^2.
\]
The parameter $y$ is a coordinate along the tangent line, not arclength
on the varied graphs. Keeping the material labels gives relative graph
$cy^2e_1+O(y^4)$. Exchanging the labels as $y$ passes through zero gives
$cy|y|e_1+O(y^4)$ instead. We call the latter continuation switched.
Its two curvature jumps are $2ce_1$ and $-2ce_1$.

For a small separation scale $\delta>0$, put
\[
 y=\sqrt\delta X,\qquad g(y)=\delta H(X).
\]
Here $g:\R\to\R^2$ is a relative graph and $H$ its rescaled profile.
The distance constraint in this model is $|H(X)|\ge2$. For motions
$h_1=g/2$ and $h_2=-g/2$, the dipole convention of
Section~\ref{ct:sec:balance} gives the contribution
\[
 -2\bigl(2ce_1\cdot h_1'(0)-2ce_1\cdot h_2'(0)\bigr)
       =-4cg_1'(0)=-4c\sqrt\delta H_1'(0).
\]
Primes on $g$ denote $y$ derivatives, and primes on $H$ denote $X$
derivatives. Since $g''(y)=H''(X)$ and
$|g''/2|^2+|-g''/2|^2=|g''|^2/2$, the relative acceleration cost also
has factor $\sqrt\delta$. This calculation motivates the following
functional; it does not estimate the omitted outer strands or identify
the model distance constraint with geometric thickness.

For an odd integer $q$, define $\mathcal A_q(c)$ as the set of maps
$H\in H^2_{\mathrm{loc}}(\R;\R^2)$ such that
\[
 |H|\ge2,\qquad
 H''-2c\operatorname{sign}(X)e_1\in L^2(\R;\R^2),
\]
whose continuous angle has increment $q\pi$ between the two ends of the
line. Values at individual points use the $C^1$ representative. Set
\begin{equation}\label{sa:eq:cell}
 J_c(H)=\frac12\int_\R
       |H''-2c\operatorname{sign}(X)e_1|^2\dd X-4cH_1'(0),
 \qquad \sigma_q(c)=\inf_{H\in\mathcal A_q(c)}J_c(H).
\end{equation}
To justify the angle condition, integrate the $L^2$ defect twice.
Cauchy--Schwarz applied to the kernel $X-t$ gives a remainder bounded
by $\|H''-2c\operatorname{sign}(\cdot)e_1\|_2|X|^{3/2}/\sqrt3$.
Consequently
\[
 H_1(X)=cX|X|+O(|X|^{3/2}+|X|+1),\qquad
 H_2(X)=O(|X|^{3/2}+|X|+1).
\]
Thus $H/|H|$ tends to $-e_1$ on the left and $e_1$ on the right.
Because $H$ never vanishes, choose an angle at zero and define
\[
 \varphi(X)=\varphi(0)+\int_0^X
     \frac{H_1H_2'-H_2H_1'}{|H|^2}\dd t.
\]
Differentiating $H/|H|$ shows that it and
$(\cos\varphi,\sin\varphi)$ satisfy the same linear rotation equation
with the same initial value, so they agree. Each tail lies in a fixed
open halfplane. There its angle is a continuous principal angle plus
one fixed integer multiple of $2\pi$, and therefore has a finite limit.
The difference of the limits is an odd multiple of $\pi$. In particular,
the definition allows arbitrary affine terms on both tails.

This freedom matters for the interpretation of $J_c$. Put
$F=H''-2c\operatorname{sign}(X)e_1$ and
$J_c^{A,B}(H)=\tfrac12\int_{-A}^B|F|^2\dd X-4cH_1'(0)$.
For every $A,B>0$, expansion of the square and integration separately
on the two half intervals give the exact identity
\begin{equation}\label{sa:eq:cell-boundary}
 \frac12\int_{-A}^B(|H''|^2-4c^2)\dd X
 =J_c^{A,B}(H)
 +2c\{H_1'(B)+H_1'(-A)-2c(A+B)\}.
\end{equation}
The admissibility conditions alone need not give finite asymptotic
slope offsets or an integrable raw excess. If additionally
$F_1\in L^1(\R)$, the limits
$\alpha_\pm=\lim_{X\to\pm\infty}(H_1'(X)-2c|X|)$ exist, and
\eqref{sa:eq:cell-boundary} tends to
\[
 \frac12\int_\R(|H''|^2-4c^2)\dd X
       =J_c(H)+2c(\alpha_++\alpha_-).
\]
Thus $J_c$ equals that raw excess only when the boundary sum vanishes.
For instance, $H=(cX|X|+\alpha X,2)$ has zero raw excess but
$J_c(H)=-4c\alpha$. Minimizing over the full class gives a lower bound
for the infimum of $J_c$ over any specified smaller clamped class, but no
comparison with a physical knot constant follows without a geometric
matching argument.

\begin{proposition}[Coercivity supplied by three crossings]
\label{sa:prop:cell}
If $q$ is odd and $|q|\ge3$, then $\sigma_q(c)$ is finite, positive,
and attained. After translations of the independent variable, every
minimizing sequence has a subsequence converging strongly in
$H^2_{\mathrm{loc}}(\R;\R^2)$ to a minimizer. Moreover,
\[
 \sigma_q(c)=c^{3/2}\sigma_q(1)=\sigma_{-q}(c),\qquad
 \sigma_q(c)\ge\frac{16}{9}\,864^{1/4}c^{3/2}.
\]
\end{proposition}
\begin{proof}
We first check invariance and coercivity, then prove compactness.
Reflecting the second coordinate preserves the functional and changes
the sign of $q$, so assume $q\ge3$.

\emph{Translation invariance.}
Let $H_d(X)=H(X+d)$ and initially take $d>0$. Changing variable
$Y=X+d$ moves the reference jump from zero to $d$. Only $0<Y<d$
contributes to the difference of the squared defects. On this interval
\[
 |H''+2ce_1|^2-|H''-2ce_1|^2=8cH_1''.
\]
Integration therefore gives
\[
 \int|H_d''-2c\operatorname{sign}(X)e_1|^2
 -\int|H''-2c\operatorname{sign}(X)e_1|^2
       =8c\bigl(H_1'(d)-H_1'(0)\bigr).
\]
The difference is finite since $H''$ is locally integrable. Oriented
integration proves the same formula for $d<0$. The linear term in
$J_c$ cancels half of this difference, hence $J_c(H_d)=J_c(H)$.
Translation also preserves the obstacle and angle increment.

Choose the initial tail angle to be $\pi$. Since the terminal angle
is at least $4\pi$, successive intermediate levels
$3\pi/2,5\pi/2,7\pi/2$ occur in that order at three parameter values.
They can be chosen in order by taking first hitting times after the
previous level. No monotonicity of the angle is required. Translate
the middle value to zero and denote the other two by $-t_-$ and $t_+$,
where $t_-,t_+>0$. Then
\[
 H_1(-t_-)=H_1(0)=H_1(t_+)=0,\qquad
 z:=H_2(0)\ge2,\quad H_2(-t_-),H_2(t_+)\le-2.
\]
Write $a=H_1'(0)$, $b=H_2'(0)$,
$T_*=t_-+t_+$, $S_*=t_-^{-1}+t_+^{-1}$, and
\[
 E_1=\int_\R|H_1''-2c\operatorname{sign}(X)|^2\dd X,
 \qquad E_2=\int_\R|H_2''|^2\dd X.
\]
These are nonnegative real numbers with
$J_c=(E_1+E_2)/2-4ca$.

\emph{Coercive inequality.}
Twice integrating the first coordinate on $[0,t_+]$ yields
\[
 0=at_++ct_+^2+
       \int_0^{t_+}(t_+-X)(H_1''-2c)\dd X.
\]
The kernel has squared integral $t_+^3/3$, so the defect integral
on this interval is at least $3(a+ct_+)^2/t_+$. Applying this same
calculation to $-H_1(-X)$ on $[0,t_-]$ and adding gives
\[
 E_1\ge3S_*a^2+12ca+3c^2T_*.
\]
For the second coordinate, let $K$ be the continuous function affine
on each of $[-t_-,0]$ and $[0,t_+]$, with values zero at the outer
endpoints and one at zero. Integrating on the two intervals separately,
\[
 \int_{-t_-}^{t_+}KH_2''\dd X
 =\frac{H_2(t_+)-z}{t_+}
  -\frac{z-H_2(-t_-)}{t_-}\le-(z+2)S_*.
\]
All integrations by parts are valid for the absolutely continuous
function $H_2'$. Since $\int K^2=T_*/3$ and $S_*T_*\ge4$,
\[
 E_2\ge\frac{3(z+2)^2S_*^2}{T_*}\ge\frac{768}{T_*^3}.
\]
The full energy, including the nonnegative defect outside the three
crossings, consequently satisfies
\begin{align*}
 J_c(H)&\ge\frac32S_*a^2+2ca+\frac32c^2T_*+\frac12E_2\\
 &=\frac32S_*\left(a+\frac{2c}{3S_*}\right)^2
     +\frac32c^2T_*-\frac{2c^2}{3S_*}+\frac12E_2\\
 &\ge\frac32S_*\left(a+\frac{2c}{3S_*}\right)^2
     +\frac43c^2T_*+\frac12E_2.                    \tag{*}
\end{align*}
In the last step $1/S_*\le T_*/4$ was used. It follows that
$J_c(H)\ge(4/3)c^2T_*+384/T_*^3$. Its right-hand side has its
minimum at $T_*^4=864/c^2$, where its value is the asserted lower bound.

There are admissible finite-energy profiles. To construct one, put
$L_0=\max\{1,\sqrt{3/c}\}$. Choose a smooth step $\eta$, zero to the
left of $-L_0$ and one to the right of $L_0$, and a smooth cutoff
$\chi$, zero for $|X|\le L_0$ and one for $|X|\ge2L_0$, taking values
in $[0,1]$. Define
\[
 r(X)=3+\chi(X)(cX^2-3),\quad
 \psi(X)=\pi+q\pi\eta(X),\quad
 H(X)=r(X)(\cos\psi(X),\sin\psi(X)).
\]
In the transition of $\chi$, $cX^2\ge3$, hence $r\ge3$. On the two
outer tails $H$ equals $-cX^2e_1$ and $cX^2e_1$, respectively.
The defect is bounded and compactly supported, hence belongs to $L^2$;
the angle increment is $q\pi$. Although $H$ is smooth, its defect
has jump $-4ce_1$ at zero because of the reference sign function.
This proves finiteness and supplies a common energy level $M>0$
for a minimizing sequence.

\emph{Uniform bounds after translation.}
All constants in this paragraph are chosen from the fixed numbers $M,c$.
For $J_c(H)\le M$, the inequalities already proved give
\[
 T_{\min}:=(384/M)^{1/3}\le T_*\le T_{\max}:=\frac{3M}{4c^2},
 \qquad E_2\le2M,
\]
\[
 \frac4{T_{\max}}\le S_*\le
 S_{\max}:=\sqrt{\frac{MT_{\max}}{24}},\qquad
 1/S_{\max}\le t_\pm\le T_{\max}.
\]
Indeed $E_2\ge48S_*^2/T_*$ yields the upper bound for $S_*$.
The completed square in (*) and $S_*\ge4/T_{\max}$ then give
\[
 |a|\le A:=\sqrt{MT_{\max}/6}+cT_{\max}/6.
\]
The triangular-kernel estimate gives
$(z+2)^2\le MT_{\max}^3/24$, so
$z\le Z:=\sqrt{MT_{\max}^3/24}$.
Twice integrating $H_2''$ toward each endpoint gives, respectively,
\[
 b\le-\frac{z+2}{t_+}+\sqrt{E_2t_+/3},\qquad
 b\ge\frac{z+2}{t_-}-\sqrt{E_2t_-/3}.
\]
Thus $|b|\le B:=\sqrt{2MT_{\max}/3}$. Finally the exact identity
$E_1+E_2=2J_c+8ca$ gives $E_1+E_2\le E:=2M+8cA$.
If these constants give incompatible inequalities, the sublevel is empty;
our sequence lies in a nonempty sublevel.

On every bounded interval the defect bound controls $H''$ in $L^2$;
the bounds for $H(0)=(0,z)$ and $H'(0)=(a,b)$ then control $H$ and
$H'$ by integration. Local Sobolev compactness and a diagonal extraction
produce weak $H^2_{\mathrm{loc}}$ and strong $C^1_{\mathrm{loc}}$
convergence. The limit obeys $|H|\ge2$.

It remains to retain the full phase. Twice integrating from zero gives
\[
 |H_1(X)-cX|X||\le A|X|+\sqrt{E/3}|X|^{3/2},
 \quad |H_2(X)|\le Z+B|X|+\sqrt{E/3}|X|^{3/2}.
\]
For the single choice
\[
 R_0=\max\{1,4A/c,16E/(3c^2)\},
\]
the first coordinate has the sign of $X$ and magnitude at least
$cX^2/2$ whenever $|X|\ge R_0$. These inequalities hold for every
sequence member and its local limit, which we now denote by $H_j$ and
$H$. The displayed bounds make their tail directions converge uniformly
to $-e_1$ on the left and $e_1$ on the right. The tails remain in
the corresponding open halfplanes, so completing their directions
from $\pm R_0$ to these limits adds no full turn.

For large $j$, $\|H_j-H\|_{C^0([-R_0,R_0])}<1$. Set
$H_t=H+t(H_j-H)$ on this interval, $0\le t\le1$; then $|H_t|\ge1$.
Its left endpoint stays in the left halfplane and its right endpoint
in the right halfplane. Choose the left argument $\lambda(t)$
continuously in $(\pi/2,3\pi/2)$. The angle formula gives the lift
\[
 \phi_t(X)=\lambda(t)+\int_{-R_0}^{X}
 \frac{H_{t,1}H_{t,2}'-H_{t,2}H_{t,1}'}{|H_t|^2}\,\dd Y.
\]
It depends continuously on $t$. Let $\beta(t)\in(-\pi/2,\pi/2)$
be the right endpoint's argument. Then
$\phi_t(R_0)-\beta(t)=2\pi n(t)$ for an integer $n(t)$.
Continuity makes $n(t)$ constant. With the halfplane completions,
the full increment is $2\pi n(t)-\pi$; it is therefore the same
for $H_j$ and $H$, namely $q\pi$. The uniform tail control prevents
turns from escaping to infinity. This homotopy is used in the
punctured plane only: it need not satisfy $|H_t|\ge2$. The limit's
obstacle $|H|\ge2$ was already obtained separately by local uniform
convergence.
Weak lower semicontinuity on $[-n,n]$, followed by monotone convergence
as $n\to\infty$, bounds its full defect by the liminf of the sequence
defects. Strong local $C^1$ convergence also gives convergence of
$H_1'(0)$. Therefore the limit attains $\sigma_q(c)$.

For a minimizing sequence this convergence improves to strong
$H^2_{\mathrm{loc}}$ convergence. For the translated profiles put
$F_j=H_j''-2c\operatorname{sign}(X)e_1$ and
$F=H''-2c\operatorname{sign}(X)e_1$. Their uniform global $L^2$ bound
and local weak convergence give $F_j\rightharpoonup F$ in $L^2(\R)$:
approximate an arbitrary $L^2$ test by compactly supported tests and
use the uniform bound on the remaining tails. Since $H$ attains the
minimum and $H_{j,1}'(0)\to H_1'(0)$,
\[
 \|F_j\|_2^2=2J_c(H_j)+8cH_{j,1}'(0)
 \longrightarrow2J_c(H)+8cH_1'(0)=\|F\|_2^2.
\]
Weak convergence and convergence of norms imply $F_j\to F$ strongly
in $L^2(\R)$. Thus $H_j''\to H''$ in $L^2$ on every bounded interval;
together with the established local $C^1$ convergence this proves
the asserted strong $H^2_{\mathrm{loc}}$ compactness.

For scaling, if $G\in\mathcal A_q(1)$, let $H(X)=G(\sqrt c X)$.
The inverse transformation shows that this is a bijection onto
$\mathcal A_q(c)$. The substitution $Y=\sqrt c X$ gives
\[
 \frac12\int|H''-2c\operatorname{sign}(X)e_1|^2\dd X
 =\frac{c^{3/2}}2\int|G''-2\operatorname{sign}(Y)e_1|^2\dd Y,
 \qquad -4cH_1'(0)=-4c^{3/2}G_1'(0).
\]
Taking infima in this bijection proves the scaling identity. The
reflection considered at the outset proves equality for opposite phases.
\end{proof}

\subsection{Half-turn and retained classes}
The half-turn classes behave differently. For $A>0$ set
\[
 h(t)=\begin{cases}(t+t^3)/2,&|t|\le1,\\t|t|,&|t|\ge1,\end{cases}
 \qquad H_A(X)=(cA^2h(X/A),2).
\]
The values and first derivatives of the two formulas for $h$ agree at
$\pm1$. Hence $H_A\in H^2_{\mathrm{loc}}$, and its modulus is at least
two. Its angle travels through the upper halfplane from $-e_1$ to $e_1$,
with increment $-\pi$. The only nonzero defect is on $[-A,A]$;
using $h''(t)=3t$ there,
\[
 \int|H_A''-2c\operatorname{sign}(X)e_1|^2\dd X
 =2c^2A\int_0^1(3t-2)^2\dd t=2c^2A.
\]
As $H_{A,1}'(0)=cA/2$, we get $J_c(H_A)=-c^2A$. Letting $A\to\infty$
proves $\sigma_{-1}(c)=-\infty$, and reflection proves the same for $q=1$.

For a retained tangency the reference is $2ce_1$ on both tails and
there is no switched dipole term. Its functional
$\frac12\int|H''-2ce_1|^2$ vanishes at $H=(cX^2+2,0)$, whose phase
is zero. It also vanishes at
\[
 H=(cX^2-2,2\sqrt c X),\qquad |H|^2=c^2X^4+4\ge4.
\]
The latter passes through the negative first axis at zero and has one
full turn; its reflection has the opposite turn. A prescribed nonzero
winding alone therefore does not guarantee a positive cost.

\subsection{What a hard-thickness profile would imply}
To formulate a separate scalar implication, fix a nonempty family
$\mathcal F$ of unit-length embedded curves. Write $\Delta(\gamma)$ for
its geometric thickness, and suppose the energies under consideration
are finite. Define extended real infima
\[
 e=\inf_{\gamma\in\mathcal F}\B(\gamma),\quad
 \mathfrak h(\delta)=\inf_{\substack{\gamma\in\mathcal F\\
                          \Delta(\gamma)\ge\delta}}\B(\gamma),\quad
 M(\theta)=\inf_{\gamma\in\mathcal F}
                   \bigl\{\B(\gamma)+\theta/\Delta(\gamma)\bigr\}.
\]
The infimum of an empty set is $+\infty$ and $\theta/0=+\infty$.

\begin{proposition}[Optimization under an assumed profile expansion]
\label{sa:prop:profile-optimization}
Assume $e\in\R$ and
\[
 \mathfrak h(\delta)=e+S\sqrt\delta+o(\sqrt\delta)
       \quad(\delta\downarrow0),\qquad S>0.
\]
Then
\[
 M(\theta)=e+3\,2^{-2/3}S^{2/3}\theta^{1/3}
                  +o(\theta^{1/3}).
\]
If exact minimizers $\gamma_\theta$ exist for sufficiently small
$\theta>0$, their thickness satisfies
$\Delta(\gamma_\theta)\sim(2/S)^{2/3}\theta^{2/3}$.
\end{proposition}
\begin{proof}
First, the two minimizations satisfy the exact identity
\[
 M(\theta)=\inf_{\delta>0}
                 \bigl\{\mathfrak h(\delta)+\theta/\delta\bigr\}.
\]
For any finite-energy curve use $\delta=\Delta(\gamma)$ to obtain
one inequality. For the other, fix $\delta$ with finite
$\mathfrak h(\delta)$ and take curves with bending tending to that
infimum. They have thickness at least $\delta$, so their regularized
energies have limsup at most $\mathfrak h(\delta)+\theta/\delta$.
Taking the infimum over $\delta$ proves equality.

Choose $\delta_0=2^{-n_0}$, where $n_0\ge0$ is the least integer for
which
\[
 \tfrac12S\sqrt\delta\le\mathfrak h(\delta)-e
       \le\tfrac32S\sqrt\delta\quad(0<\delta<\delta_0).
\]
The assumed little-oh remainder ensures existence. Since increasing
the thickness threshold shrinks the admissible set, $\mathfrak h$ is
nondecreasing. In particular
\[
 \mathfrak h(\delta_0)-e\ge\mathfrak h(\delta_0/2)-e
                  \ge(S/2)\sqrt{\delta_0/2}>0.
\]
Fix
\[
 C=3S/2+1,\qquad
 g_0=\min\{1,\inf_{\delta\ge\delta_0}(\mathfrak h(\delta)-e)\}>0.
\]
Testing $\delta=\theta^{2/3}$ shows
$M(\theta)-e\le C\theta^{1/3}$ when
$\theta<\delta_0^{3/2}$. For
\[
 0<\theta<\min\{\delta_0^{3/2},[g_0/(2(C+1))]^3\},
\]
any scalar competitor within $o(\theta^{1/3})$ of the infimum has
$\delta<\delta_0$, because its energy excess is eventually below $g_0$.
The two nonnegative terms in its excess then imply, with
$t=\delta/\theta^{2/3}$,
\[
 (C+1)^{-1}\le t\le[2(C+1)/S]^2.
\]
For an exact curve minimizer these inequalities hold with $C$ in place
of $C+1$: its bending is at least $\mathfrak h(\Delta)$ and its
penalty is $\theta/\Delta$.

On this fixed compact range of $t$ the rescaled remainder is uniformly
small. Indeed if $R(\delta)=\mathfrak h(\delta)-e-S\sqrt\delta$, then
for $0<t_-\le t\le t_+<\infty$,
\[
 \sup_t\frac{|R(\theta^{2/3}t)|}{\theta^{1/3}}
 \le\sqrt{t_+}\sup_{0<\delta\le\theta^{2/3}t_+}
                   \frac{|R(\delta)|}{\sqrt\delta}\longrightarrow0.
\]
The remaining function is $f(t)=S\sqrt t+t^{-1}$, because
$S\sqrt\delta+\theta/\delta=\theta^{1/3}f(t)$.
Its derivative has the sign of $(S/2)t^{3/2}-1$, so its unique minimum
is at
\[
 t_*=(2/S)^{2/3},\qquad f(t_*)=3\,2^{-2/3}S^{2/3}.
\]
Testing this fixed $t_*$ gives the upper asymptotic bound. Confined
approximate scalar minimizers, the uniform remainder, and
$f(t)\ge f(t_*)$ give the lower bound.

If $d_\theta=\Delta(\gamma_\theta)$ for an exact minimizer, then
\[
 M(\theta)\ge\mathfrak h(d_\theta)+\theta/d_\theta
              \ge\inf_{\delta>0}
                   \{\mathfrak h(\delta)+\theta/\delta\}=M(\theta).
\]
There is equality throughout. The uniform expansion and the energy
asymptotic imply $f(d_\theta/\theta^{2/3})\to f(t_*)$.
Compactness of this scalar sequence and uniqueness of the minimum
force $d_\theta/\theta^{2/3}\to t_*$, as claimed.
\end{proof}

\begin{remark}[What is conditional]
The profile expansion in the last proposition is a hypothesis about the
full admissible family. Strong compactness gives no rate of approach to
a limiting geometry. The local functional \eqref{sa:eq:cell} has been
analyzed with its stated obstacle, phase and affine tail freedoms;
identifying its constants with a geometric hard-thickness profile requires
both a lower bound and a matching recovery construction, including the
energy of the outer strands. None of those identifications is used in
the earlier compactness or rigidity proofs.
\end{remark}

\appendix
\renewcommand{\sectionname}{}

\section{Recovering embeddings from the full regular \texorpdfstring{$C^1$}{C1} closure}\label{rc:section}
Recall that $\CC(K)$ consists of anchored, unit-speed $H^2$ embeddings
of the knot class $K$. In this appendix set
\[
 D_K=H^2\cap\overline{\CC(K)}^{C^1},\qquad
 S_K=\overline{\CC(K)}^{H^2},
\]
and let $W_K^{\rm seq}$ consist of weak $H^2$ limits of sequences in
$\CC(K)$. All three definitions concern the parametrized curves on
$\T$. In particular every element of $D_K$ has speed one.
\begin{theorem}\label{rc:theorem}
For every $u\in D_K$ there are $w_j\in\CC(K)$ with
$w_j\to u$ strongly in $H^2$. Therefore
\[
 D_K=S_K=W_K^{\rm seq},\qquad
 \inf_{D_K}\B=\inf_{S_K}\B=e_K.
\]
\end{theorem}
The initial embedded approximants are assumed to converge only in $C^1$.
Their second derivatives can be unbounded. We must construct new
approximants, rather than extract a bounded subsequence from the old ones.
For curves with finitely many isolated self-intersections, Wacker
constructs same-class approximants with convergent total curvature
\cite[Theorem 3.1]{Wacker2021}. The theorem here gives strong $H^2$
recovery without that restriction. Its proof uses elementary Sobolev
estimates, configurations of points in a plane, and an explicit ambient
isotopy. It does not use
Theorem~\ref{thm:no-loss}.

Here is the local language. For a positive integer $N$ define
\[
 \operatorname{Conf}_N(\R^2)
 =\{(z_1,\ldots,z_N)\in(\R^2)^N:z_i\ne z_l\ (i\ne l)\}.
\]
A path in this open set describes the disjoint graph strands
$x\mapsto(x,z_i(x))$. The labels are fixed. Their order need not agree
with the orientation of the original closed curve. A \emph{relative
class} means a homotopy class with both endpoint configurations fixed.
The homotopy parameter and the axial coordinate $x$ are different
variables. We write $p*q$ for concatenation and $\bar p$ for reversal.
Thus $p*\bar p$ contracts while its endpoints stay fixed, by shortening
its outward and returning journeys simultaneously. These elementary
path cancellations will keep track of endpoint information.

An endpoint collar of $I=[a,b]$ means an interval $[a,a+\delta]$ or
$[b-\delta,b]$, with $\delta>0$. A compact core is a compact set strictly
inside the open region under consideration. The notation
$H^2_{\rm loc}(O)$ means strong $H^2$ convergence on every closed
subinterval compactly contained in the open set $O$. It also gives
convergence on finite unions of such intervals. These definitions are
used separately for each component of a configuration path.

\subsection{Three analytic tools and their domains of use}
First suppose $g\in H^2(I;\R^d)$ and $g(x_0)=g'(x_0)=0$. The fundamental
theorem for $H^1$ functions, applied twice, gives
\[
 g'(x)=\int_{x_0}^x g''(t)\dd t,\qquad
 g(x)=\int_{x_0}^x(x-t)g''(t)\dd t.
\]
For $|x-x_0|\le r$, Cauchy--Schwarz consequently gives
\[
 |g'(x)|\le r^{1/2}\|g''\|_{L^2(I_r)},\qquad
 |g(x)|\le r^{3/2}\|g''\|_{L^2(I_r)}/\sqrt3,
 \quad I_r=(x_0-r,x_0+r).
\]
Let $\chi_r(x)=\chi((x-x_0)/r)$, where the fixed smooth $\chi$ equals
one on $[-1/2,1/2]$ and vanishes outside $(-1,1)$. For $0<r\le1$,
\begin{equation}\label{rc:pointcutoff}
 \|\chi_rg\|_{H^2(I)}\le C\|g''\|_{L^2(I_r)}\longrightarrow0.
\end{equation}
Indeed, the product's second derivative is
$\chi_rg''+2\chi_r'g'+\chi_r''g$. The last two terms have $L^2$ norms
bounded by $Cr^{-1}(2r)^{1/2}r^{1/2}\|g''\|_2$ and
$Cr^{-2}(2r)^{1/2}r^{3/2}\|g''\|_2$; the first has the same bound
without these scaling factors. The zero and first derivative terms
have smaller powers of $r$. Finally the integral of $|g''|^2$ over
$I_r$ tends to zero. Subtracting an affine first jet before applying
this estimate allows us to replace a graph by that jet near $x_0$.

Second, we need joins that match values and derivatives. Fix smooth
$\eta:[0,1]\to[0,1]$ equal to one near $0$ and zero on $[1/4,1]$.
For $t=(x-a)/r$ set
\[
 P(x)=b_0\eta(t)+rc_0t\eta(t)
       +b_1\eta(1-t)+rc_1(t-1)\eta(1-t).
\]
It has values $b_0,b_1$ and derivatives $c_0,c_1$ at $a,a+r$,
and is affine with the specified first jet near either endpoint.
Writing $b=\max|b_i|$ and $c=\max|c_i|$, differentiation and the
substitution $x=a+rt$ give
\begin{equation}\label{rc:hermite}
 \|P''\|_2\le C(br^{-3/2}+cr^{-1/2}),\qquad
 \|P\|_\infty\le C(b+rc),\qquad
 \|P'\|_\infty\le C(b/r+c).
\end{equation}
For example a position term contributes $r^{-2}$ upon differentiation
twice and $r^{1/2}$ to its $L^2$ norm. A derivative term contains the
additional factor $r$. Matching first jets is sufficient for a
piecewise $H^2$ function to be $H^2$: integration by parts on the
pieces has no boundary jump in either the function or its first
weak derivative. A correction with nonzero derivative data is affine,
not constant, in its endpoint collar.

Third, strong convergence transfers between graph coordinates on
strictly smaller domains. We record the argument because it does not
apply to a sequence known only to converge in $C^1$. Suppose scalar
$H^2$ functions $q_j\to q$ strongly, all increasing, satisfy
$c\le q_j'\le C$. On compact subsets of their common image their
inverses $r_j=q_j^{-1}$ obey
\[
 r_j'=\frac1{q_j'\circ r_j},\qquad
 r_j''=-\frac{q_j''\circ r_j}{(q_j'\circ r_j)^3}.
\]
They converge in $C^1$. To justify convergence of the last numerator,
choose nested intervals $K\Subset V\Subset W\Subset O$ in the old
coordinate, with strong convergence on $O$. For a new compact interval
$K'$ with $r(K')\subset K$, uniform convergence gives $r_j(K')\subset V$
for all sufficiently large $j$. For $g_j\to g$ in $L^2(W)$,
\[
 \|(g_j-g)\circ r_j\|_{L^2(K')}\le C^{1/2}\|g_j-g\|_{L^2(W)}.
\]
This follows by changing variables using $q_j'\le C$. Approximate
$g$ in $L^2(W)$ by a continuous function $g_\varepsilon$. The two
composition errors involving $g-g_\varepsilon$ are bounded by
$2C^{1/2}\|g-g_\varepsilon\|_2$; the remaining error tends to zero
by uniform continuity of $g_\varepsilon$ on a compact neighborhood of
$\overline V$. First let $j\to\infty$, then $\varepsilon\to0$.
This proves $g_j\circ r_j\to g\circ r$ in $L^2(K')$ and therefore
strong $H^2$ convergence of the inverses. Decreasing maps reduce to
this case by reversing the coordinate.

For a rotation of graph axes write
$\xi_j(x)=Ax+Bf_j(x)$ and
$\widetilde f_j=(Cx+Df_j)\circ\xi_j^{-1}$, where $A,B,C,D$ are
fixed linear maps of the indicated dimensions. On an overlap with
$|\xi_j'|\ge c$, the inverse formulas and the chain rule prove strong
$H^2$ convergence of $\widetilde f_j$ on $K'$. All coefficients involving
first derivatives converge uniformly; every second derivative term is
an $L^2$ composition just treated. The nested domains ensure that no
estimate evaluates an old graph beyond its controlled open set.

\subsection{A flat contact can store any relative class}
\begin{lemma}\label{rc:flat}
Let $b\in H^2(I;(\R^2)^N)$ be collision-free in endpoint collars.
Suppose that at one interior point $x_0$ all $b_i(x_0)$ agree and
all $b_i'(x_0)$ agree. For any relative class with the same endpoint
configurations and any $\varepsilon>0$ there is a collision-free
$H^2$ path in that class, equal to $b$ in smaller endpoint collars,
whose $H^2$ distance from $b$ is less than $\varepsilon$.
Those smaller collars may be prescribed in advance, provided they are
compactly contained, relative to $I$, in the given collision-free collars.
If $b$ lies with positive margin in an open disk $D_*$, the path may
be kept in that disk. If a representative of the desired class also
lies there, the relative homotopy may be kept there too.
\end{lemma}
\begin{proof}
We first work in the whole plane. Allocate four errors of size less
than $\varepsilon/4$. Apply \eqref{rc:pointcutoff} to $b_i-b_1$.
Fix the prescribed smaller collars. Choose a radius whose doubled
interval misses the given collision-free endpoint collars and for which the sum of the flattening errors is less than
$\varepsilon/4$. Replace the relative graphs by
$(1-\chi_r)(b_i-b_1)$. The resulting path $\widehat b$ agrees with
$b$ outside that interval, and all its components equal a common
$H^2$ graph $a(x)$ on its middle interval.

Choose a smooth scalar $0\le q\le1$ which is zero on neighborhoods
of the prescribed smaller collars,
equals one on the middle interval, and is positive at every collision
of $\widehat b$. Such a function exists because its collision set is
a compact set disjoint from the smaller collars. Add $q(x)D_i$ to
component $i$, with constant vectors $D_i\in\R^2$. We may choose
$D=(D_1,\ldots,D_N)$ arbitrarily small so that this path has no collisions.
Here is the measure argument that proves this assertion. For a pair
$i\ne l$, a collision where $q>0$ would require
\[
 D_i-D_l=-\frac{\widehat b_i(x)-\widehat b_l(x)}{q(x)}.
\]
On each compact subinterval of $\{q>0\}$ the right side is $C^1$ and
has planar Lebesgue measure zero as an image: subdividing an interval
into lengths at most $h$ covers its image by $O(h^{-1})$ disks of
radius $Mh$, where $M$ bounds the derivative, and their total area
is $O(h)$. A countable compact exhaustion gives the same conclusion
on $\{q>0\}$. The inverse image of this null set under the surjective
linear map $D\mapsto D_i-D_l$ is null by linear coordinates and
Tonelli's theorem. The finite union over pairs is null in $\R^{2N}$.
Outside it the path is collision-free, including where $q=0$, where
no collision of $\widehat b$ occurs. The $D_i$ are distinct because
the relative positions on the middle interval are precisely $D_i-D_l$.

Fix a subinterval of length $r_1>0$ inside that middle interval.
There are constants $C_q,C_{r_1}$, fixed before selecting $D$, such that
\[
 \|qD\|_{H^2}\le C_q|D|,
 \qquad C_{r_1}=C(r_1^{1/2}+r_1^{-1/2}+r_1^{-3/2})
\]
controls rescaling of a fixed smooth profile onto that subinterval.
Choose a generic $D$ satisfying
$0<|D|<\varepsilon/[4(1+C_q+C_{r_1})]$. This is our second error
budget and reserves room for ramps of size $D$.

Subtract the common translation $b_1(x)$ from every component of every
path. This identifies relative classes bijectively, since common
translation cannot create collisions. On the chosen interval the
constructed path is now the constant configuration $D$. Let $p$ be
a representative of the desired translated class. Split the constructed
path at this constant portion as $p_-*p_+$, with $p_-$ ending and $p_+$
starting at $D$. The loop that must be inserted there has class
\[
 \lambda=\bar p_-*p*\bar p_+.
\]
Indeed $p_-*\lambda*p_+$ contracts to $p$ by the two path cancellations.
This formula specifies the basepoint; a free loop class would not suffice.

Put $d=|D|$ and $A=D/d$. Let $s$ be the positive scaling path from
$D$ to $A$. At $A$ choose a smooth based loop $\beta$, stationary near
its endpoints, in the class $\bar s*\lambda*s$. A continuous loop
in the open configuration space has a positive minimum separation.
Polygonal approximation, then smoothing, with total uniform error less
than a quarter of that separation gives the asserted representative;
linear interpolation realizes its relative homotopy. Reparametrizing
the loop by a smooth increasing function constant near the endpoints
makes it stationary there.

For any $0<\alpha\le d$, insert in the reserved interval the three paths
\[
 D\longrightarrow\alpha A,\qquad \alpha\beta,
 \qquad \alpha A\longrightarrow D,
\]
with positive scaling on the outer thirds. Use fixed smooth time changes
stationary near their joining points. These paths stay collision-free.
Their combined based class is $\lambda$: the homotopy that rescales
$\beta$ by positive factors transports its basepoint along the same
scaling paths, and the resulting forward and backward scaling segments
cancel. Adding the common graph back preserves that identity.
The $H^2$ change in the reserved interval is bounded by
\[
 C_{r_1}|D|+C_{r_1}\alpha\bigl(1+\|\beta\|_{H^2([0,1])}\bigr).
\]
The first term includes the ramps and removal of the constant $D$;
the second follows by differentiating the rescaled loop twice.
Choose, only now,
\[
 \alpha=\frac12\min\left\{d,
 \frac{\varepsilon}{4C_{r_1}(1+\|\beta\|_{H^2})}\right\}.
\]
All four error budgets are respected. This order is essential: the loop
may be arbitrarily complicated, but its amplitude is chosen after its
finite $H^2$ norm is known. The joins are $H^2$ since their first jets
match; the outer collars were never changed.

For the disk assertion let $m>0$ be the minimum distance of $b$ from
$\partial D_*$. Run the construction with error smaller than both
$\varepsilon$ and $m/(4C_I)$, where
$\|g\|_{C^1(I)}\le C_I\|g\|_{H^2(I)}$. The new path lies in $D_*$
with margin at least $3m/4$. If both paths to be compared lie in the
disk, their compact images lie in a smaller concentric disk. There is
an injective radial map $\Psi:\R^2\to D_*$ which equals the identity
on that smaller disk: choose its radial function strictly increasing,
equal to the radius on the smaller disk, and tending to the radius
of $D_*$ at infinity. Apply $\Psi$ to every component of any relative
homotopy. Injectivity preserves distinctness, while the identity region
fixes both paths and their endpoints. This confines the homotopy.
\end{proof}

\subsection{Local replacement, including previously controlled strands}
\begin{lemma}\label{rc:local}
Let $I=[a,b]$, let $f\in H^2(I;(\R^2)^N)$, and suppose
collision-free $h_j\in H^2(I;(\R^2)^N)$ converge to $f$ in $C^1$.
Prescribe the following data:
\begin{enumerate}
\item open controlled sets $O_i\subset(a,b)$, one for each label,
on which $h_{j,i}\to f_i$ strongly in $H^2_{\rm loc}(O_i)$;
\item finitely many compact cores $J_1,\ldots,J_m\Subset(a,b)$,
and closed outer endpoint collars $C_-,C_+$ disjoint from their union;
\item a sequence of positive tolerances $\tau_n$;
\item optionally, an open transverse disk $D_*$ containing all $f_i(I)$
with positive distance from its boundary.
\end{enumerate}
There are indices $j_n\uparrow\infty$ and collision-free replacements
$k_n\in H^2(I;(\R^2)^N)$ such that:
\begin{enumerate}
\item $k_n=h_{j_n}$ on fixed neighborhoods of $C_-\cup C_+$,
and the two paths have the same homotopy class relative endpoints;
\item $\|k_n-f\|_{C^1(I)}\le\min\{\tau_n,1/n\}$, and there is an
open neighborhood $U\Subset(a,b)\setminus(C_-\cup C_+)$ of all the
cores on which $k_n\to f$ strongly in $H^2$;
\item $k_{n,i}\to f_i$ strongly in $H^2_{\rm loc}(O_i)$ for every label;
\item when $D_*$ is prescribed, the replacements and the relative
homotopies lie in $\operatorname{Conf}_N(D_*)$.
\end{enumerate}
\end{lemma}
\begin{proof}
We prove the controlled-set, collar, tolerance and disk assertions
simultaneously by induction on $N$. Write $J=\bigcup_\nu J_\nu$,
and relabel extracted sequences by $j$ during the construction.
Use $\widehat\tau_n=\min\{\tau_n,1/n\}$ in every $C^1$ budget below.
Choose an active interval whose
interior contains that union and whose closure avoids the prescribed
outer collars. Every change below is supported inside it. Estimates
on any controlled set outside this interval are retained by leaving
the functions unchanged.

For $N=1$, take a smooth $0\le\chi\le1$ supported inside the active
interval and equal to one near $J$, and use
$b_j=(1-\chi)h_j+\chi f$. On $K\Subset O_1$ the product rule gives
\[
 \|b_j-f\|_{H^2(K)}
 \le C_\chi\|h_j-f\|_{H^2(K)}\longrightarrow0.
\]
On a neighborhood of $J$ the difference is zero. The $C^1$ error tends
to zero globally. The straight homotopy and the path lie in $D_*$
when that disk is prescribed, because it is convex. There are no
collision restrictions with one label. Extract source indices so that
the resulting global $C^1$ error is at most $\widehat\tau_n$.

Now suppose the assertion holds for every smaller multiplicity.
If at some point of the active interior all values and first derivatives
of $f$ agree, choose $\chi=1$ also near that point. The same blend $b_j$
has a full flat contact and is collision-free near the outer endpoints.
Apply Lemma~\ref{rc:flat} on the closed active interval, in the
relative class of $h_j$ restricted there, with $H^2$ error tending
to zero; extend the result by $h_j$ outside that interval. The blend
equals $h_j$ near its endpoints, so first jets match and fixed
neighborhoods of the prescribed outer collars remain unchanged.
The $H^2$ error preserves every local controlled estimate and strong
convergence near $J$. Sobolev
embedding retains $C^1$ convergence. For large $j$, convexity places
$b_j$ with a fixed positive margin in a prescribed disk, so the disk
clause of Lemma~\ref{rc:flat} applies. The additional errors can be
chosen after extracting source indices so that the total $C^1$ error
is at most $\widehat\tau_n$. This completes this case.

We may therefore assume that there is no full flat contact in the
active interval. Put
\[
 Z=\{x:f_1(x)=\cdots=f_N(x)\}.
\]
The set $Z$ is closed there and has no accumulation point in its
interior. Indeed, if $x_n\in Z$ tend to $x$, then every
$g_i=f_i-f_1$ satisfies $g_i(x_n)=g_i(x)=0$, and the difference quotient
gives $g_i'(x)=0$. This would be a full flat contact. Choose nested
closed intervals $I_0\Subset\operatorname{int}I_1$ inside the active
interval, with $J\subset\operatorname{int}I_0$ and endpoints outside
$Z$. There are only finitely many nodes in $Z\cap I_1$.

We first improve the sequence away from these nodes. At a point
$x\notin Z$, partition the labels according to equality of $f_i(x)$.
Each group has fewer than $N$ labels. Distinct group values have a
positive minimum distance $d$. Choose pairwise disjoint disks about
those values with radius less than $d/4$; in the prescribed disk case
also choose the radii less than one quarter of the boundary distances.
On a sufficiently small axial interval $U$ the limiting graphs stay
in the middle halves of these disks. Uniform convergence puts the
source graphs there too for all large indices. Apply the induction
hypothesis to each group on $U$, with smaller compact cores. All group
homotopies stay in their respective disks and hence cannot collide
with other groups. Running these homotopies at the same time gives
a relative homotopy for the full configuration. Their endpoint collars
allow the changed paths to extend unchanged outside $U$.

The preceding local operation must preserve earlier improvements.
We give the selection procedure explicitly. Choose increasing compact
exhaustions $E_p$ of $\operatorname{int}I_1\setminus Z$ and $K_{i,p}$
of $O_i$. They may be taken as finite unions of closed intervals,
with slightly larger compact neighborhoods still in the respective
open sets. For each $E_p$ take a finite cover by the intervals $U$
just constructed, together with nested open cores
$V\Subset W\Subset U$, with the $V$'s covering $E_p$.
In that operation ask for strong convergence on $\overline W$.
Enumerate the finite covers, and process their group operations one
at a time. At each step augment the controlled open sets, label by
label, by all previously obtained open strong-convergence regions.
The induction hypothesis preserves convergence on their compact
subsets. Every step acts on a sequence, possibly extracting a further
infinite subsequence of original indices; it never asserts small
curvature for an arbitrary individual source curve.

After a finite number $m$ of operations there is a sequence satisfying
those $m$ requirements. Choose an increasing $m(n)$ which has processed
the covers of $E_1,\ldots,E_n$. From that stage's nested subsequence
select an original index larger than the preceding selected index,
with $C^1$ error below $1/n$ and with $H^2$ errors on
$E_1,\ldots,E_n$ and $K_{i,1},\ldots,K_{i,n}$ below $1/n$.
These are finitely many demands. Each compact set has a finite cover
by controlled open intervals, so strong convergence of that stage's
sequence supplies such an index. Splitting the bound among the finite
cover estimates, if necessary, makes the stated total error available.
The resulting diagonal sequence converges strongly locally off $Z$
and preserves every $O_i$. Each selected member involved only finitely
many modifications, so it keeps its source's class and fixed outer
collars. All disk conditions hold exactly at every operation. A smaller
prescribed tolerance can replace the $C^1$ bound $1/n$. Rename this
preliminary sequence $h_j$.

It remains to improve it near the finitely many nodes meeting $I_0$.
Take disjoint closed node intervals inside $\operatorname{int}I_1$;
their endpoints are outside $Z$. We work on one, translating the node
and its common value to $(0,0)$. Write $a_i=f_i'(0)$ and group labels
by equality of $a_i$. There are at least two groups. Select a nonzero
linear functional $\ell:\R^2\to\R$ with different values on distinct
$a_i$; only finitely many forbidden lines in the space of functionals
must be avoided. Define
\[
 \sigma=\min_{a_i\ne a_l}|\ell(a_i-a_l)|>0,\qquad
 \delta=\sigma/2,\qquad
 \omega(r)=\max_i\sup_{|x|\le r}|f_i'(x)-a_i|.
\]
Fix $r_*>0$ so small that $[-2r_*,2r_*]$ is inside the node interval,
contains no other full node, and
$2\|\ell\|\omega(2r_*)\le\sigma/4$. On this interval every projected
intergroup derivative keeps the sign of $\ell(a_i-a_l)$ and has
magnitude at least $3\sigma/4$. Discard indices until
$\max_i\|h_{j,i}'-f_i'\|_\infty\le\sigma/(8\|\ell\|)$.
The corresponding source derivatives have magnitude at least $\delta$
with the same signs.

From now on the actual node replacement interval is
$N_*=[-r_*,r_*]$. Keep the preliminary paths unchanged outside $N_*$,
and fix the outer collars
$[-r_*,-3r_*/4]$ and $[3r_*/4,r_*]$. All axial changes below are the
identity on these collars. This restriction loses no controlled region:
the preliminary sequence already converges strongly locally off the
full-node set $Z$, including at partial contacts. The finitely many
smaller node intervals still cover all full nodes meeting $I_0$;
the compact complement of their interiors in $I_0$ is covered by
finitely many preliminary controlled regions.

For a finite configuration $z$ let $\operatorname{rad}(z)$ be the
smallest radius of a disk containing its points. It is continuous:
moving each point by at most $e$ changes this radius by at most $e$,
using the old enclosing disk enlarged by $e$, and then reversing the
comparison. On $[-r_*,r_*]$ choose $x_j$ minimizing
$\operatorname{rad}(h_j(x))$, denote the minimum by $R_j>0$, and choose
an enclosing center $y_j$. Positivity follows because $N\ge2$ and the
source path is collision-free. Evaluation at $0$ gives $R_j\to0$.
Any limit point of $x_j$ must be a full coincidence of $f$ in the
closed interval, and thus equals $0$. Consequently $x_j\to0$ and
$y_j\to0$.

The rescaled source graphs
\[
 H_{j,i}(z)=\frac{h_{j,i}(x_j+R_jz)-y_j}{R_j}
\]
are defined on every fixed compact $z$-interval for large $j$.
Their values at $0$ lie in the closed unit disk. Passing to a
subsequence gives $H_{j,i}(0)\to b_i$, $|b_i|\le1$. Their derivatives
are $h_{j,i}'(x_j+R_jz)$ and converge uniformly on compact intervals
to $a_i$. Integration therefore proves
\[
 H_j\longrightarrow A,\qquad A_i(z)=a_i z+b_i,
 \quad\hbox{in }C^1_{\rm loc}.
\]
For every fixed $z$, minimality of $R_j$ gives
$\operatorname{rad}(H_j(z))\ge1$ for large $j$. Continuity of the radius
implies $\operatorname{rad}(A(z))\ge1$. In particular $A$ never has a
full coincidence. This normalization uses no estimate on $h_j''$.

Fix a constant $C_0\ge1$ bounding the position basis in
\eqref{rc:hermite} and the derivative basis by $C_0$ times distance
from its inner endpoint. Choose once and for all
\[
 L>4+\frac{16\|\ell\|(C_0+1)}{\delta},\qquad M=L+2.
\]
This precedes the next replacements. For intergroup pairs,
$|\ell(H_{j,i}(0)-H_{j,l}(0))|\le2\|\ell\|$, whereas the projected
derivative has fixed sign and magnitude at least $\delta$.
Thus their projected zeros lie in $(-L,L)$; the two exterior intervals
$[L,M]$ and $[-M,-L]$ have strictly ordered groups with uniform margins.

Apply the proper-group construction above to $A$ on $[-M,M]$.
There are no full nodes, so a finite cover suffices to obtain strong
convergence on an open neighborhood of $[-L,L]$. At every operation
retain earlier open cores as controlled sets. We obtain collision-free
$Q_j$ in the class of $H_j$, equal to $H_j$ in fixed collars of
$\pm M$, such that
\[
 Q_j\to A\text{ in }C^1([-M,M]),\qquad
 d_j:=\|Q_j-A\|_{H^2([-L,L])}\longrightarrow0.
\]
The exterior portions of $Q_j$ still have their ordered groups for
large $j$, by the uniform margins just proved. Only induction at
strictly smaller multiplicities was used to construct $Q_j$.

We next identify exactly the topology that can remain in those tails.
For an ordered list of groups $G_1,\ldots,G_m$, the \emph{ordered
chamber} is the set of configurations for which every $\ell$-projection
in $G_g$ is smaller than every one in $G_{g+1}$. Write
$z_i=c_g+u_i^g$ for $i\in G_g$, where $c_g$ is that group's centroid
and $\sum_{i\in G_g}u_i^g=0$. Internal distinctness concerns only the
$u_i^g$. For fixed internal configurations the centroid restrictions are
\[
 \ell(c_{g+1}-c_g)>
 \max_{i\in G_g}\ell(u_i^g)-\min_{l\in G_{g+1}}\ell(u_l^{g+1}).
\]
They define a convex open set. Fix $e\in\R^2$ with $\ell(e)=1$.
A continuous choice of allowed centroids is $c_1=0$ and
\[
 c_{g+1}-c_g=
 \left(1+\max_{i\in G_g}\ell(u_i^g)
             -\min_{l\in G_{g+1}}\ell(u_l^{g+1})\right)e.
\]
Linear interpolation of the centroids to this choice leaves all
inequalities strict and all internal vectors unchanged. It is a
deformation of the chamber onto this section over the product of
centered internal configuration spaces.

For relative paths the endpoint motion of this deformation matters.
Denote the centroid deformation by $D_t$, $0\le t\le1$, and
parametrize a path $p$ from $U$ to $V$ on $[0,1]$. Define the
partial endpoint tracks
\[
 r_{U,t}(s)=D_{ts}(U),\qquad r_{V,t}(s)=D_{ts}(V),
 \qquad 0\le s\le1.
\]
For every $t$, the concatenation
\[
 r_{U,t}*(D_t\circ p)*\bar r_{V,t}
\]
runs from $U$ to $V$: its three pieces meet at $D_t(U)$ and $D_t(V)$.
At $t=0$ it is $p$ with constant endpoint pieces, removable by an
endpoint-fixing reparametrization. Varying $t$ is therefore an explicit
homotopy relative endpoints. At $t=1$, write $r_U=r_{U,1}$,
$r_V=r_{V,1}$ and $p_{\rm sec}=D_1\circ p$ to obtain
\[
 p\simeq r_U*p_{\rm sec}*\bar r_V.
\]
When comparing two paths with endpoints $U,V$, use the same $r_U,r_V$.
The common tracks cancel from the comparison, leaving exactly the
tuple of their internal relative classes. Centering a group's configurations removes a translation
factor $\R^2$, whose paths are homotopic relative endpoints by straight
interpolation. Thus the missing tail classes can be specified as
ordinary within-group configuration classes. Singleton groups have
only the trivial class. This proves the relative assertion, including
the endpoint transports, without discarding any group shapes.

Choose a new spatial scale
\[
 \varepsilon_j=\sqrt{d_j}+j^{-1}.
\]
It tends to zero and, since $\varepsilon_j\ge\sqrt{d_j}$,
\[
 \frac{d_j}{\sqrt{\varepsilon_j}}\le d_j^{3/4}\longrightarrow0
\]
when $d_j>0$; the numerator is zero when $d_j=0$. The original
$R_j$ located the topology and affine limit, whereas $\varepsilon_j$
is chosen after $Q_j$ to control the rescaled analytical error.
We can move the central path from its source location and scale to
this scale without losing its relative class. First replace $H_j$ by $Q_j$ on $[-M,M]$
using the established relative homotopy. Regard the portions between
$L$ and $M$, and between $-M$ and $-L$, as parts of the tails.
Move the retained central configuration by
\[
 (t,z)\longmapsto(1-t)y_j+
             ((1-t)R_j+t\varepsilon_j)Q_j(z),\qquad |z|\le L.
\]
The scaling factor stays positive. At $z=\pm L$ the configurations
stay in their respective ordered chambers, since translation and
positive scaling preserve projected order. Move the axial interval
$[x_j-LR_j,x_j+LR_j]$ to $[-L\varepsilon_j,L\varepsilon_j]$
by increasing changes of axial variable on $N_*$ which equal the
identity on its fixed outer collars. Both central intervals lie in
$(-r_*/2,r_*/2)$ for all sufficiently large $j$. A tail follows its
moving central endpoint by concatenating the endpoint's transport
path or its reversal, according to orientation, and reparametrizing
on the new tail interval. This explicitly gives the full path
homotopy with fixed outer endpoints. The discarded $Q_j$ portions
were in the chambers, so they contribute to those transported tails.

The analytical central path is now
\[
 q_{j,i}(x)=\varepsilon_jQ_{j,i}(x/\varepsilon_j),
                  \qquad |x|\le L\varepsilon_j.
\]
To see the scaling explicitly, put
$G_{j,i}(z)=Q_{j,i}(z)-a_i z-b_i$ on $[-L,L]$.
The error relative to $a_i x+\varepsilon_jb_i$ is
$\varepsilon_jG_{j,i}(x/\varepsilon_j)$. Two differentiations
and the substitution $z=x/\varepsilon_j$ give
\[
 \int_{-L\varepsilon_j}^{L\varepsilon_j}
 \left|\bigl(\varepsilon_jG_{j,i}(x/\varepsilon_j)\bigr)''\right|^2\dd x
 =\varepsilon_j^{-1}\int_{-L}^{L}|G_{j,i}''(z)|^2\dd z.
\]
Consequently
\[
 \|(q_{j,i}-a_i x-\varepsilon_jb_i)''\|_2
 \le\varepsilon_j^{-1/2}d_j\longrightarrow0.
\]
The corresponding first derivative norm is at most
$\varepsilon_j^{1/2}d_j$ and the function norm at most
$\varepsilon_j^{3/2}d_j$. On this shrinking interval the constant
$\varepsilon_jb_i$ has $L^2$ norm at most $\sqrt{2L}\varepsilon_j^{3/2}$.
Thus the central $H^2$ error relative to $a_i x$ tends to zero.
The fixed-interval Sobolev estimate for $Q_j-A$ also gives endpoint
position errors at most $\varepsilon_j(1+C_Ld_j)$ and derivative
errors at most $C_Ld_j$. Here $C_L$ depends on the fixed $L$, not on $j$.

For the tails first fix a small $\rho>0$ with $4\rho<r_*$.
The replacement tails are exactly $[-r_*,-L\varepsilon_j]$ and
$[L\varepsilon_j,r_*]$. All tail comparisons and relative homotopies
below concern these intervals. Define
\[
 f_i^\rho(x)=f_i(x)-\chi(x/\rho)(f_i(x)-a_i x),
\]
where this cutoff equals one for $|x|\le\rho$ and zero for
$|x|\ge2\rho$. Estimate \eqref{rc:pointcutoff}, with the harmless
change of cutoff radius, proves $f^\rho\to f$ in $H^2$ as $\rho\downarrow0$.
Moreover
$|(f_i^\rho)'(x)-a_i|\le C_{\rm cut}\omega(2\rho)$ on $|x|\le2\rho$.
Choose $\rho$ so that
$C_{\rm cut}\|\ell\|\omega(2\rho)\le\delta/8$.
The signed projected intergroup derivative is then at least $\delta$;
since its position difference vanishes at $0$, its projected gap on
either side is at least $\delta|x|$. On $2\rho\le|x|\le r_*$
the same lower bound follows from the projected derivative estimate
on $[-2r_*,2r_*]$.

Assume $L\varepsilon_j<\rho/10$ and $C_Ld_j\le1$.
On $[L\varepsilon_j,\rho/2]$ add to $a_i x$ the clamped correction
from \eqref{rc:hermite} with the actual central first jet at the inner
endpoint and zero correction jet at the outer one. Reflect this
construction on the negative side. The interval lengths are between
$2\rho/5$ and $\rho/2$. Its global $H^2$ correction is bounded by
\[
 C_L\bigl(\varepsilon_j\rho^{-3/2}+d_j\rho^{-1/2}\bigr),
\]
where the lower derivatives are absorbed by a fixed constant since
$\rho\le1$. The correction vanishes near its outer endpoint. The
intervals $[\rho/2,\rho]$ and $[-\rho,-\rho/2]$ consequently provide
a plateau where every member of a derivative group is its same affine
line.

We must check intergroup separation before repairing the internal
collisions. At a tail point with $|x|\ge L\varepsilon_j$ the size of
one correction is at most
\[
 2C_0\varepsilon_j+C_0C_Ld_j|x|.
\]
For the derivative basis this uses its vanishing value at the inner
endpoint and $|x-L\varepsilon_j|\le|x|$ on the positive side;
the negative side is identical after reflection. A projected difference
of two corrections therefore has size at most
$4\|\ell\|C_0\varepsilon_j+2\|\ell\|C_0C_Ld_j|x|$.
Our fixed choice of $L$ makes the first term smaller than
$\delta|x|/4$. Discard indices until
$2\|\ell\|C_0C_Ld_j<\delta/4$. The reference gap
$\delta|x|$ then leaves at least $\delta|x|/2$ throughout each joining
interval. On the plateau and the rest of $f^\rho$ it is already positive.

Farther out choose a fixed smooth cutoff, for this $\rho$, equal to
one up to $2\rho$ and zero for $|x|\ge3\rho$, and blend $f^\rho$
to the preliminary $h_j$. On the compact transition annuli strong
convergence to $f$ was obtained in the preliminary diagonal. The
product rule makes the blending error tend to zero in $H^2$ for fixed
$\rho$. The reference and $h_j$ both lie in the same ordered chambers
there for large $j$: each projected gap of $f$ has a positive lower
bound on the closed annulus, and $h_j\to f$ uniformly. Convex
blending of those projected inequalities preserves their strictness.
Outside the transition keep $h_j$, including its outer node collars.
We have constructed provisional tails with no intergroup collisions.

Within a group these tails may have collisions. Nevertheless each
has collision-free endpoint collars. At the outer endpoint it equals
$h_j$ on the fixed collar $[3r_*/4,r_*]$, or its reflected counterpart,
because $3\rho<3r_*/4$. At the inner endpoint its value is the actual,
distinct $Q_j$ configuration; continuity gives a collision-free collar
for this fixed $j$. Only this inner collar may depend on $j$.
Lemma~\ref{rc:flat} keeps the fixed outer collar unchanged. Matching its first derivative with that of the
central path gives an $H^2$ join. On the affine plateau the whole group
has a flat contact. The chamber calculation therefore tells us which
relative internal class is missing from each provisional tail, when
it is compared with the transported original tail having the same
endpoints. Prescribe those classes and apply Lemma~\ref{rc:flat} to
each group separately. For a singleton the class is already unique,
and no correction is needed.

For fixed $j,\rho$, let $s_{j,\rho}>0$ be the smallest projected gap
between adjacent groups on the closed tails
$[-r_*,-L\varepsilon_j]$ and $[L\varepsilon_j,r_*]$.
These minima are positive because both entire tails now lie in their
respective ordered chambers. Let
$C_{j,\rho}$ be an $H^2$-to-$C^0$ embedding constant for their intervals.
Choose the sum of all flat-contact correction errors smaller than
\[
 \min\left\{j^{-1},\frac{s_{j,\rho}}
                            {4\|\ell\|C_{j,\rho}}\right\}.
\]
Every projected gap changes by less than $s_{j,\rho}/2$. The corrected
tails therefore remain in their chambers, become internally
collision-free, and have the required full relative classes. Combined
with the retained central path, they recover the original relative
class. Their endpoint collars remain those of the provisional tails.
This completes the construction for fixed $\rho$ and sufficiently
large $j$.

For completeness, we now choose the two limits in their required order.
Take $\rho_n\downarrow0$ so that each radius satisfies the preceding
derivative condition and
$\|f^{\rho_n}-f\|_{H^2}<1/n$. Require also that the sum of flattening
$C^1$ errors over all nodes be less than $\widehat\tau_n/3$. Choose the radii for the
finitely many nodes first. For those fixed radii choose a source index
$j_n>j_{n-1}$ for which all of the following finite demands hold:
\[
 L\varepsilon_{j_n}<\rho_n/10,\quad C_Ld_{j_n}\le1,\quad
 2\|\ell\|C_0C_Ld_{j_n}<\delta/4,
\]
\[
 C_L\left(\varepsilon_{j_n}^{3/2}
       +\frac{d_{j_n}}{\sqrt{\varepsilon_{j_n}}}
       +\frac{\varepsilon_{j_n}}{\rho_n^{3/2}}
       +\frac{d_{j_n}}{\rho_n^{1/2}}\right)<\frac1n.
\]
Include the errors on the entire fixed stage-$n$ exterior annuli
$\{2\rho_n\le |x|\le r_*\}$, including their transition regions, and the first
$n$ compact controlled-set errors in the demands, with their finite
product constants absorbed by subdividing $1/n$. Include the global
$C^1$ error of the preliminary source $h_{j_n}$, and require its sum
with the remaining central and blending $C^1$ errors to be below
$\widehat\tau_n/3$.
Every demand tends to zero with the radii fixed, so the least eligible
index in the current subsequence exists. Make the final flat-contact
errors smaller than $1/n$, their projected-gap bounds, and
$\widehat\tau_n/(3C_n)$, where $C_n$ bounds the finitely many relevant Sobolev
embedding constants. This last choice is made after the actual gaps
and constants are known. For several nodes impose all these finite
conditions simultaneously, dividing the budgets by their number.

The resulting paths converge strongly to $f$ on every full node
neighborhood: the reference error, central error, joining error,
annulus error, and flat-contact error have each been made small.
Away from the nodes the preliminary estimates persist. A compact
subset of an original $O_i$ is covered by these node neighborhoods
and finitely many compact exterior intervals. Including their
exhaustion errors in the diagonal proves the required componentwise
strong convergence. The same argument gives strong convergence on a
neighborhood of $J$. All modifications have fixed outer collars,
and the displayed bounds prove global $C^1$ convergence with the
prescribed tolerances.

Finally, if a transverse disk was prescribed, let $m>0$ be the
reference boundary margin. Add the finite requirements that the old
and new paths be within $m/4$ of $f$ in $C^0$; the convergence just
proved permits them. Both paths now lie in a smaller concentric disk.
The radial compression argument in Lemma~\ref{rc:flat} confines their
already proved relative homotopy to $D_*$. This retains the disk
invariant even though the intermediate central rescaling was described
in the full plane. Every induction invariant has now been verified,
and the induction is complete.
\end{proof}

\subsection{From graph homotopies to ambient isotopies}
We verify the topological step before using the local construction on
a closed curve. Suppose two $C^1$ configuration paths in an open disk
$D_*$ agree in endpoint collars and have a continuous relative homotopy
$F^0(x,t)$ contained in that disk. Compactness of $I\times[0,1]$ gives
a positive minimum pairwise separation $s$ and a positive minimum
disk-boundary distance $m$ for that homotopy. With one label omit $s$.
Choose $e>0$ smaller than $s/32$ and $m/16$.

First make the homotopy fixed on smaller endpoint collars. Its values
near an endpoint are uniformly close to the fixed endpoint configuration,
by uniform continuity in $x,t$. The common endpoint path is also close
to that configuration. Choose an axial collar width such that both
errors are below $e$, and interpolate with a smooth cutoff between
the common path and $F^0$ there. Every intermediate configuration is
within $e$ of the endpoint configuration; hence it stays distinct and
inside the disk. At $t=0,1$ this changes nothing, because those paths
already agree on that collar. It gives a homotopy fixed on smaller
collars. Recompute its positive separation and boundary margin.

Subdivide the time interval so finely that consecutive configurations
of this homotopy differ uniformly by less than $e$. Approximate each
interior time slice in $C^0$ by a $C^1$ path with error below $e$,
retaining the common collars; keep the first and last paths exactly.
Such relative approximation follows by ordinary convolution on the
interior and a cutoff to the fixed $C^1$ collar path. Interpolate
successive approximating slices linearly, using a smooth time function
constant near either end of each time piece. This interpolation is
within $2e$ of one original slice. Its pairwise separation is therefore
at least $s-4e$ and its boundary margin at least $m-2e$.
The resulting $F_i(x,t)$ is $C^1$ in $x$, smooth in $t$ on the pieces,
stationary near their joins, and fixed in axial endpoint collars.
Both $\partial_tF_i$ and its first $x$ derivative are continuous;
the same statements for further time derivatives hold because the
time interpolation is stationary near joins.

For this final homotopy take its actual margins $s_F,m_F>0$.
Choose
\[
 0<r<\min\{s_F/3,m_F/3\},
\]
omitting $s_F$ for one label. Let a smooth
$\zeta:[0,\infty)\to[0,1]$ equal one near $0$ and vanish on $[1,\infty)$.
Let $\xi$ be a smooth axial cutoff supported in the chart and equal
to one outside the fixed collars where any motion occurs. Define
\[
 V_t(x,y)=\xi(x)\sum_{i=1}^N
 \zeta\left(\frac{|y-F_i(x,t)|^2}{r^2}\right)
                     (0,\partial_tF_i(x,t)).
\]
The transverse supports about different strands are disjoint and lie
inside $D_*$. The field extends by zero to a compactly supported
$C^1$ field in ambient space; it is also $C^1$ in time because of the
stationary joins. On the graph of strand $i$ all other cutoffs vanish,
and its own cutoff equals one. Where it moves, $\xi=1$. Thus
\[
 V_t(x,F_i(x,t))=(0,\partial_tF_i(x,t)).
\]
The indicated graph trajectories solve the ordinary differential
equation for $V_t$.

The needed flow theorem states that a continuous time-dependent field
with a global integrable spatial Lipschitz bound has unique solutions
on the time interval; for a $C^1$ field the solution depends $C^1$ on
its initial point. This is the version of
\cite[Theorem~2.5, Corollary~2.6, and Theorem~2.10, pp.~40--41 and~46]{Teschl}
used here. Compact support and the displayed finite collection of
$C^1$ cutoffs bound the field and its spatial derivative uniformly.
They check its global existence and differentiability hypotheses.
Backward evolution gives the inverse maps. Uniqueness identifies the
flow with the prescribed graph motion on each strand. We have obtained
an ambient $C^1$ isotopy, equal to the identity outside the chart and
in its endpoint collars. A configuration homotopy has thus been
converted into the ambient isotopy needed to preserve knot type.

\subsection{Patching all branches and choosing arclength last}
\begin{proof}[Proof of Theorem~\ref{rc:theorem}]
Take $v_j\in\CC(K)$ with $v_j\to u$ in $C^1$. We first construct
regular $H^2$ embeddings $\widetilde v_j$ of the same knot type which
converge in $C^1$ and have strong convergence in every local graph
coordinate. They need not have controlled second derivatives in their
original parameter.

Every spatial fiber of a regular $C^1$ closed curve is finite, with a
uniform bound. To check this, let $\mu=\min|u'|>0$ and choose
$0<r_0<1/8$ so that
\[
 |u'(t)-u'(s)|\le\mu/2
       \quad\hbox{whenever }\dist(s,t)\le2r_0.
\]
With $e_s=u'(s)/|u'(s)|$ we have $e_s\cdot u'(t)\ge\mu/2$ on the
lifted interval of radius $2r_0$. Its axial projection is strictly
increasing, so no two distinct parameters there can have the same
image. Thus distinct members of one fiber have circular distance at
least $2r_0$. Their disjoint arcs of radius $r_0$ give a uniform finite
bound on the number of members.

At $p\in u(\T)$ choose an axis with nonzero scalar product with every
tangent belonging to its finite fiber. Choose disjoint parameter arcs
about the fiber points on which those axial derivatives retain their
signs and have magnitude at least $c_p>0$. Each branch is a graph
$(x,f_i(x))$ over a common axial interval about $p$. If the old curve
coordinates are $(x(s),y(s))$, then
\[
 f_i'=(y'/x')\circ x^{-1},\qquad
 f_i''=\left(\frac{y''}{(x')^2}-\frac{y'x''}{(x')^3}\right)
                                                   \circ x^{-1}.
\]
The denominators are bounded away from zero and the inverse change
of variables has bounded Jacobian, so these graphs are $H^2$.
The complement of the chosen parameter arcs is compact and its image
does not contain $p$; its distance $d_p$ from $p$ is positive.
Choose a sufficiently small cylinder about $p$ inside the ball of
radius $d_p/4$, with its axial closure inside all these graph domains.
Its transverse disk can be chosen so that every captured graph stays
strictly inside the disk throughout the axial interval. This follows
by continuity of finitely many branches through $p$, choosing the
axial interval after a transverse radius. No other limiting branch
enters the cylinder. A curve whose $C^0$ distance from $u$ is below
$d_p/4$ has no uncaptured parameter arc there either.

Inside each cylinder choose nested open spatial regions
\[
 C_\alpha\Subset V_\alpha\Subset W_\alpha\Subset U_\alpha,
\]
where $U_\alpha$ is the full modification cylinder, the axial projection
of $\overline{W_\alpha}$ is strictly inside its interval, and
$C_\alpha$ contains its selected point $p$. Compactness of $u(\T)$
allows finitely many $C_\alpha$ to cover the image. Keep these charts,
regions, axial derivative bounds, and disk margins fixed. All bounds
persist with half their size for sufficiently $C^1$-close curves.
The graph paths of the source embeddings are collision-free.

Process these finitely many charts in turn. In chart $U_\alpha$ apply
Lemma~\ref{rc:local} to the compact axial projection of
$\overline{W_\alpha}$. It supplies strong graph convergence on an
open neighborhood of that projection, with fixed endpoint collars
and a homotopy in the transverse disk. Apply the preceding ambient
isotopy construction. The curve remains an embedding of type $K$,
because all captured branches move by that isotopy and the remaining
curve lies outside its support.

At a later chart the controlled open sets $O_i$ include, for each
branch, its limiting portions in every previously processed
$W_\alpha$. Strong convergence on those portions transfers to the
new axial coordinate by the third analytic tool. To check its domain
condition, take a compact portion of this overlap. Its inverse image
in the old coordinate lies in some $K\Subset V\Subset W\Subset O$
where the old strong estimate holds. Uniform convergence of the
coordinate inverses puts all moving inverse images inside $V$ for
large indices. Thus the chain rules use only the old controlled set,
and the local lemma preserves the transferred estimates.

In particular, strong convergence persists on every previously
processed compact region $\overline{V_\alpha}\Subset W_\alpha$.
A compact limiting branch portion there has a finite cover by compact
coordinate intervals inside the controlled open region $W_\alpha$.
On portions changed by the later chart, coordinate transfer and
Lemma~\ref{rc:local} preserve these estimates. Outside the modification
support the graphs are unchanged, and endpoint collars identify the
descriptions near the joins. The finite cover gives convergence on
$\overline{V_\alpha}$. The same argument on each compact subset of
$W_\alpha$ preserves local strong convergence there. Only finitely many
charts and extractions occur. These estimates suffice for the energy
assembly below, whose weights are supported in the $V_\alpha$.
The
$C^1$ tolerance clause may be used to make its total $C^1$ error tend
to zero throughout the construction.

For each individual source index the patched curve remains $H^2$
in its current original parameter. Indeed, if $x_j(s)$ is its axial
coordinate and $k_j$ its replacement graph, then
\[
 (k_j\circ x_j)''=(k_j''\circ x_j)(x_j')^2
                             +(k_j'\circ x_j)x_j''.
\]
The first term is $L^2$ by the change-of-variable bound and the positive
axial derivative bound. The second is $L^2$ because $k_j'$ is bounded
and $x_j''\in L^2$. The endpoint collars give matching first jets.
This proves membership in $H^2$ for each index, without claiming a
uniform bound for $x_j''$. The first chain rule and uniform convergence
of graph functions and their first derivatives prove $C^1$ convergence
in this parameter. Regularity persists because the derivatives remain
uniformly close to the nowhere-zero $u'$.

Denote the final curves by $\widetilde v_j$. We next deduce convergence
of geometric length and bending energy from their strong graph
convergence. For one graph $(x,f(x))$, the arbitrary-parameter formula
\eqref{eq:nl-7} gives the densities, with respect to increasing $dx$,
\[
 \sqrt{1+|f'|^2},\qquad
 \frac{|f''|^2}{(1+|f'|^2)^{3/2}}
 -\frac{(f'\cdot f'')^2}{(1+|f'|^2)^{5/2}}.
\]
For the second expression use
$P_{(1,f')}^\perp(0,f'')=(0,f'')-
 ((f'\cdot f'')/(1+|f'|^2))(1,f')$ and square it.
The first derivative coefficients converge uniformly. Every product
of two second derivative components converges in $L^1$: for example
\[
 \|g_j^ag_j^b-g^ag^b\|_1
 \le\|g_j^a-g^a\|_2\|g_j^b\|_2
       +\|g^a\|_2\|g_j^b-g^b\|_2\longrightarrow0.
\]
Here $g_j=f_j''$ and the superscripts denote vector components. This
establishes $L^1$ convergence of the graph bending density.

To combine charts without double counting, choose finitely many smooth
nonnegative functions supported in the processed $V_\alpha$ and
positive on a cover of the compact image. Their sum is positive on
an open neighborhood of that image. Dividing by the sum produces
smooth weights whose sum is one there; all source images are in that
neighborhood for large $j$. In each chart the weights evaluated on
the graphs converge uniformly. Integrate the weighted densities,
sum over all its branches, and then over charts. Their supports are
compactly inside the controlled graph domains, so the just-proved
convergence applies to every summand. Branches are counted with their
parameter multiplicity even at contacts. Reversing the orientation of
a branch changes neither geometric integral. We conclude
\[
 L_j:=\LL(\widetilde v_j)\longrightarrow1,\qquad
 \B(\widetilde v_j)\longrightarrow\B(u).
\]

Only now change scale and parameter. On lifts from $[0,1]$ define
\[
 A_j(s)=L_j^{-1}\int_0^s|\widetilde v_j'(r)|\dd r,
 \qquad
 w_j(t)=L_j^{-1}\bigl(
 \widetilde v_j(A_j^{-1}(t))-\widetilde v_j(0)\bigr).
\]
The $A_j$ are increasing circle diffeomorphisms with $A_j(0)=0$ and
$A_j(1)=1$. Since the speeds converge uniformly to one,
$A_j\to\Id$ and $A_j^{-1}\to\Id$ in $C^1$.
Each inverse is $H^2$ by the formulas already proved, for its individual
positive derivative bound. The chain rule then gives $w_j\in H^2$ and
\[
 w_j'=
 \left(\frac{\widetilde v_j'}{|\widetilde v_j'|}\right)
                                  \circ A_j^{-1},\qquad |w_j'|=1.
\]
It also gives $w_j\to u$ in $C^1$, since $u(0)=0$. Translation and
positive dilation are ambient isotopies, and increasing
reparametrization does not alter the image. Hence $w_j\in\CC(K)$.
The scale and parameter invariances from \eqref{eq:nl-7} yield
\[
 \|w_j''\|_2^2=\B(w_j)
       =L_j\B(\widetilde v_j)\longrightarrow\B(u)=\|u''\|_2^2.
\]
Thus the normalized sequence is bounded in $H^2$. Any weakly convergent
subsequence must have weak limit $u$, because the $C^1$ limit is $u$.
Weak compactness then gives weak convergence of the whole sequence:
a subsequence violating a weak neighborhood of $u$ would have a
further weakly convergent subsequence with limit $u$, a contradiction.
In $L^2$, expand
\[
 \|w_j''-u''\|_2^2=\|w_j''\|_2^2+\|u''\|_2^2
                     -2\int_\T w_j''\cdot u''\dd s\longrightarrow0.
\]
The lower derivatives converge uniformly. This proves strong $H^2$
recovery.

Strong $H^2$ convergence implies $C^1$ convergence, so $S_K\subset D_K$.
Every strong sequence is weakly convergent. Conversely a weakly
convergent $H^2$ sequence is bounded, hence precompact in $C^1$;
all its $C^1$ subsequential limits equal its weak limit. It therefore
converges in $C^1$ to that limit. This proves
$S_K\subset W_K^{\rm seq}\subset D_K$, and recovery proves equality.
For $u\in S_K$, a strong approximating sequence has
$\B(u)=\lim\B(w_j)\ge e_K$. Since $\CC(K)\subset S_K$, the reverse
infimum inequality is immediate. This proves the stated infima as well.
\end{proof}

\begin{remark}[A second proof of no loss]
Theorem~\ref{rc:theorem} also implies Theorem~\ref{thm:no-loss}.
A weak $H^2$ limit $u$ of a bending-minimizing sequence belongs to
$D_K$ by compact embedding. Recovery gives $\B(u)\ge e_K$, while
weak lower semicontinuity gives $\B(u)\le e_K$. Equality of the
second-derivative norms and weak convergence give strong $H^2$
convergence; boundedness supplies weakly convergent subsequences.
The earlier ambient-variation proof is independent of this recovery
construction.
\end{remark}

\begin{corollary}[Bridge inequality in the $H^2$ closure]\label{rc:closure-bridge}
Every $u\in D_K$ satisfies
$\TC(u)\ge2\pi\operatorname{bridge}(K)$.
\end{corollary}
\begin{proof}
For the unit-speed recovery sequence $w_j\in\CC(K)$,
\[
 |\TC(w_j)-\TC(u)|\le\|w_j''-u''\|_1
                         \le\|w_j''-u''\|_2\longrightarrow0.
\]
The embedded bridge inequality therefore passes to the limit.
This recovers the $H^2$ case of \cite[Theorem 1.3]{Miura2026};
that theorem also treats regular $C^1$ curves of finite total curvature.
\end{proof}

\subsection{The additional consequence for \texorpdfstring{$C^1$}{C1} ambient traces}
\begin{corollary}\label{rc:c1ambient}
Let $u\in D_K$ satisfy $\B(u)=e_K$, and let
$X\in C_c^1(\R^3;\R^3)$ have trace $h=X\circ u\in H^2(\T;\R^3)$.
Then
\[
 \E(u+th)\ge e_K\quad\hbox{if }|t|\|DX\|_\infty<1,
 \qquad D\E(u)[h]=0.
\]
\end{corollary}
\begin{proof}
For such a fixed $t$ put $\Phi_t=\Id+tX$. The inequality
\[
 |\Phi_t(x)-\Phi_t(y)|\ge
                (1-|t|\|DX\|_\infty)|x-y|
\]
proves injectivity. For any $y$ the map $x\mapsto y-tX(x)$ is a
contraction of the complete space $\R^3$, so has a unique fixed point;
this proves surjectivity. The derivative $\Id+tDX$ is invertible, and
the inverse function theorem gives a global $C^1$ inverse. The same
bounds for $\Id+\theta tX$, $0\le\theta\le1$, give an ambient isotopy.
The target curve is regular and $H^2$, since
\[
 \Phi_t\circ u=u+th,\qquad
 |(\Phi_t\circ u)'|\ge1-|t|\|DX\|_\infty>0.
\]

We must place this target in the relevant $C^1$ closure. Take smooth
compactly supported $X_n\to X$ in $C^1$, for example by convolution.
For large $n$ the same strict derivative bound holds. Choose
$v_m\in\CC(K)$ converging to $u$ in $C^1$. For fixed $n$ the curves
$(\Id+tX_n)\circ v_m$ are regular $H^2$ embeddings of type $K$ and
converge in $C^1$ to $(\Id+tX_n)\circ u$. Their $H^2$ membership
follows from
\[
 ((\Id+tX_n)\circ v_m)''
  =tD^2X_n(v_m)[v_m',v_m']+(\Id+tDX_n(v_m))v_m''.
\]
Both coefficients are bounded on the compact image, the first term is
bounded, and the second is $L^2$. Choose an increasing $m_n$ so that
this fixed-$n$ $C^1$ approximation error is below $1/n$. The other
error is at most
$|t|(1+\|u'\|_\infty)\|X_n-X\|_{C^1}$ and tends to zero.
Thus these embedded curves converge in $C^1$ to $u+th$.

Normalize every curve by translation, length scaling, and arclength
as above. This operation is continuous in $C^1$ on regular curves:
uniform convergence of their positive speeds gives convergence of
lengths, cumulative arclength maps and their inverses, and the formula
for the normalized tangent gives convergence of first derivatives.
It preserves $H^2$ for each curve by the inverse and chain rules.
The normalized target therefore belongs to $D_K$.
Theorem~\ref{rc:theorem} gives its bending energy at least $e_K$.
That energy equals $\E(u+th)$ by normalization, proving the inequality.
The regular $H^2$ energy is differentiable along $u+th$ by the variation
calculation in Section~\ref{sec:variation}. Both positive and negative
small $t$ are permitted, and $\E(u)=e_K$, so its derivative at zero
vanishes.

The mollified traces were used only for $C^1$ approximation. No step
requires $X_n\circ u\to X\circ u$ strongly in $H^2$, which does not
follow from ordinary ambient $C^1$ mollification. The explicit trace
assumption and the full closure theorem provide the missing analytic
control at the target.
\end{proof}

\begingroup
\renewcommand{\addcontentsline}[3]{}
\section*{Use of language models}
ChatGPT assisted with drafting, mathematical
proofreading and reference checks. Mathematical validation, proof decisions
and the final text remain the responsibility of the authors.

\section*{Acknowledgement}
The second author's work is partially funded by the Excellence Initiative
of the German federal and state governments.
\endgroup

\end{document}